\nonstopmode
\documentclass[12pt]{amsart}
\usepackage[all, cmtip]{xy}
\usepackage{amsmath}
\usepackage{amscd} 
\usepackage{amssymb}
\usepackage{mathtools}
\usepackage{diagrams}

\numberwithin{equation}{section}

\theoremstyle{plain} 

\newtheorem{theorem}[equation]{Theorem}
\newtheorem{lemma}[equation]{Lemma}
\newtheorem{proposition}[equation]{Proposition}
\newtheorem{corollary}[equation]{Corollary}

\theoremstyle{definition}

\newtheorem{definition}[equation]{Definition}

\newtheorem{remark}[equation]{Remark}
\newtheorem*{remark*}{Remark}
\newtheorem*{remarks*}{Remarks}
\newtheorem{examples}[equation]{Examples}
\newtheorem{example}[equation]{Example}

\newtheorem*{VariableNoNum}{{\VariableText}}
\newtheorem{Variable}[equation]{{\VariableText}}

\theoremstyle{definition}
\newtheorem*{VariableNoNumBold}{{\VariableText}}
\newtheorem{VariableBold}[equation]{{\VariableText}}

\theoremstyle{definition}

\newlength{\asidelength}

\def\Changed/{\ifvmode\else\vadjust{%
\vbox to 0pt{\vskip -\baselineskip%
\hbox to 0pt{\hss\vrule height 0pt depth 1.2\baselineskip\hskip 1em}\vss}}\fi}
\def\CHanged{\ifvmode\else\vadjust{%
\vbox to 0pt{\vskip -\baselineskip%
\hbox to 0pt{\hss\vrule height 0pt depth 1.2\baselineskip\hskip 1em}\vss}}\fi}

\def\Math#1{\def\MathString{#1}\futurelet\MathDelim\MathChoose}
\def\MathChoose{\ifmmode\let\MathDo\MathString%
              \else\let\MathDo\MathSkip\fi%
              \MathDo}
\def\MathSkip{\ifx\MathDelim/\def\MathDo{$\MathString$\EatOne}%
              \else\def\MathDo{$\MathString$}\fi%
              \MathDo}

\def\Text#1{\def\TextString{#1}\futurelet\TextDelim\TextSkip}
\def\TextSkip{\ifx\TextDelim/\def\TextDo{\TextString\EatOne}%
              \else\let\TextDo\TextString\fi%
              \TextDo}
\def\EatOne#1{}

\def\SkipToEndScan#1\EndScan{}
\def\Scan#1#2#3{\ifx#1#2#3\expandafter\SkipToEndScan\fi\Scan#1}
\def\Upper#1{%
\Scan#1aAbBcCdDeEfFgGhHiIjJkKlLmMnNoOpPqQrRsStTuUvVwWxXyYzZ#1#1\EndScan}
\def\Phrase#1 #2/#3/#4=#5 #6/#7/#8.{%
\expandafter\edef\csname#2#3\endcsname{\noexpand\Text{#6#7}}
\expandafter\edef\csname\Upper#2#3\endcsname{\noexpand\Text{\Upper#6#7}}
\expandafter\edef\csname#1#2#3\endcsname{\noexpand\Text{#5 #6#7}}
\expandafter\edef\csname\Upper#1#2#3\endcsname{\noexpand\Text{\Upper#5 #6#7}}
\expandafter\edef\csname#2#4\endcsname{\noexpand\Text{#6#8}}
\expandafter\edef\csname\Upper#2#4\endcsname{\noexpand\Text{\Upper#6#8}}
}
\newcommand{\orb}[1]{{#1\kern-2pt}{\scriptscriptstyle\rm{-orb}}}

\newcommand{\Smash}{\wedge}

\newcommand{\GL}{\operatorname{GL}}

\newcommand{\HH}{\operatorname{H}}

\renewcommand{\star}{\operatorname{St}}
\newcommand{\link}{\operatorname{Lk}}
\newcommand{\Part}{\operatorname{Part}}
\newcommand{\rk}{\operatorname{rk}}
\newcommand{\hhh}{\operatorname{h}\!}
\newcommand{\thhh}{\tilde{\operatorname{h}}}
\def\HomotopyOrbit#1on#2/{\ensuremath{#2_{\hhh#1}}}
\def\RedHomotopyOrbit#1on#2/{\ensuremath{#2_{\thhh#1}}}

\DeclareMathOperator{\Lie}{\Lcal}
\DeclareMathOperator{\Lierep}{Lie}

\def\doCal#1{%
\ifx#1\doAllCalEnd\def\doAllCal{\relax}\else%
 \expandafter\edef\csname#1cal\endcsname{{\noexpand\mathcal #1}}\fi}
\def\doAllCal#1{\doCal#1\doAllCal}
\doAllCal ABCDEFGHIJHLMNOPQRSTUVWXYZ\doAllCalEnd

\def\doBar#1{%
\ifx#1\doAllBarEnd\def\doAllBar{\relax}\else%
 \expandafter\edef\csname#1bar\endcsname{{\noexpand\overline{#1}}}\fi}
\def\doAllBar#1{\doBar#1\doAllBar}
\doAllBar ABCDEFGHIJHLMNOPQRSTUVWXYZ\doAllBarEnd
\def\doWiggle#1{%
\ifx#1\doAllWiggleEnd\def\doAllWiggle{\relax}\else%
 \expandafter\edef\csname#1wiggle\endcsname{{\noexpand\tilde{#1}}}\fi}
\def\doAllWiggle#1{\doWiggle#1\doAllWiggle}
\doAllWiggle ABCDEFGHIJHLMNOPQRSTUVWXYZabcdefghijklmnopqrstuvwxyz\doAllWiggleEnd

\newcommand{\n}{ {\mathbf{n}} }
\newcommand{\m}{{\mathbf{m}}}

\newcommand{\pgroupsntof}[1]{\Scal_{p}\kern-1pt\left(#1\right)}
\newcommand{\pgroupsof}[1]{{\overline{\Scal}_{p}}\kern-1pt\left(#1\right)}

\newcommand{\reals}{{\mathbb{R}}}
\newcommand{\integers}{{\mathbb{Z}}}
\newcommand{\naturals}{{\mathbb{N}}}

\newcommand{\field}{{\mathbb{F}}}

\Phrase a superfluous//=a prunable//.

\begin{document}
\title{A branching rule for partition complexes}

\author{G. Z. Arone}
\address{Kerchof Hall, U. of Virginia, P.O. Box 400137,
         Charlottesville VA 22904 USA}
\curraddr{Department of Mathematics, Stockholm University, SE - 106 91 Stockholm, Sweden}         
\email{zga2m@virginia.edu}
\urladdr{http://www.math.virginia.edu/~zga2m}

\date{\today}


\begin{abstract} 
Let $\Sigma_n$ be the symmetric group, and let $\Sigma_{n_1}\times \cdots\times \Sigma_{n_k}\subset \Sigma_n$ be a Young subgroup. Let $\Pi_n$ be the complex of partitions of $\n$. Our main result is a $\Sigma_{n_1}\times \cdots\times \Sigma_{n_k}$-equivariant decomposition of $|\Pi_n|$. As an application, 
we obtain new information about the quotient space of $|\Pi_n|$ by a Young subgroup. 
\end{abstract}
\maketitle

\section{Introduction}

Let $\Pi_{n}$ denote the poset of (proper,
nontrivial) partitions of the set $\n=\{1, \ldots, n\}$. Let $|\Pi_n|$ be the geometric realization of $\Pi_n$. It is well-known that for $n\ge 3$ there is a homotopy equivalence \[|\Pi_n|\simeq \bigvee_{(n-1)!} S^{n-3}.\]

However the poset $\Pi_n$ has a natural action of the symmetric group $\Sigma_n$, which is not preserved by this equivalence. The space $|\Pi_n|$, with its natural action of $\Sigma_n$, arises in various contexts. For example, the homology group $\HH_{n-3}(|\Pi_n|)$ is closely related to the Lie representation. 

Let $\Sigma_{n_1}\times \cdots \times \Sigma_{n_k}\subset \Sigma_n$ be a  Young subgroup. In this paper we investigate the $\Sigma_{n_1}\times \cdots \times \Sigma_{n_k}$-equivariant homotopy type of $|\Pi_n|$. To state our result, we need to recall one more definition. Let $\Lie[x_1, \ldots, x_k]$ be the free Lie algebra (over $\integers$) on $k$ generators. Let $M(n_1, \ldots, n_k)\subset \Lie[x_1,\ldots, x_k]$ be the subgroup generated by commutators that contain $x_i$ exactly $n_i$ times, for all $1\le i\le k$. It is well-known that $M(n_1, \ldots, n_k)$ is a free Abelian group. The rank of $M(n_1, \ldots, n_k)$ is given by a famous formula due to Witt \cite{witt}:
\begin{equation}\label{equation: witt}
\frac{1}{n}\sum_{d|\gcd(n_1, \ldots, n_k)} \mu(d){\frac{n}{d} \choose \frac{n_1}{d},\ldots,\frac{n_k}{d}}.
\end{equation}
Here $d$ ranges over positive common divisors of $n_1, \ldots, n_k$. Later we will use a particular basis $B(n_1, \ldots, n_k)$ of $M(n_1, \ldots, n_k)$. For now, all that matters is that $B(n_1, \ldots, n_k)$ is a finite set with size given by Witt's formula~\eqref{equation: witt}.

The following is our main theorem. By a $G$-equivariant equivalence we mean a map of $G$ spaces that induces an equivalence of $H$-fixed points for every subgroup $H$ of $G$. 
\begin{theorem}\label{theorem: main}
Suppose $n=n_1+\cdots+n_k$. Let $g=\gcd(n_1, \ldots, n_k)$. There is a $\Sigma_{n_1}\times \cdots \times \Sigma_{n_k}$-equivariant homotopy equivalence
\begin{equation}\label{equation: main}
|\Pi_n| \longrightarrow \bigvee_{d|g} \left(\bigvee_{B(\frac{n_1}{d}, \ldots, \frac{n_k}{d})} \Sigma_{n_1}\times \cdots \times {\Sigma_{n_k}}_+ \wedge_{\Sigma_d} S^{n-d-1} \wedge |\Pi_d|^\diamond\right).
\end{equation}
Here $d$ ranges over positive divisors of $g$.
\end{theorem}
Here is a brief explanation of the terms in equation~\eqref{equation: main}. Throughout the paper, $n=n_1+\cdots+n_k$, and $\Sigma_{n_1}\times \cdots\times \Sigma_{n_k}\subset \Sigma_n$ is the subgroup of permutations that leave invariant the blocks 
\[\{1, \ldots, n_1\}, \{n_1+1, n_1+2, \ldots, n_1+n_2\}, \ldots,\{n_1+\cdots+n_{k-1}+1, \ldots, n\}.\]
If $d>0$ divides all of $n_1, \ldots, n_k$, then $d$ divides $n$, and 
\[
\underbrace{\Sigma_d\times \cdots\times \Sigma_d}_{\frac{n}{d}}
\] 
is a subgroup of $\Sigma_{n_1}\times\cdots \times\Sigma_{n_k}$.  We have group inclusions
\[
\Sigma_d\hookrightarrow {\Sigma_d}^{\frac{n}{d}}\hookrightarrow \Sigma_{n_1}\times \cdots\times \Sigma_{n_k}\hookrightarrow \Sigma_n.
\]
We identify $\Sigma_d$ with its image under each one of these inclusions.

If $X$ is a topological space, $X^\diamond$ denotes the unreduced suspension of $X$. We have to use the unreduced suspension of $|\Pi_d|$ in the main theorem, because $|\Pi_d|$ does not have a $\Sigma_d$-invariant basepoint. However, note that if $\Sigma_{n_1}\times\cdots\times\Sigma_{n_k}$ is a proper Young subgroup of $\Sigma_n$, then $|\Pi_n|$ does have a $\Sigma_{n_1}\times\cdots\times\Sigma_{n_k}$-invariant basepoint. The map~\eqref{equation: main} is in fact a pointed equivariant homotopy equivalence. 

By  $S^{n-d-1}$ we mean a desuspension of the sphere $S^{n-d}=(S^d)^{\Smash \frac{n-d}{d}}$, equipped with the natural action of $\Sigma_d$. Note that since $d$ divides $n$, it divides $n-d$, so this is well-defined. 
There is a $\Sigma_d$-equivariant homeomorphism $S^d\cong S^1\Smash \widehat S^{d-1}$, where on the right hand side $\Sigma_d$ acts trivially on $S^1$, and acts on $\widehat S^{d-1}$ via the reduced standard representation. It follows that there is a $\Sigma_d$-equivariant homeomorphism
$$S^{n-d-1}\cong S^{\frac{n-d}{d}-1}\Smash (\widehat S^{d-1})^{\Smash \frac{n-d}{d}}.$$

It follows that $S^{n-d-1} \wedge |\Pi_d|^\diamond$ is a space with a pointed action of $\Sigma_d$. The space 
\[
\Sigma_{n_1}\times\cdots\times {\Sigma_{n_k}}_+ \wedge_{\Sigma_d} S^{n-d-1} \wedge|\Pi_d|^\diamond
\]
is homeomorphic to a wedge sum of $\Sigma_{n_1}\times\cdots\times \Sigma_{n_k}/{\Sigma_d}$ copies of $S^{n-d-1} \wedge |\Pi_d|^\diamond$, and it has a natural action of $\Sigma_{n_1}\times\cdots\times \Sigma_{n_k}$.

We should note that our definition of $\Pi_d$ only works well for $d\ge 2$. For $d=2$, $|\Pi_2|$ is the empty set, which we identify with the $-1$-dimensional sphere, since its unreduced suspension $|\Pi_2|^\diamond$ is homeomorphic to $S^0$. We decree that $|\Pi_1|$ is the $-2$-dimensional sphere. In practice this means that the space $|\Pi_1|$ is undefined, but for $l\ge 0$, $S^{l} \wedge |\Pi_1|^\diamond := S^{l-1}$. In particular, the wedge summand corresponding to $d=1$ in equation~\eqref{equation: main} is equivalent to 
\[
\bigvee_{B(n_1, \ldots, n_k)} {\Sigma_{n_1}\times \cdots\times\Sigma_{n_k}}_+ \wedge S^{n-3}.
\]

\begin{examples}
If $\gcd(n_1, \ldots, n_k)=1$, Theorem~\ref{theorem: main} says that $|\Pi_n|$ is $\Sigma_{n_1}\times\cdots\times\Sigma_{n_k}$-equivariantly equivalent to a wedge sum of copies of $S^{n-3}$, that are freely permuted by $\Sigma_{n_1}\times\cdots\times\Sigma_{n_k}$. In the special case of the subgroup $\Sigma_{n-1}\times \Sigma_1\subset \Sigma_n$, we note that $|B(n-1, 1)|=1$ and conclude that that there is a $\Sigma_{n-1}$-equivariant homotopy equivalence 
\[|\Pi_n|\simeq_{\Sigma_{n-1}}{\Sigma_{n-1}}_+ \wedge S^{n-3}.
\] 
This equivalence is a well-known result. For example, see~\cite{donaumatching} for a closely related statement.

In the case of the subgroup $\Sigma_{n-2}\times \Sigma_2$, where $n$ is even, one easily calculates that $|B(n-2, 2)|= \frac{n-2}{2}$ and $|B(\frac{n-2}{2}, 1)|=1$. We conclude that there is a $\Sigma_{n-2}\times \Sigma_2$-equivariant equivalence
\begin{equation}\label{equation: second case}
|\Pi_n|\simeq \left(\bigvee_{\frac{n-2}{2}} \Sigma_{n-2}\times {\Sigma_2}_+ \wedge S^{n-3}\right) \vee \Sigma_{n-2}\times {\Sigma_2}_+\wedge_{\Sigma_2} S^{n-3}.
\end{equation}
Note that in the last summand $S^{n-3}$ has an action of $\Sigma_2$: it is homeomorphic to the smash product
$S^{\frac{n}{2}-2}\wedge (\hat{S}^1)^{\Smash{\frac{n}{2}-1}}$, where $S^{\frac{n}{2}-2}$ is a sphere with a trivial action of $\Sigma_2$ while $ (\hat{S}^1)^{\Smash{\frac{n}{2}-1}}$ is a sphere on which $\Sigma_2$ acts by the suspension of the antipodal action. Formula~\eqref{equation: second case} was first obtained by Ragnar Freij by different methods. See~\cite{freij}, Theorems 5.3 and 5.5. 

In the case $n=4$ we obtain a $\Sigma_2\times\Sigma_2$-equivariant equivalence
\[
|\Pi_4|\to \Sigma_{2}\times {\Sigma_2}_+ \wedge S^{1} \vee \Sigma_{2}\times {\Sigma_2}_+\wedge_{\Sigma_2} \hat S^{1}
\]
where $\hat S^1$ denotes a circle on which $\Sigma_2$ acts by reflection. This equivalence is illustrated in Figure~\ref{figure: n=4}. $|\Pi_4|$ is a one-dimensional complex. The map is defined by collapsing to a point a $\Sigma_2\times\Sigma_2$-equivariantly contractible subcomplex, depicted with thin lines. The subcomplex is drawn separately on the left in a form that makes its contractibility apparent. The reader is invited to check that the subcomplex of thin lines is $\Sigma_2\times\Sigma_2$-invariant, and that the quotient space of $|\Pi_4|$ by this subcomplex is homeomorphic to $\Sigma_{2}\times {\Sigma_2}_+ \wedge S^{1} \vee \Sigma_{2}\times {\Sigma_2}_+\wedge_{\Sigma_2} \hat S^{1}$.
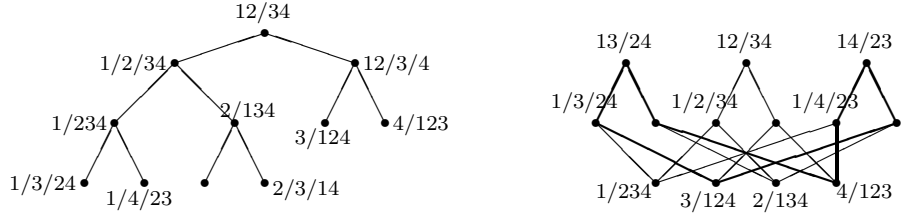
\begin{figure}\label{figure: n=4}
\caption{The $\Sigma_2\times \Sigma_2$-equivariant equivalence $|\Pi_4|\to \Sigma_{2}\times {\Sigma_2}_+ \wedge S^{1} \vee \Sigma_{2}\times {\Sigma_2}_+\wedge_{\Sigma_2} \hat S^{1}$ is defined by collapsing the contractible subcomplex drawn with thin lines. The collapsed subcomplex is drawn separately on the left.}

\setlength{\unitlength}{0.4cm}
\begin{picture}(40, 8)(-22, -4)
\multiput(-19,-3)(2,0){4}{\circle*{0.3}}
\multiput(-19, -3)(4,0){2}{\line(1,2){1}}
\multiput(-17, -3)(4,0){2}{\line(-1,2){1}}
\multiput(-18, -1)(4,0){2}{\circle*{0.3}}
\put(-18, -1){\line(1,1){2}}
\put(-14, -1){\line(-1,1){2}}
\put(-9, -1){\circle*{0.3}}
\put(-9, -1){\line(-1,2){1}}
\put(-11, -1){\circle*{0.3}}
\put(-11, -1){\line(1,2){1}}
\multiput(-16, 1)(6,0){2}{\circle*{0.3}}
\put(-16, 1){\line(3,1){3}}
\put(-10, 1){\line(-3,1){3}}
\put(-13,2){\circle*{0.3}}

\put(-14, 2.5){\tiny 12/34}
\put(-18.5, 0.75){\tiny 1/2/34}
\put(-9.75, 0.75){\tiny 12/3/4}
\put(-20.1, -1.25){\tiny 1/234}
\put(-14.5, -0.8){\tiny 2/134}
\put(-12, -1.7){\tiny 3/124}
\put(-8.75, -1.25){\tiny 4/123}

\put(-21.5, -3.25){\tiny 1/3/24}
\put(-18.3, -3.75){\tiny 1/4/23}
\put(-12.75, -3.4){\tiny 2/3/14}


\multiput(-1,1)(4,0){3}{\circle*{0.3}}
\multiput(-2,-1)(2,0){6}{\circle*{0.3}}
\multiput(0, -3)(2,0){4}{\circle*{0.3}}

\put(-2, 1.5){\tiny 13/24}
\put(2, 1.5){\tiny 12/34}
\put(6, 1.5){\tiny 14/23}
\put(-3.5, -0.5){\tiny 1/3/24}
\put(0.5, -0.5){\tiny 1/2/34}
\put(4.5, -0.5){\tiny 1/4/23}
\put(-2, -3.5){\tiny 1/234}
\put(0.8, -3.8){\tiny 3/124}
\put(3.2, -3.8){\tiny 2/134}
\put(6, -3.5){\tiny 4/123}


\put(2,-1){\line(1, 2){1}}
\put(4,-1){\line(-1, 2){1}}
\put(-2,-1){\line(1,-1){2}}
\put(0,-1){\line(2,-1){4}}
\put(2,-1){\line(1,-1){2}}
\put(2,-1){\line(-1,-1){2}}
\put(4,-1){\line(1,-1){2}}
\put(4,-1){\line(-1,-1){2}}
\put(6,-1){\line(-3,-1){6}}
\put(8,-1){\line(-2,-1){4}}
\thicklines
\linethickness{0.4mm}
\put(-2, -1){\line(1, 2){1}}
\put(0,-1){\line(-1, 2){1}}
\put(6, -1){\line(1, 2){1}}
\put(8,-1){\line(-1, 2){1}}

\put(-2,-1){\line(2,-1){4}}
\put(0,-1){\line(3,-1){6}}
\put(2, -3){\line(3,1){6}}
\put(6, -1){\line(0,-1){2}}
\end{picture}
\end{figure}
\end{examples}
Our Theorem~\ref{theorem: main} is a strengthening of the main result of~\cite{A-K}. There is an overlap of methods as well. In both papers, the proof uses the relationship between $|\Pi_n|$ and Lie algebras, along with a decomposition result for free groups, or free Lie algebras. In [loc. cit], the authors used the Hilton-Milnor theorem, which is a decomposition result for the topological free group generated by a wedge sum of connected spaces. Using the Hilton-Milnor theorem, it was possible to prove a weak version of Theorem~\ref{theorem: main}. In this paper we use an algebraic version of the Hilton-Milnor theorem, which is a decomposition result for a free Lie algebra on a direct sum. Using this algebraic version enables us to have a point-set level description of the map in~\eqref{equation: main}, and to prove that it is an equivariant homotopy equivalence.
%

The proof of the main theorem also depends on describing the fixed points space $|\Pi_n|^G$ for certain subgroups $G\subset \Sigma_n$. Then once proved, Theorem~\ref{theorem: main} can be used to describe the fixed points space for {\it all} subgroups $G\subset \Sigma_n$, in terms of posets of subgroups. These results are not really new (see Remark~\ref{remark: lukas} below), but they are perhaps not widely known, so we will give a brief outline.

Suppose $G$ acts effectively on $\n$. Each orbit of the action of $G$ on $\n$ is a $G$-invariant subset of $\n$. We say that the action of $G$ on $\n$ is {\it isotypical} if all the orbits of $G$ are pairwise isomorphic as $G$-sets. The following is proved in the paper as Lemma~\ref{lemma: not isotypical}.
\begin{lemma}\label{lemma: nondiagonal}
Suppose $G$ acts on $\n$, and the action is not isotypical. Then $|\Pi_n|^G\simeq *$.
\end{lemma}
We learned recently that this lemma was observed independently by Markus Hausmann. We suspect that it had been observed by others as well, but we do not know of a published reference. 

It remains to analyze the fixed points of isotypical actions. An important special case of isotypical actions is that of transitive actions. If $G$ acts transitively on $\n$, then $\n$ can be identified, as a $G$-set, with the coset space $G/H$, where $H$ is a subgroup of $G$ (determined up to conjugacy). For transitive actions, it is not difficult to show that 
the poset of $G$-invariant partitions of $G/H$ is isomorphic to the poset of subgroups of $G$ that contain $H$. This is Lemma~\ref{lemma: transitive} (it had appeared in the paper of White-Williamson~\cite[Lemma 3]{ww}, and they attributed the result to Klass~\cite{klass}).

Using the transitive case together with the main theorem, one can describe the fixed point space $|\Pi_n|^G$ for a general isotypical action of $G$ on $\n$ as follows. Suppose $G$ acts isotypically on $\n$. Then there exists a unique $d$ dividing $n$ such that $G$ is a transitive subgroup of $\Sigma_d$, embedded diagonally in $\Sigma_n$. Our description of $|\Pi_n|^G$ depends on this $d$. The following is proved in the paper as~\ref{proposition: fixed points}.
\begin{proposition}\label{proposition: fixed points intro}
Suppose that $\Sigma_d$ is embedded diagonally in $\Sigma_n$ and $G$ is a transitive subgroup of $\Sigma_d$. Let $C_{\Sigma_d}(G)$ be the centralizer of $G$ in $\Sigma_d$. There is an equivalence ($*$ denotes join)
\[
|\Pi_n|^G \simeq C_{\Sigma_d}(G)^{\frac{n}{d}}_+ \wedge_{C_{\Sigma_d}(G)} \Sigma {|\Pi_d|}^G * |\Pi_{\frac{n}{d}}|.
\]
\end{proposition} 
\begin{example}
There is some overlap between this part of the paper and~\cite{adl2}. In that paper, the authors analyzed the fixed point space of $|\Pi_n|^P$ where $P$ is a $p$-subgroup of $\Sigma_n$. More precisely, it was shown that the only $p$-groups $P$ for which the fixed points space $|\Pi_n|^P$ is not contractible are elementary abelian groups acting freely on $\n$. It also was observed in [loc. cit.] that if $n=p^k$ and $P\cong({\integers}/p)^k$ acts freely and transitively on $\n$, then $|\Pi_n|^P$ is isomorphic to the poset of proper non-trivial subgroups of $P$, which is closely related to the Tits building for $\GL_k(\field_p)$. We complete the calculation in Corollary~\ref{corollary: el abelian}, where we describe the fixed points for elementary abelian groups that act freely but not necessarily transitively. 
\end{example}

\begin{remark}\label{remark: lukas}
After making this paper public we learned that our results about fixed points of the partition poset (in particular Proposition~\ref{proposition: fixed points intro} and its consequences) were discovered independently by Lukas Brantner in the spring of 2014. He proved Proposition~\ref{proposition: fixed points intro} directly, using discrete Morse theory~\cite{brantner}.
\end{remark}

In addition to fixed points, we consider quotient spaces of $|\Pi_n|$ by an action of a Young subgroup. Here we seem to have something genuinely new to say. Quotient spaces of $|\Pi_n|$ by subgroups of $\Sigma_n$ have received some attention over the years. For starters, it is known that the quotient space of $|\Pi_n|$ by the full symmetric group is contractible for $n \ge 3$ (see Kozlov~\cite{kozlov}). Results about quotient spaces by certain subgroups of $|\Sigma_n|$ were obtained, for example, by Patricia Hersh~\cite{hersh} and Ralf Donau~\cite{donaumatching}.
Here we observe that Theorem~\ref{theorem: main} has the following immediate consequence regarding quotient by action of a Young subgroup.
\begin{proposition}\label{proposition: orbits}
Suppose $n=n_1+\cdots+n_k$. Let $g=\gcd(n_1, \ldots, n_k)$. There is a homotopy equivalence
\[
_{\Sigma_{n_1}\times\cdots\times\Sigma_{n_k}}\!\!\backslash|\Pi_n| \longrightarrow \bigvee_{d|g}\bigvee_{B(\frac{n_1}{d}, \ldots, \frac{n_k}{d})} {_{\Sigma_d}}\backslash \! \left(S^{n-d-1} \wedge |\Pi_d|^\diamond\right).
\]
\end{proposition}
Proposition~\ref{proposition: orbits} reduces the problem of calculating quotient spaces of  $\Pi_n$ by Young subgroups of $\Sigma_n$ to the problem of calculating quotient spaces of the form 
\begin{equation} \label{equation: orbits}
_{\Sigma_d}\!\backslash S^{ld-1}\wedge|\Pi_d|^\diamond.
\end{equation}
Here $l\ge 1$, and $S^{ld-1}$ is equipped with an action of $\Sigma_d$: desuspension of the natural action on $S^{ld}$. The following proposition summarizes our results about these quotient spaces.
\begin{proposition}\label{proposition: orbits omnibus}
The quotient space $_{\Sigma_d}\!\backslash S^{ld-1}\wedge |\Pi_d|^\diamond$ is:
\begin{enumerate}
\item rationally contractible unless $d=1$ or $d=2$ and $l$ is even.
\item equivalent to $\Sigma^l\reals P^{l-1}$ if $d=2$.
\item equivalent to the $p$-localization of the orbit space $_{\Sigma_p}\backslash S^{lp-1}$ if $d=p=$ an odd prime, and $l$ is odd.
\item equivalent to the $p$-localization of the homotopy cofiber of the quotient map $S^{lp-1}\to {_{\Sigma_p}}\backslash S^{lp-1}$ if $d=p=$ an odd prime, and $l$ is even.
\item contractible if $l=1$
\item contractible if $l=2$ and $d$ is an odd prime.
\end{enumerate}
\end{proposition}
In other words, we manage to analyze $_{\Sigma_d}\!\backslash S^{ld-1}\wedge |\Pi_d|^\diamond$ when $d$ is a prime (or $l=1$). We have no doubt that more complicated values of $d$ are within reach, and we are hopeful that eventually these spaces will be well-understood for all $l$ and $d$. In the meantime, propositions~\ref{proposition: orbits} and~\ref{proposition: orbits omnibus} enable us to 
\begin{enumerate}
\item describe the torsion-free part of $\HH_*(_{\Sigma_{n_1}\times\cdots\times\Sigma_{n_k}}\!\!\backslash|\Pi_n|)$ in all cases (but this is not new). This is Corollary~\ref{corollary: rational}. 
\item prove that $\widetilde\HH_*(_{\Sigma_{n_1}\times\cdots\times\Sigma_{n_k}}\!\!\backslash|\Pi_n|)$  has $p$-primary torsion only for primes $p$ that satisfy $p\le \gcd(n_1, \ldots, n_k)$ (Lemma~\ref{lemma: torsion}).
\item describe the homotopy type of $_{\Sigma_{n_1}\times\cdots\times\Sigma_{n_k}}\!\!\backslash|\Pi_n|$ when $\gcd(n_1, \ldots, n_k)$ equals $1$ or a prime number (Corollary~\ref{corollary: gcd one} and~Proposition~\ref{proposition: prime}).
\item describe the homotopy type of $_{\Sigma_{p^2}\times\Sigma_{p^2}}\!\!\backslash |\Pi_{2p^2}|$ where $p$ is a prime.
\item classify all Young subgroups for which the quotient space is equivalent to a wedge of spheres (Corollary~\ref{corollary: wedge of spheres})
\end{enumerate}
\subsubsection*{Organization of the paper}  In Section~\ref{section: generalities} we review background material on simplicial complexes (with emphasis on the partition complex), the join construction and equivariant topology. In Section~\ref{section: hilton} we review some generalities regarding free Lie algebras, and state the algebraic Hilton-Milnor theorem. 

In Section~\ref{section: lie} we show how the algebraic Hilton-Milnor theorem gives rise to a branching rule for the Lie representation. This is an algebraic analogue of our main theorem. 

In Section~\ref{section: lie to partitions} we review Barcelo's theorem on the isomorphism between the Lie representation and the cohomology of the partition complex. In Section~\ref{section: collapse} we define the map~\eqref{equation: main}, and prove that it induces an isomorphism in cohomology, by comparing it with the isomorphism that we constructed earlier in Section~\ref{section: lie}. 

In Section~\ref{section: fixed points} we analyze the fixed point space of $|\Pi_n|$ for some groups, including all groups that occur  as an isotropy of either the source or target of the map~\eqref{equation: main}. In Section~\ref{section: main proof} we use the calculations in the preceding section to show that the map in~\eqref{equation: main} induces an equivalence of fixed points spaces for every group that occurs as an isotropy of either the source or target of the map. We conclude that the map is an equivariant equivalence, thus proving the main theorem. 

In Section~\ref{section: application fixed} we use the main theorem to describe (in terms of subgroup posets) the fixed point space $|\Pi_n|^G$ for every $G\subset \Sigma_n$. In Section~\ref{section: application orbit} we use the main theorem to study the quotient space of $|\Pi_n|$ by an action of a Young subgroup.

\subsubsection*{Acknowledgements} The initial impetus for this work was provided by Ralf Donau's paper~\cite{donaumatching} and my desire to understand its connection with~\cite{A-K}. I thank Donau for several useful email exchanges, and in particular for sharing with me the results of some of his computer-aided calculations of the homology groups of quotient spaces of $|\Pi_n|$. 

I also thank the Hausdorff Institute of Mathematics in Bonn, where much of the writing was done, for its hospitality.

\section{Preliminaries} \label{section: generalities}
\subsection{Simplicial complexes and posets}\label{subsection: posets}
%

We will assume that the reader is familiar with basic notions concerning simplicial complexes, the order complex of a poset, and geometric realizaiton. Given a poset $\Pcal$, let $|\Pcal|$ denote its geometric realization.

Many posets have an initial and/or final element. We will denote the initial and final element of a poset $\Pcal$ by $\hat0$ and $\hat1$, or $\hat0_{\Pcal}$ and $\hat1_{\Pcal}$ if needed. Given a poset $\Pcal$, we denote the poset $\Pcal\setminus\{\hat 0 , \hat 1\}$ by $\overline{\Pcal}$.
\begin{example}\label{example: simplex}
Let $\Ccal_m$ be the poset of integers between $1$ and $m$, with the usual linear order. Then $|\Ccal_m|=\Delta^{m-1}$ is the standard $m-1$-dimensional simplex. The simplicial boundary of $\Ccal_{m}$ is the simplicial complex consisting of all the increasing chains in $\Ccal_m$, except for the unique maximal chain $1<\cdots<m$. Its geometric realization is the topological boundary of $\Delta^{m-1}$, i.e. the sphere $S^{m-2}$.
\end{example}
\begin{example}\label{example: sphere}
Let $\Bcal_m$ be the poset of subsets of $\m$. It is easy to see that there is an isomorphism of posets $\Bcal_m \cong (\Bcal_1)^m$. Since $|\Bcal_1|\cong [0,1]$, there is a homeomorphism $|\Bcal_m| \cong [0,1]^m$. Next, consider the posets $\Bcal_m\setminus\{\hat0\}$ -- the poset of non-empty subsets of $\m$, and $\overline{\Bcal_m}$ -- the poset of proper, non-empty subsets of $\m$. It is well-known, and not difficult to check, that the order complexes of $\Bcal\setminus\{\hat 0\}$ and $\overline{\Bcal_m}$ are isomorphic to the barycentric subdivisions of $\Ccal_{m}$ and of its simplicial boundary, respectively (see Example~\ref{example: simplex}). In particular, there is a homeomorphism $|\overline{\Bcal_m}| \cong S^{m-2}$. The action of $\Sigma_m$ on $\Bcal_m$ induces an unpointed action on $|\overline{\Bcal_m}|$, i.e. on $S^{m-2}$. Once again, it is well-known and not difficult to prove that this action is homeomorphic to the action of $\Sigma_m$ on the unit sphere in $\reals^{m-1}$, where $\reals^{m-1}$ is endowed with the reduced standard action of $\Sigma_m$. 
\end{example}

Given two elements $x$ and $y$ of a poset $\Pcal$, their least upper bound, if it exists, is denoted $x\vee y$. The greatest lower bound is denoted $x\wedge y$. In the course of the paper we will use several times the following elementary criterion for the contractibility of a poset.
\begin{lemma}\label{lemma: lattice}
Suppose $\Pcal$ has an element $\theta$ such that $x\vee \theta$ exists for all $x\in \Pcal$ (or alternatively, $x\wedge \theta$ exists for all $x$). Then $|\Pcal|$ is contractible.
\end{lemma}
\begin{proof}
The assumption implies that there is a well-defined map of posets $-\vee \theta \colon \Pcal \to \Pcal$. This map has the property that $x\le x\vee \theta$ for all $x\in \Pcal$. In other words, there is a natural transformation from the identity functor to $-\vee \theta$. It is well-known that a natural transformation between functors induces a homotopy between maps of geometric realizations. So the induced map of geometric realizations $|-\vee \theta|\colon |\Pcal| \to |\Pcal|$ is homotopic to the identity map. But this map factors through $|\Pcal_{\ge \theta}|$, where $\Pcal_{\ge\theta}\subset\Pcal$ is the subposet of all elements $x$ satisfying $x\ge \theta$. This poset has an initial element, namely $\theta$, so $|\Pcal_{\ge \theta}|$ is contractible. It follows that the map $|-\vee \theta|$ is null-homotopic, and therefore the identity map of $|\Pcal|$ is null-homotopic. It follows that $|\Pcal|$ is contractible. The proof under the assumption of existence $x\wedge \theta$ is similar.
\end{proof}

\subsubsection{The join operation on spaces and posets}
Next we will review the join operation, since it plays a role in this paper. 
First recall that the join of two topological spaces $X$ and $Y$, denoted $X*Y$, is the double mapping cylinder (a.k.a homotopy pushout) of the diagram
\[
X\leftarrow X\times Y \rightarrow Y.
\]
It is easy to check that the operation is associative up to natural isomorphism. It is well-known that if $X$ and $Y$  are pointed spaces, there is a natural homotopy equivalence
\begin{equation}\label{equation: joint}
X*Y\stackrel{\simeq}{\to} S^1\wedge X \wedge Y.
\end{equation}

This fact can be generalized slightly to a situation where only one of the spaces $X$ and $Y$ has a basepoint. Recall that $X^\diamond$ denotes the unreduced suspension of $X$. This space has a canonical basepoint (the ``south pole''). We note that $X^\diamond$ is homeomorphic to the pointed homotopy cofiber of the canonical map $X_+\to S^0$.
\begin{lemma}\label{lemma: join}
Suppose $Y$ is a well-pointed space. Then there is a natural weak equivalence
\[
X*Y\stackrel{\simeq}{\to} X^\diamond \wedge Y.
\]
\end{lemma}
\begin{proof}
Consider the following diagram
\[
\begin{array}{ccccc}
X & \leftarrow &  X\times * & \rightarrow & * \\
\downarrow & & \downarrow & & \downarrow \\
X & \leftarrow & X \times Y & \rightarrow & Y \\
\downarrow & & \downarrow & & \downarrow \\
* & \leftarrow & X_+ \wedge Y & \rightarrow & Y
\end{array}
\]
The columns in the diagram are homotopy cofibration sequences, so it induces a homotopy cofibration sequence of homotopy pushouts of the rows. The homotopy pushout of the top row is contractible, so we have an equivalence between the pushouts of the middle and bottom rows. The pushout of the middle row is $X*Y$. The bottom row is homeomorphic to the diagram obtained by taking object-wise smash product of the following diagram with $Y$
\[
*  \leftarrow  X_+   \rightarrow  S^0.
\]
Therefore, the homotopy pushout of the bottom row is the smash product of the unpointed homotopy cofiber of the map $X_+\to S^0$ with $Y$. Since the spaces $X_+, S^0$ and $Y$ are well-pointed, the unpointed homotopy cofiber can be replaced with the pointed homotopy cofiber.
\end{proof}
It follows that the expression $S^{n-d-1} \wedge |\Pi_d|^\diamond$ in the main theorem (see~\eqref{equation: main}) can be replaced with $S^{n-d-1} * |\Pi_d|$. This form seems to be more convenient for the purpose of proving the theorem.
\begin{examples}\label{examples: joins}
The empty set is a unit for the join operation. That is, there are natural isomorphisms $X* \emptyset \cong \emptyset * X \cong X$.

Taking join with $S^0$ is the same as taking unreduced suspension. We use notations $S^0* X$, $X^\diamond$ and $\Sigma X$ interchangeably.

A special case of Lemma~\ref{lemma: join} is the well-known equivalence $S^m * S^n\cong S^{m+n+1}$ (in fact these spaces are homeomorphic). We think of the empty set as the $-1$-dimensional sphere. Then the formula remains valid for $m, n\ge -1$.

The join of two simplices $\Delta^i * \Delta^j$ is homeomorphic to $\Delta^{i+j+1}$. It also is easy to see that $\Delta^0 * S^j$, the join of a point and a sphere, is homeomorphic to $\Delta^{j+1}$. It follows that for all $i\ge 0$, $j\ge -1$, $\Delta^i * S^j$ is homeomorphic to $\Delta^{i+j+1}$. 
\end{examples}

There is a version of the join operation for simplicial complexes. Let $(U, \Phi)$ and $(V, \Psi)$ be simplicial complexes. Their simplicial join is defined to be the simplicial complex $(U, \Phi)*(V, \Psi):=(U\coprod V, \Phi * \Psi)$, where $\Phi*\Psi$ is the collection of subsets $S\subset U\coprod V$ satisfying $S\cap U \in \Phi$ and $S\cap V\in \Psi$. It is well known that geometric realization takes simplicial join to space-level join.

Now let $\Pcal$ and $\Qcal$ be posets. The join of $\Pcal$ and $\Qcal$ is defined to be the following poset. On the level of underlying sets, $\Pcal * \Qcal= \Pcal \coprod \Qcal$. The partial order is defined by keeping the old order on $\Pcal$ and $\Qcal$, and decreeing that every element of $\Pcal$ is smaller than every element of $\Qcal$. It is easy to see that the order complex of $\Pcal * \Qcal$ is isomorphic to the simplicial join of the two order complexes. It follows that $|\Pcal * \Qcal|\cong |\Pcal| * |\Qcal|$.

There is another model for the join of two posets, that will be useful to us. First, some notation. For an element $a\in \Pcal$, let $\Pcal_{<a}=\{x\in\Pcal\mid x<a\}$. Define $\Pcal_{>a}, \Pcal_{\le a},$ and $\Pcal_{\ge a}$ analogously. Suppose $b\in \Qcal$. Then $( a, b)\in \Pcal\times \Qcal$. The following lemma is due to Quillen~\cite[Proposition 1.9]{quillen}. See also~\cite[Theorem 5.1 (b), (c)]{walker}. 
\begin{lemma}\label{lemma: poset join}
There are natural homeomorphisms
\[
|(\Pcal\times \Qcal)_{>(a, b)}|\cong |\Pcal_{>a}|*|\Qcal_{>b}|
\]
\[
|(\Pcal\times \Qcal)_{<( a, b)}|\cong |\Pcal_{<a}|*|\Qcal_{<b}|
\]
\end{lemma}

Let $\Pcal$ be a poset with an initial and final element. Recall that $\overline{\Pcal}=\Pcal\setminus\{\hat0, \hat1\}$. Suppose that $\overline{\Pcal}=\emptyset$. Then there are two possibilities. Either $\Pcal$ has two elements, or just one. In the first case, we identify $|\overline{\Pcal}|$ with the empty space, as expected. But if $\Pcal$ has just one element, which is then both initial and final, then we adopt the convention that $|\overline{\Pcal}|=S^{-2}$  (recall that we identify the empty space with $S^{-1}$). In practice this means that the space $|\overline{\Pcal}|$ is not defined, but any suspension of this space is. In particular $S^0* |\overline{\Pcal}| = S^{-1}=\emptyset$, and in general $S^n * |\overline{\Pcal}|\cong S^{n-1}$.

With this convention in place, we have the following formula for the ``product of closed intervals, with initial and final object removed''.
\begin{proposition}\label{proposition: fancy join}
Let $\Pcal_1, \ldots, \Pcal_n$ be posets, each with an initial and final object. There is a natural homeomorphism
\[
|\overline{\Pcal_1\times\cdots\times\Pcal_n}| \cong S^{n-2}* |\overline{\Pcal_1}| * \cdots * |\overline{\Pcal_n}|.
\]
Moreover, a permutation of factors on the left hand side corresponds on the right hand side to permutation of factors, and the standard action of a permutation on $S^{n-2}$ (as in Example~\ref{example: sphere}).
\end{proposition}
\begin{proof}
The case $n=2$ is~\cite[Theorem 5.1 (d)]{walker}. The general case is proved by a similar argument, plus induction. 
\end{proof}
Once again, note that if for some $i$ the poset $\Pcal_i$ consists of a single object, there is a homeomorphism
\[
S^{n-2} * |\overline{\Pcal_1}| * \cdots * |\overline{\Pcal_n}| \cong S^{n-3} * |\overline{\Pcal_1}| * \cdots * {|\overline{\Pcal_{i-1}}|} * {|\overline{\Pcal_{i+1}}|}\cdots*|\overline{\Pcal_n}|.
\]

\subsection{Poset of partitions} \label{subsection: partitions} Our main interest is in the poset of partitions (i.e., equivalence relations) of a finite set. For a finite set $S$, let $\Part_S$ be the collection of partitions of $S$, partially ordered by refinement. We say that $\lambda_1<\lambda_2$ if $\lambda_2$ is finer than $\lambda_1$ (this is opposite of the convention we used in some of our previous papers, but it seems convenient here). Throughout the paper, let $\n=\{1, \ldots, n\}$ be the standard set with $n$ elements. We denote the poset of partitions of $\n$ by $\Part_n$.

The poset $\Part_n$ has an initial and a final object. As usual we denote them by $\hat 0$ and $\hat 1$ (we will not include $n$ in the notation for these, trusting that it will not cause confusion). They may be called the ``indiscrete'' and the ``discrete'' partition respectively (note that for $n=1$, $\hat 0 = \hat 1$). 

We often will want to consider the poset of partitions without the initial and final object, so we introduce separate notation for it. For $n> 1$ define $\Pi_n=\Part_n\setminus\{\hat0, \hat 1\}$. It is well-known that for $n\ge 3$ there is a homotopy equivalence $|\Pi_n| \simeq \bigvee_{(n-1)!} S^{n-3}$. For $n=1, 2$ this equivalence can not be true as stated, if only because the right hand side is not defined. Note that $\Pi_2=\emptyset$, a.k.a, $S^{-1}$, which is compatible with the general formula. In keeping with the conventions of section~\ref{subsection: posets}, we identify $|\Pi_1|$ with $S^{-2}$. This means that $|\Pi_1|$ is not defined as a space, but $|\Pi_1|^\diamond=\emptyset$, and for all $n\ge 0$, $S^n *  |\Pi_1|=S^n \wedge|\Pi_1|^\diamond := S^{n-1}$. 

Let $\widetilde \HH^*(X)$ be the reduced cohomology of a space $X$. It is defined using the augmented simplicial or singular chain complex. Since for $n\ge 3$ $\Pi_n$ is a wedge of spheres of dimension $n-3$, it only has reduced homology in this dimension. For $n=2$ $\Pi_2=\emptyset$ and it is still true that $\widetilde \HH^{-1}(\emptyset)\cong \integers$, if one defines cohomology using the augmented chain complex. Since we defined $|\Pi_1|$ to be $S^{-2}$, we also define $\widetilde \HH^{-2}(\Pi_1)=\integers$. This convention will enable us to make statement valid for all $\Pi_n$, without needing to make exceptions for $n=1$.

We notice that the poset $\Part_n$ (and also $\Pi_n$) is equipped with a rank function: for a partition $\lambda$, let $\rk(\lambda)$ be the number of equivalence classes of $\lambda$. We note the following elementary properties of the rank function: if $\lambda_1<\lambda_2$ then $\rk(\lambda_1)<\rk(\lambda_2)$. In particular, distinct partitions of same rank are not comparable. The action of the symmetric group on $\Part_n$ preserves the rank function.

We will need to deal with intervals in $\Part_n$. Let $\lambda_1, \lambda_2$ be partitions, and suppose that $\lambda_1<\lambda_2$. Let us suppose that $\lambda_1$ has $k$ equivalence classes, and label them $c_1, \ldots, c_k$. Since $\lambda_1<\lambda_2$, each equivalence class of $\lambda_1$ is a union of classes of $\lambda_2$. Let us suppose that for $i=1, \ldots, k$, $c_i$ is the union of $m_i$ classes of $\lambda_2$. Then it is easy to see the following formula for the closed interval $[\lambda_1, \lambda_2]$
\[
[\lambda_1, \lambda_2]\cong \Part_{m_1}\times \cdots \times \Part_{m_k}.
\]
It follows that the open interval $(\lambda_1, \lambda_2)$ is isomorphic to the poset $\overline{\Part_{m_1}\times \cdots \times \Part_{m_k}}$. By Proposition~\ref{proposition: fancy join}, there is a homeomorphism
\begin{equation}\label{equation: open interval}
|(\lambda_1, \lambda_2)|\cong S^{k-2}* |\Pi_{m_1}| * \cdots * |\Pi_{m_k}|
\end{equation}
\begin{remark}\label{remark: more than one}
We will need the following simple observation: if $\lambda_1 < \lambda_2$, and $\rk(\lambda_1)<r<\rk(\lambda_2)$ then there are at least two partitions $\lambda$ of rank $r$ satisfying $\lambda_1 < \lambda < \lambda_2$. For example, if $\rk(\lambda_2)-\rk(\lambda_1)>2$ then an intermediate partition will have rank equal to $\rk(\lambda_1)+1$. There are two possibilities: either $\lambda_2$ is obtained from $\lambda_1$ by splitting two equivalence classes in two parts each, or by splitting a single class in three parts. In the first case, there are two intermediate partitions between $\lambda_1$ and $\lambda_2$, in the second case there are three.
\end{remark}
\subsubsection{Stars and links}
Let $\Pcal$ be a poset. Let $P=(p_0<\cdots < p_k)$ be a non-empty increasing chain in $\Pcal$. The star of $P$ is the poset of elements of $\Pcal$ that are comparable with each one of the elements $p_0, \ldots, p_k$. Denote it $\star(P)$. The star decomposes canonically as a join of posets. More precisely
\[
\star(P)\cong \Pcal_{<p_0} * \{p_0\} * (p_0, p_1) * \{p_1\}*\cdots* \{p_k\} * \Pcal_{>p_k}.
\]
It follows in particular that $|\star(P)|$ is contractible, since the join of any space with a point is contractible.

The link of $P$, denoted $\link(P)$ is the subposet of $\star(P)$ consisting of elements that are distinct from $p_0, \ldots, p_k$. The link also decomposes as a join. More precisely
\[
\link(P)\cong \Pcal_{<p_0} * (p_0,p_1) *\cdots*  \Pcal_{>p_k}.
\] 

Let us consider the case of $\Pi_n$. Let $\Lambda=(\lambda_0<\cdots<\lambda_k)$ be an increasing chain of partitions in $\Pi_n$. Then $|\star(\Lambda)|$ is a subspace of $|\Pi_n|$. We will need to identify the boundary of this subspace.
\begin{lemma}\label{lemma: boundary}
The boundary of $|\star(\Lambda)|$ is the geometric realization of the subcomplex of $\star(\Lambda)$ defined by chains that do not contain $\Lambda$ as a subchain.
\end{lemma}
\begin{proof}
Let $\Theta=(\theta_0<\cdots<\theta_m)$ be an increasing chain in $\star(\Lambda)$. It belongs to the boundary precisely if there exists a partition $\mu$ that is comparable with all the $\theta_j$s, but not all the $\lambda_i$s. If $\Lambda$ is a subchain of $\Theta$ then clearly such a partition does not exists. Conversely, if $\Theta$ does not contain one of the $\lambda_i$s then by remark~\ref{remark: more than one} one can find another partition $\mu$ of the same rank as $\lambda_i$ that is comparable with all the $\theta_j$s. Then $\mu$ is not comparable with $\lambda_i$.
\end{proof}

\begin{definition}\label{definition: binary}
Let us say that an increasing chain of partitions $\Lambda=(\lambda_0<\cdots<\lambda_k)$ is {\it binary} if for all $i\ge 0$, each equivalence class of $\lambda_i$ is the union of at most two equivalence classes of $\lambda_{i+1}$. For $i=k$ this means that each component of $\lambda_k$ has at most two elements. However, note that we do not require that $\lambda_0$ has two equivalence classes. 
\end{definition}
We need to analyze the boundary of $\star(\Lambda)$ for a binary chain. We already observed that
\[
\star(\Lambda)=(\hat0, \lambda_0) * \{\lambda_0\} * (\lambda_0, \lambda_1) * \cdots * \{\lambda_k\}*(\lambda_k, \hat1).
\]
By~\eqref{equation: open interval}, $(\hat0, \lambda_0)\cong \Pi_c$, where $c$ is the set of equivalence classes of $\lambda_0$. The geometric realization of an open interval $(\lambda_{i-1}, \lambda_i)$ is homeomorphic to a join of a sphere and, for each equivalence class of $\lambda_{i-1}$ a copy of $|\Pi_{m_j}|$, where $m_j$ is number of classes of $\lambda_i$ within the given class of $\lambda_{i-1}$. The assumption that $\Lambda$ is binary means precisely that $m_j\le 2$ in all cases. This means that each $\Pi_{m_j}$ is either $S^{-1}$, i.e., empty, if $m_j=2$ or is $S^{-2}$, if $m=1$. Recall that taking join of a sphere with $S^{-2}$ has the effect of lowering the dimension of a sphere by one. This means that $|(\lambda_{i-1}, \lambda_i)|$ is homeomorphic to a sphere of dimension $k-2$, where $k$ is the number of equivalence classes of $\lambda_{i-1}$ that are the union of two classes of $\lambda_i$. It follows that $|\{\lambda_0\} * (\lambda_0, \lambda_1) * \cdots * \{\lambda_k\}*(\lambda_k, \hat1)|$ is homeomorphic to a simplex. Its dimension is easily calculated to be $n-c-1$. We conclude that there are homeomorphisms
\begin{equation}\label{equation: binary star}
|\star(\Lambda)|\cong \star(\Lambda)_{<\lambda_0}*\star(\Lambda)_{\ge \lambda_0}\cong |\Pi_c|* \Delta^{n-c-1}
\end{equation}
Let $\partial\Delta^{n-c-1}$ be the topological boundary of the simplex. It is homeomorphic to $S^{n-c-2}$ of course. We may identify $|\Pi_c|* \partial\Delta^{n-c-1}$ with a subspace of $|\star(\Lambda)|$. Our next goal is to prove the unsurprising assertion that this is in fact the boundary of $|\star(\Lambda)|$ considered as a subspace of $|\Pi_n|$.
\begin{proposition}\label{proposition: boundary}
Using the notation above, the boundary of $|\star(\Lambda)|$ is precisely $|\Pi_c|* \partial\Delta^{n-c-1}$.
\end{proposition}
\begin{proof}
We already saw that there is a decomposition
\[
\star(\Lambda)=\star(\Lambda)_{<\lambda_0}*\star(\Lambda)_{\ge \lambda_0}\cong \Pi_c * \left(\{\lambda_0\} * (\lambda_0, \lambda_1) * \cdots * \{\lambda_k\}*(\lambda_k, \hat1)\right).
\]
By Lemma~\ref{lemma: boundary}, the boundary of $|\star(\Lambda)|$ is given by the subcomplex spanned by all chains that do not contain $\Lambda$ as a subchain. Let $\operatorname{Bd}(\star(\Lambda)_{\ge \lambda_0})$ be the subcomplex of $\star(\Lambda)_{\ge \lambda_0}$ spanned by chains that do not contain $\Lambda$ as a subchain. Clearly, the boundary of $|\star(\Lambda)|$ is the geometric realization of $\Pi_c* \operatorname{Bd}(\star(\Lambda)_{\ge \lambda_0})$. So we need to show that $\operatorname{Bd}(\star(\Lambda)_{\ge \lambda_0})$ is in fact the topological boundary of $\star(\Lambda)_{\ge \lambda_0}$.

We saw that $\star(\Lambda)_{\ge \lambda_0}$ is combinatorially equivalent to a simplex. It is well-known, and not hard to prove, that a chain belongs to the topological boundary if and only if the geometric realization of the link of the chain is contractible. So we need to prove that the link of a chain in $\star(\Lambda)_{\ge \lambda_0}$ is contractible precisely if it does not contain $\Lambda$ as a subchain.

Let $\Theta=(\theta_0< \cdots < \theta_m)$ be a chain in $\star(\Lambda)_{\ge \lambda_0}$. The link of this chain is the join of intervals
\[
[\lambda_0, \theta_0)* (\theta_0, \theta_1)*\cdots*(\theta_m, \hat1),
\]
where the intervals are taken in $\star(\Lambda)_{\ge \lambda_0}$. If $\Theta$ does not contain $\Lambda$ as a subchain, then one of these intervals contains a $\lambda_i$. Let us suppose that $\lambda_i\in (\theta_{j-1}, \theta_j)$. Then $(\theta_{j-1}, \theta_j)=(\theta_{j-1}, \lambda_i)*\{\lambda_i\}*(\lambda_i, \theta_j)$, and so the geometric realization is contractible. Conversely, if $\Theta$ does contain $\Lambda$ as a subchain then the link is a join of copies of $\Pi_m$ for various $m$, and therefore is not contractible.
\end{proof}
\subsection{Equivariant topology}
Let $G$ be a finite group. We only want to consider nice spaces with an action of $G$. Thus when we say that $X$ is a $G$-space, we assume that $X$ is homeomorphic, as a $G$ space, to the geometric realization of a $G$-simplicial set. Equivalently, we could choose to work with $G$-CW complexes. 

A (finite, pointed) $G$ space $X$ can be written as a (finite, pointed) homotopy colimit of spaces of the form $G/H$, where $H$ is an isotropy group of $X$. 

A $G$-map $f\colon X \to Y$ is a $G$-equivariant weak equivalence if it induces weak equivalences of fixed point spaces $f^H\colon X^H \to Y^H$ for every subgroup $H$ of $G$. It is well-known that for nice spaces, an equivariant weak equivalence is an equivariant homotopy equivalence. 

In fact, it is sometimes not necessary to check the fixed points condition for every subgroup of $G$. It is enough to do it for subgroups that occur as an isotropy of one of the spaces in questions. A proof of the following lemma can be found in~\cite[4.1]{dwyer}, but it is older than that.
\begin{lemma}\label{lemma: dwyer lemma}
Let $f\colon X \to Y$ be a $G$-map. Suppose that $f$ induces an equivalence $f^H\colon X^H \to Y^H$ for every subgroup $H$ that occurs as an isotropy of $X$ or of $Y$. Then $f$ is an equivariant equivalence.
\end{lemma}

Suppose $K, H$ are subgroups of $G$ and $X$ is a pointed space with an action of $H$. Then one may form the $G$-space $G_+\wedge_H X$. We will need several times to calculate the fixed points space $(G_+\wedge_H X)^K$. The following lemma is proved by a standard calculation with the double cosets formula.
\begin{lemma}\label{lemma: double cosets}
Let $N_G(K; H)=\{g\in G\mid g^{-1}Kg\subset H\}$. Note that $H$ acts on this set on the right.
There is a homeomorphism
\[
(G_+\wedge_H X)^K \cong \bigvee_{[g]\in N_G(K; H)/H} X^{g^{-1}Kg}.
\]
\end{lemma}
In particular, if $K$ is not subconjugate to $H$ then $(G_+\wedge_H X)^K \cong *$.

We will often have to deal with the following special case: $G\subset \Sigma_d\subset \Sigma_n$, where $d|n$ and $\Sigma_d$ is embedded diagonally in $\Sigma_n$ (we hope that the reader will not be confused by the change in the role of $G$).
\begin{lemma}\label{lemma: closed system}
Suppose $G\subset \Sigma_d\subset \Sigma_n$, where $\Sigma_d$ is embedded in $\Sigma_n$ diagonally. Let $\sigma$ be an element of $N_{\Sigma_n}(G; \Sigma_d)$. Thus $\sigma$ is an element of $\Sigma_n$ that conjugates $G$ to a subgroup of $\Sigma_d$. Then there exists an element $\eta\in \Sigma_d$ such that  conjugation by $\eta$ on $G$ is the same as conjugation by $\sigma$.
\end{lemma}
\begin{proof}
An inclusion of $G$ into $\Sigma_n$ is the same thing as an effective action of $G$ on $\n$. If the inclusion factors through $\Sigma_d$, then $\n$ is isomorphic, as a $G$-set, to a disjoint union of $\frac{n}{d}$ copies of the set $\mathbf d$, which equipped with an effective action of $G$. Specifying an element $\sigma$ of $\Sigma_n$ that conjugates $G$ into a subgroup of $\Sigma_d$ is the same as specifying a second action of $G$ on ${\mathbf d}$ (let ${\mathbf d}'$ denote $\mathbf d$ with the second action of $G$), a group isomorphism $\sigma_G\colon G\to G$, and an isomorphism of sets 
\[
\sigma\colon \coprod_{\frac{n}{d}}{\mathbf d}\to \coprod_{\frac{n}{d}}{\mathbf d}'
\]
that is equivariant with respect to the isomorphism $\sigma_G$. Such an isomorphism exists if and only if $\coprod_{\frac{n}{d}}{\mathbf d}$ and $\coprod_{\frac{n}{d}}{\mathbf d}'$ are isomorphic permutation representations of $G$. But this is possible if and only if $\mathbf d$ and ${\mathbf d}'$ are isomorphic permutation representations of $G$, which means that the isomorphism $\sigma_G$ can be realized as conjugation by an element of $\Sigma_d$.
\end{proof}
Let $C_W(G)$ denote the centralizer of $G$ in $W$.
\begin{corollary}\label{corollary: centralizer}
Suppose $G\subset \Sigma_d\subset W\subset \Sigma_n$  where $\Sigma_d$ is embedded diagonally in $\Sigma_n$, and $W$ is any intermediate subgroup. There is an isomorphism of sets, induced by inclusion
\[
C_{W}(G)/C_{\Sigma_d}(G) \stackrel{\cong}{\to} N_{W}(G; \Sigma_d)/\Sigma_d.
\]
\end{corollary}
\begin{proof}
By definition, $$N_{W}(G; \Sigma_d)=\{w\in W\mid w^{-1} G w\subset \Sigma_d\}.$$ By Lemma~\ref{lemma: closed system} for every such $w$ one can find an $\eta\in \Sigma_d$ such that conjugation by $\eta$ on $G$ coincides with conjugation by $w$. It follows that $w\eta^{-1}\in C_{W}(G)$. Thus ever element of $N_{W}(G; \Sigma_d)/\Sigma_d$ has a representative in $C_{W}(G)$. Clearly two elements of $C_{W}(G)$ differ by an element of $\Sigma_d$ if and only if they differ by an element of $C_{\Sigma_d}(G)$. The corollary follows.
\end{proof}
\begin{corollary}\label{corollary: simplified fixed points}
Suppose $G\subset \Sigma_d\subset W\subset\Sigma_n$ where $\Sigma_d$ is embedded diagonally in $\Sigma_n$, and $W$ is any intermediate subgroup. Suppose $X$ is a space with an action of $\Sigma_d$. Then there is a homeomorphism
\[
(W_+ \wedge_{\Sigma_d} X)^G\cong C_W(G)_+\wedge_{C_{\Sigma_d}(G)} X^G.
\]
In particular, if $X^G$ is contractible, then so is $(W_+ \wedge_{\Sigma_d} X)^G$.
\end{corollary}
\begin{proof}
By Lemma~\ref{lemma: double cosets} there is a homeomorphism
\[
(W_+ \wedge_{\Sigma_d} X)^G\cong \bigvee_{N_W(G; \Sigma_d)/\Sigma_d} X^{w^{-1}Gw}.
\]
By Corollary~\ref{corollary: centralizer}, $N_W(G; \Sigma_d)/\Sigma_d\cong C_{W}(G)/C_{\Sigma_d}(G)$. In particular, $w$ can always be chosen in the centralizer of $G$, so $X^{w^{-1}Gw}= X^G$. It follows that 
\[
(W_+ \wedge_{\Sigma_d} X)^G\cong \bigvee_{C_{W}(G)/C_{\Sigma_d}(G)} X^{G}.
\]
The right hand side is isomorphic to $C_W(G)_+\wedge_{C_{\Sigma_d}(G)} X^G$ but the latter presentation is better in the sense that it gives the correct action of $C_W(G)$.
\end{proof}
\section{Free Lie algebras and the algebraic Hilton-Milnor theorem}\label{section: hilton}
For a free abelian group $V$, let $\Lie[V]$ be the free Lie algebra (over $\integers$) generated by $V$. 
If $V$ has a $\integers$-basis $x_1, \ldots, x_k$, we denote $\Lie[V]$ by $\Lie[x_1, \ldots, x_k]$. 

The multiplication in $\Lie[V]$ satisfies the following well-known relations: the square zero relation $[u, u]=0$ (which implies anticommutativity) and the Jacobi relation $[u, [v, w]] +[w, [u, v]] +[v, [w, u]] =0$. In fact, $\Lie[x_1, \ldots, x_k]$ can be constructed as the quotient of the free non-associative and non-unital algebra on $\{x_1, \ldots, x_k\}$ by the two sided ideal generated by these relations (for example see \cite{reutenauer}, Theorem 0.4 and its proof).

Throughout the paper, monomials are understood to be in non-commutative and non-associative variables. We will also refer to them as parenthesized monomials. Elements of $\Lie[x_1, \ldots, x_k]$ can be represented as linear combinations of parenthesized monomials in $x_1, \ldots, x_k$. As is customary in the context of Lie algebras, we will use square brackets to indicate the parenthesization. For example, $[[x_1, [x_2, x_1]], x_2]$ represents an element of $\Lie[x_1, x_2]$. 


Let $(n_1, \ldots, n_k)$ be non-negative integers. We say that a monomial in $x_1, \ldots, x_k$ has multi-degree $(n_1, \ldots, n_k)$ if it contains $n_i$ copies of $x_i$, for all $i=1, \ldots, k$. For example, the monomial $[[x_1, [x_2, x_1]], x_2]$ has multi-degree $(2,2)$. We define $M(n_1, \ldots, n_k)\subset \Lie[x_1, \ldots, x_k]$ to be the $\integers$-submodule generated by monomials of multi-degree $(n_1, \ldots, n_k)$. Note that $M(0,\ldots, 0)$ is the trivial group. 

The following theorem is well-known. 
\begin{theorem}\label{theorem: free Lie}
Additively, $\Lie[x_1, \ldots, x_k]$ is a free abelian group. Moreover, there is an additive isomorphism
\begin{equation}\label{equation: decomposition}
\Lie[x_1, \ldots, x_k] \cong \bigoplus_{n_1, \ldots, n_k\ge 0} M(n_1, \ldots, n_k)
\end{equation}
\end{theorem}
\begin{proof}
For the first assertion see, for example,~\cite[Corollary 0.10]{reutenauer}. The second assertion follows from the first and the fact that the square zero relation and the Jacobi relations are generated by relations that are sums of monomials of the same degree.
%
\end{proof}
Note that if $u\in M(n_1, \ldots, n_k)$ and $v\in M(m_1, \ldots, m_k)$ then $[u, v]$ and $[v, u]$ are in $M(n_1+m_1, \ldots, n_k+m_k)$. It follows that isomorphism~\eqref{equation: decomposition} makes $\Lie[x_1, \ldots, x_k]$ into an $\naturals^k$-graded algebra (here $\naturals$ stands for non-negative integers). It also has an $\naturals$-grading, given by the total number of generators in a monomial.

It follows from Theorem~\ref{theorem: free Lie} that each $M(n_1, \ldots, n_k)$ is a free Abelian group. Next, we choose an explicit bases for these groups. There are many known bases for the free Lie algebra. We will use the `basic monomials' defined by M. Hall and P. Hall. We will now recall the construction. 
\begin{definition}
Basic monomials in $x_1, ..., x_k$ and their weights are defined inductively as follows: To begin with, $x_1, \ldots, x_k$ are basic monomials of weight one. Next, suppose we have defined all the basic monomials of weight less than $n$, and ordered them in some way, so that monomials of lower weight precede those of higher weight. Let us say that the basic monomials of weight less than $n$ are $u_1, \ldots, u_\alpha$, for some $\alpha\ge k$. Then basic monomials of weight $n$ are all monomials of the form $[u_i, u_j]$ such that 
\begin{enumerate}
\item $u_i$ and $u_j$ are basic monomials of weight less than $n$, and the sum of their weights equals n, \item $i>j$ \item either $u_i=x_i$ for some $i\le k$, or $u_i=[u_s, u_t]$  and $t\le j$.
\end{enumerate}
Now order the newly formed basic monomial of weight $n$ in some way, append them to the list of basic monomials, and repeat the procedure. 
\end{definition}
It is well-known that basic monomials form a basis for $\Lie[x_1, \ldots, x_k]$ \cite{mhall, phall}. Let $B(n_1, \ldots, n_k)$ be the set of basic monomials of multi-degree $(n_1, \ldots, n_k)$. Then $B(n_1, \ldots, n_k)$ is a basis for $M(n_1, \ldots, n_k)$. The size of $B(n_1, \ldots, n_k)$ is given by the Witt formula~\eqref{equation: witt}.

Last topic in this section is the algebraic Hilton-Milnor theorem. In topology, the Hilton-Milnor theorem is a statement about the homotopy decomposition of the space $\Omega\Sigma(X_1\vee\ldots\vee X_k)$ as a weak product indexed by basic monomials in $k$ variables $x_1, \ldots, x_k$. Equivalently, one can view the theorem as a statement about the homotopy decomposition of the pointed free topological group generated by $X_1, \ldots, X_k$, where $X_i$ are connected spaces. This is so because the functor $\Omega \Sigma$ is equivalent to the pointed free topological group functor. By algebraic Hilton-Milnor theorem we mean an analogous decomposition of a free Lie algebra on a direct sum. Let us point out a couple of differences between the topological and algebraic versions of the theorem: the algebraic version gives an isomorphism, rather than a homotopy equivalence, and it does not require a connectivity hypothesis.

Before we state the theorem, we need to define a certain operation on monomials. 
\begin{definition}\label{definition: resolution}
Let $w\in B(n_1, \ldots, n_k)$ (more generally, $w$ can be any monomial of multi-degree $(n_1, \ldots, n_k)$). By definition, $w$ is a monomial in $x_1, \ldots, x_k$ that contains $n_i$ copies of $x_i$ for each $i$. The {\it resolution} of $w$, denoted $\tilde w$, is the multi-linear monomial in $n$ variables obtained from $w$ by replacing the $n_1$ occurrences of $x_1$ with $x_1, \ldots, x_{n_1}$, in the order in which they occur in $w$, replacing the $n_2$ occurrences of $x_2$ with $x_{n_1+1}, \ldots, x_{n_1+n_2}$, and so on. 
\end{definition}
\begin{example}
Consider the following monomial of multi-degree $(2,2)$ $w(x_1, x_2)=[[[x_2, x_1], x_1], x_2]$. To obtain the resolution of $w$ replace the two occurrences of $x_1$ with $x_1$ and $x_2$, and the two occurrences of $x_2$ with $x_3$ and $x_4$. We obtain the monomial $\tilde w(x_1, x_2, x_3, x_4)=[[[x_3, x_1], x_2], x_4]$ in $\Lierep_4$. 
\end{example}

Now let $V_1, \ldots, V_k$ be free Abelian groups. Let $w\in B(n_1, \ldots, n_k)$ be a basic monomial. We associate with $w$ a group homomorphism
\[
\alpha_w\colon V_1^{\otimes n_1}\otimes \cdots\otimes V_k^{\otimes n_k} \to \Lie[V_1 \oplus \cdots\oplus V_k].
\]
The homomorphism is defined by the following formula 
\[
\alpha_w\left(\otimes_{i=1}^k\otimes_{j=1}^{n_i} v_i^j\right)=\tilde w(v_1^1, v_1^2, \ldots, v_1^{n_1}, v_2^1, \ldots, v_k^{n_k}).
\] 
where $v_i^j\in V_i$ for all $1\le i \le k, 1\le j \le n_i$. Since $\tilde w$ is a multi-linear monomial in the correct number of variables, this is a well-defined homomorphism.
\begin{example}
Let $w=[[[x_2, x_1], x_1], x_2]$. Then $w$ determines a homomorphism $\alpha_w\colon V_1^{\otimes 2}\otimes V_2^{\otimes 2}\to \Lie[V_1\oplus V_2]$, by the formula
\[
\alpha_w(v_1^1 \otimes v_1^2 \otimes v_2^1\otimes v_2^2)= \tilde w(v_1^1, v_1^2, v_2^1, v_2^2) = [[[v_3, v_1], v_2], v_4].
\]
\end{example}
The homomorphism $\alpha_w$ extends to a homomorphism of Lie algebras 
\[
\tilde\alpha_w\colon \Lie[V_1^{\otimes n_1}\otimes \cdots\otimes V_k^{\otimes n_k}] \to \Lie[V_1 \oplus \cdots\oplus V_k].
\]
Taking direct sum over $w\in B(n_1, \ldots, n_k)$ we obtain a homomorphism of abelian groups
\[
\Sigma\, \tilde\alpha_w\colon \bigoplus_{w\in B(n_1, \ldots, n_k)} \Lie[V_1^{\otimes n_1}\otimes \cdots\otimes V_k^{\otimes n_k}] \to \Lie[V_1 \oplus \cdots\oplus V_k].
\]
\begin{remark}
There is an isomorphism 
\[
\bigoplus_{w\in B(n_1, \ldots, n_k)} \Lie[V_1^{\otimes n_1}\otimes \cdots\otimes V_k^{\otimes n_k}] \cong M(n_1, \ldots, n_k) \otimes  \Lie[V_1^{\otimes n_1}\otimes \cdots\otimes V_k^{\otimes n_k}] 
\]
so we constructed a homomorphism 
\[
M(n_1, \ldots, n_k) \otimes  \Lie[V_1^{\otimes n_1}\otimes \cdots\otimes V_k^{\otimes n_k}] \to \Lie[V_1 \oplus \cdots\oplus V_k].
\]
However, we emphasize that this isomorphism is not canonical - it depends on a choice of basis of $M(n_1, \ldots, n_k)$. 
\end{remark}

Finally, taking direct sum over all $n$-tuples $(n_1, \ldots, n_k)$, we obtain a homomorphism of abelian groups.
\begin{equation}\label{equation: hilton-milnor map}
\bigoplus_{n_1, \ldots, n_k\ge 0} \bigoplus_{B(n_1, \ldots, n_k)} \Lie[V_1^{\otimes n_1}\otimes \cdots\otimes V_k^{\otimes n_k}] \to \Lie[V_1 \oplus \cdots\oplus V_k].
\end{equation}
This is not a homomorphism of Lie algebras, but notice that it is a homomorphism of graded abelian groups, provided that on the left hand side one endows $V_1^{\otimes n_1}\otimes \cdots\otimes V_k^{\otimes n_k}$ with grading $n_1+\cdots+n_k$.

The following is the algebraic Hilton-Milnor theorem
\begin{theorem}\label{theorem: hilton-milnor}
The homomorphism~\eqref{equation: hilton-milnor map} is an isomorphism.
\end{theorem}
\begin{proof}
Let $V, W$ be free Abelian groups. Consider the natural split surjection of Lie algebras. $\Lie[V \oplus W]\to \Lie[V]$. The key to the proof is the following lemma, which identifies the kernel  of this homomorphism as a free Lie algebra with an explicit set of generators.
\begin{lemma}\label{lemma: neisendorfer}
There is a short exact sequence of free Lie algebras, (which splits on the level of graded abelian groups). 
\[
0\to \Lie\left[\bigoplus_{i=0}^\infty W \otimes V^{\otimes i}\right]  \to \Lie[V\oplus W]\to \Lie[V]\to 0.
\]
Here the first map is determined by the group homomorphisms $W\otimes V^{\otimes i}\to \Lie[V\oplus W]$ that send the element $w\otimes v_1\otimes\cdots\otimes v_i$ to the iterated commutator $[...[w, v_1], v_2], \ldots] v_i]$.
\end{lemma}
The lemma is proved as part of Example 8.7.4 (``The Hilton-Milnor example'') in Neisendorfer's book~\cite{neisendorfer}. It should be noted that Neisendorfer works with graded Lie algebras, and in the proof he uses some lemmas that assume that Lie algebras are zero-connected. However, Lemma~\ref{lemma: neisendorfer} holds for ungraded Lie algebras as well. To see this, notice that we can make our algebras connected by assigning $V$ and $W$ grading $2$, and this does not change anything. Whether $V$ has grading $0$ or $2$, the resulting Lie algebras $L[V]$ are canonically isomorphic as graded abelian groups. 

Once we have Lemma~\ref{lemma: neisendorfer}, the rest of the proof of the theorem proceeds exactly as for the classic Hilton-Milnor theorem. For details, see the proof of Theorem 4 in~\cite{milnor}. The input for that theorem is the exact analogue of Lemma~\ref{lemma: neisendorfer} for topological free groups. The proof of our theorem can emulate Milnor's proof almost word for word. One just needs to replace ``free group'' with ``free Lie algebra'', ``homotopy equivalence'' with ``isomorphism'' and ``complexes'' with ``free Abelian groups''.
\end{proof}
\section{From Lie algebras to the Lie representation}\label{section: lie}
Recall from previous section that $M(1, \ldots, 1)\subset \Lie[x_1, \ldots, x_n]$ is the submodule generated by monomials that contain each generator exactly once. We call such monomials multilinear. By the Witt formula, $M(1, \ldots, 1)$ is a free abelian group of rank $(n-1)!$. Moreover, there is a natural action of $\Sigma_n$, permuting the generators. We arrange it so that it is a right action. We call $M(1, \ldots, 1)$, considered as an integral representation of $\Sigma_n$, the Lie representation, and denote it by $\Lierep_n$

The sequence of representations $\Lierep_1, \Lierep_2, \ldots, \Lierep_n,\ldots$ forms an operad in abelian groups - the Lie operad. Let us recall what this means. Suppose that we have an isomorphism $\phi\colon \mathbf{n_1}\coprod \cdots\coprod \mathbf{n_i} \stackrel{\cong}{\to} \n$. Then there is a homomorphism associated with $\phi$, 
\[
s_\phi\colon\Lierep_i\otimes\Lierep_{n_1}\otimes\cdots\otimes\Lierep_{n_i}\to \Lierep_n.
\]
The homomorphism $s_\phi$ is natural with respect to isomorphisms of the sets $\mathbf{n_1}, \ldots, \mathbf{n_i}$ and also with respect to permutations of these set. Let us describe $s_\phi$ explicitly in terms of monomials. Recall that $\Lierep_n$ is generated by (parenthesized) multilinear monomials in variables $x_1, \ldots, x_n$. Let $$w_i(x_1, \ldots, x_i), w_{n_1}(x_1, \ldots, x_{n_1}), \ldots, w_{n_i}(x_1, \ldots, x_{n_i})$$ be monomials representing elements of $\Lierep_i, \Lierep_{n_1}, \ldots, \Lierep_{n_i}$ respectively. To distinguish between elements of $\mathbf{n_1}, \mathbf{n_2}, $ etc., let $1^j, \ldots {n_j}^j$ be elements of $\mathbf{n_j}$. Then $s_\phi(w_i\otimes w_{n_1}, \ldots, w_{n_i})(x_1, \ldots, x_n)$ equals to
\[
w_i(w_{n_1}(x_{\phi(1^1)}, \ldots, x_{\phi({n_1}^1)}), \ldots, w_{n_i}(x_{\phi(1^i)}, \ldots, x_{\phi({n_i}^i)})).
\] 
In words, $s_\phi(w_i\otimes w_{n_1}, \ldots, w_{n_i})$ is obtained by taking the monomial $w_i(x_1,\ldots, x_i)$, substituting the monomials $w_{n_1}, \ldots, w_{n_i}$ for the variables $x_1, \ldots, x_i$, and then reindexing the variables as prescribed by the isomorphism $\phi$.

Unless said otherwise, we take $\phi$ to be the obvious isomorphism that takes the elements of $\mathbf{n_1}$ to $1, \ldots, n_1$, the elements of $\mathbf{n_2}$ to $n_1+1, \ldots, n_1+n_2$, etc.

Lie algebras are the same thing as algebras over the Lie operad. There is a well-known formula for a free algebra over an operad. In the case of Lie algebras it says that there is a natural isomorphism
\begin{equation}\label{equation: operadic}
\epsilon\colon\bigoplus_{n=1}^\infty \Lierep_n \otimes_{\integers[\Sigma_n]} V^{\otimes n} \to \Lie[V].
\end{equation}
See~\cite{MSS}, Propositions 1.25 and 1.27. The isomorphism $\epsilon$ is described explicitly in terms of monomials as follows. Let $w(x_1, \ldots, x_n)$ be a parenthesized multilinear monomial in $x_1, \ldots x_n$. The group $V^{\otimes n}$ is generated by elements of the form $v_1\otimes \cdots\otimes v_n$. The homomorphism $\epsilon$ sends $w\otimes v_1\otimes\cdots\otimes v_n$ to $w(v_1, \ldots, v_n)$. 

Note that $\epsilon$ preserves the natural grading. Thus $\epsilon$ restricts to an isomorphism of $\Lierep_n \otimes_{\integers[\Sigma_n]} V^{\otimes n}$ with the degree $n$ part of $\Lie[V]$.

The isomorphism $\epsilon$, together with the algebraic Hilton-Milnor theorem, yield a branching rule for the restriction of $\Lierep_n$ to a Young subgroup of $\Sigma_n$. To see this, consider $\Lie[V_1\oplus\cdots\oplus V_k]$, the free Lie algebra on a direct sum of $k$ free abelian groups. Combining the operadic model for the free Lie algebra and the Hilton-Milnor theorem, we obtain a commutative square of isomorphisms
{\footnotesize
\begin{equation}\label{equation: Hilton-Milnor plus operadic}
\begin{array}{ccc}
\underset{\underset{d\ge 1}{\underset{w\in B(n_1,..., n_k)}{n_1, \ldots, n_k\ge 0}}}{\bigoplus} \Lierep_d \otimes_{\Sigma_d}(V_1^{\otimes n_1}\otimes ...\otimes V_k^{\otimes n_k})^{\otimes d} & \to & \underset{n\ge 1}{\bigoplus}\Lierep_n\otimes_{\Sigma_n}(V_1 \oplus ...\oplus V_k)^{\otimes n}\\
\downarrow & & \downarrow \\
\\
\underset{\underset{w\in B(n_1,\ldots, n_k)}{n_1, \ldots, n_k\ge 0}}{\bigoplus}\Lie[V_1^{\otimes n_1}\otimes \cdots\otimes V_k^{\otimes n_k}] & \to & \Lie[V_1 \oplus \cdots\oplus V_k]
\end{array}
\end{equation}
}
Here the top map is defined to be the unique isomorphism that makes the diagram commute. An easy diagram chase shows that this isomorphism preserves multi-degree in the variables $V_1, \ldots, V_k$. This means that both the source and the target of the top map are direct sums of terms of the form $A\otimes_{\Sigma_{n_1}\times\cdots\times\Sigma_{n_k}} V_1^{\otimes n_1}\otimes\cdots V_k^{\otimes n_k}$, where $n_1, \ldots, n_k$ are natural numbers and $A$ is a (free abelian) group with an action of $\Sigma_{n_1}\times\cdots \times \Sigma_{n_k}$. Fixing a $k$-tuple $(n_1, \ldots, n_k)$ and restricting to terms of multi-degree $(n_1, \ldots, n_k)$, we find that the top isomorphism restricts to an isomorphism (natural in the variables $V_1, \ldots, V_k$) of the following form:
\begin{multline*}
\underset{\underset{w\in B(\frac{n_1}{d},..., \frac{n_k}{d})}{d|\gcd(n_1, \ldots, n_k)}}{\bigoplus} \Lierep_d \otimes_{\integers[\Sigma_d]}(V_1^{\otimes \frac{n_1}{d}}\otimes ...\otimes V_k^{\otimes \frac{n_k}{d}})^{\otimes d} \to \\ \to \Lierep_n\otimes_{\integers[\Sigma_{n_1}\times\cdots\times \Sigma_{n_k}]} (V_1^{\otimes n_1}\otimes ...\otimes V_k^{\otimes n_k}).
\end{multline*}
Recall that for each $d$ that divides $n_1, \ldots, n_k$, we view $\Sigma_d$ as a subgroup of $\Sigma_{n_1}\times\cdots\times\Sigma_{n_k}$ via the diagonal inclusion. It follows that the source of this isomorphism can be rewritten as follows
\[
\underset{\underset{w\in B(\frac{n_1}{d},..., \frac{n_k}{d})}{d|\gcd(n_1, \ldots, n_k)}}{\bigoplus} \Lierep_d \otimes_{\integers[\Sigma_d]}\integers[\Sigma_{n_1}\times\cdots\times \Sigma_{n_k}]\otimes_{\integers[\Sigma_{n_1}\times\cdots\times \Sigma_{n_k}]}(V_1^{\otimes n_1}\otimes ...\otimes V_k^{\otimes n_k}).
\]
Any natural isomorphism between this functor and the functor $$ \Lierep_n\otimes_{\integers[\Sigma_{n_1}\times\cdots\times \Sigma_{n_k}]} (V_1^{\otimes n_1}\otimes ...\otimes V_k^{\otimes n_k})$$ is induced by an isomorphism of $\integers[\Sigma_{n_1}\times\cdots\times \Sigma_{n_k}]$-modules
\begin{equation}\label{equation: branching for lie}
\underset{\underset{w\in B(\frac{n_1}{d},..., \frac{n_k}{d})}{d|\gcd(n_1, \ldots, n_k)}}{\bigoplus} \Lierep_d \otimes_{\integers[\Sigma_d]}\integers[\Sigma_{n_1}\times\cdots\times \Sigma_{n_k}]\stackrel{\cong}{\to} \Lierep_n.
\end{equation}
(One way to see this is  to take $(n_1, \ldots,n_k)$-cross-effects of these functors of $V_1, \ldots, V_k$).

Recall that the isomorphism is induced by diagram~\eqref{equation: Hilton-Milnor plus operadic}. Our next goal is to describe the homorphism~\eqref{equation: branching for lie} explicitly. To do this it is enough to describe, for each $d$ that divides $n_1, \ldots, n_k$, a $\Sigma_d$-equivariant homomorphism
\begin{equation} \label{equation: lie diagonal}
\bigoplus_{B(\frac{n_1}{d},\ldots, \frac{n_k}{d})} \Lierep_d \to \Lierep_n.
\end{equation}
Let $w\in B(n_1, \ldots, n_k)$. In Definition~\ref{definition: resolution} we defined the resolution of $w$, denoted $\tilde w$, to be a certain multi-linear monomial in $n$ variables. 
%
%
In particular, resolution defines a function $$B(n_1, \ldots, n_k)\to \Lierep_n.$$ We emphasize that $B(n_1, \ldots, n_k)$ is just a set, so this is just a function between sets. Sometimes it is convenient to think of it as a homomorphism $\integers[B(n_1, \ldots, n_k)]\to \Lierep_n$. In a similar way, we have a homomorphism $\integers[B(\frac{n_1}{d},\ldots, \frac{n_k}{d})]\to \Lierep_{\frac{n}{d}}$. 
Now consider the following sequence of homomorphisms
\begin{multline*}
\Lierep_d\otimes \integers\left[B\left(\frac{n_1}{d},\ldots, \frac{n_k}{d}\right)\right] \to \Lierep_d\otimes \integers\left[B\left(\frac{n_1}{d},\ldots, \frac{n_k}{d}\right)^d\right] \stackrel{\cong}{\to} \\ \to  \Lierep_d\otimes \integers\left[B\left(\frac{n_1}{d},\ldots, \frac{n_k}{d}\right)\right]^{\otimes d} \to \Lierep_d  \otimes \Lierep_{\frac{n}{d}}^{\otimes d} \to \Lierep_n.
\end{multline*}
Here the first homomorphism is induced by the diagonal $B\left(\frac{n_1}{d},\ldots, \frac{n_k}{d}\right)\to B\left(\frac{n_1}{d},\ldots, \frac{n_k}{d}\right)^d$, the second is the obvious isomorphism, the third is induced by the map $B(\frac{n_1}{d},\ldots, \frac{n_k}{d})\to \Lierep_{\frac{n}{d}}$ that we defined earlier, and the last one uses the operad structure on $\{\Lierep_n\}$. Since the source of this homomorphism is canonically isomorphic to $\bigoplus_{B(\frac{n_1}{d},\ldots, \frac{n_k}{d})} \Lierep_d $, we have constructed a homomorphism as promised in~\eqref{equation: lie diagonal}. The homomorphism is clearly $\Sigma_d$-equivariant.

In words, the homomorphism~\eqref{equation: lie diagonal} can be described as follows. Let $w\in B\left(\frac{n_1}{d},\ldots, \frac{n_k}{d}\right)$. In particular, $w$ is a monomial of multi-degree $(\frac{n_1}{d}, \ldots, \frac{n_k}{d})$. Starting with $w$ define the monomial $\tilde w$ in $\frac{n}{d}$ distinct variables by replacing the $\frac{n_1}{d}$ occurrences of $x_1$ with $x_1, \ldots, x_{\frac{n_1}{d}}$ in the same order as they occur in $w$, replacing the $\frac{n_2}{d}$ occurrences of $x_2$  with $x_{\frac{n_1}{d}+1}, \ldots, x_{\frac{n_1}{d}+\frac{n_2}{d}}$, etc. Then define a homomorphism $\Lierep_d\to \Lierep_n$ in the following way: given a monomial $w_d(x_1, \ldots, x_d)$, send it 
to the monomial 

\begin{equation}\label{equation: explicit formula}
w_d(\tilde w(x_1, x_{1+d},..., x_{1+n-d}), \tilde w(x_2, x_{2+d},..., x_{2+n-d}), \ldots, \tilde w(x_d, x_{2d},..., x_{n}))
\end{equation}
\begin{example}
Let $n_1=n_2=4$, $d=2$. Then~\eqref{equation: lie diagonal} specializes to a homomorphism
$\bigoplus_{B(2, 2)} \Lierep_2 \to \Lierep_8.$ The set $B(2,2)$ has one element: the basic monomial $w=[[[x_2, x_1], x_1], x_2]$. Then $\tilde w$ is obtained from $w$ by replacing the two occurrences of $x_1$ with $x_1$ and $x_2$, and the two occurrences of $x_2$ with $x_3$ and $x_4$. Thus $\tilde w(x_1, x_2, x_3, x_4)=[[[x_3, x_1], x_2], x_4]$. Now the homomorphism $\Lierep_2\to \Lierep_8$ is defined follows. Suppose $w_2(x_1, x_2)$ is a monomial in $x_1, x_2$, representing an element of $\Lierep_2$. then $w_2$ is sent to the following monomial in $8$ variables: 
$$w_2(\tilde w(x_1, x_3, x_5, x_7), \tilde w(x_2, x_4, x_6, x_8)).$$
In particular our homomorphism does the following 
\[ [x_2, x_1]\mapsto [[[[x_6, x_2], x_4], x_8],[[[x_5, x_1], x_3], x_7]]\] 
\[
[x_1, x_2]\mapsto [[[[x_5, x_1], x_3], x_7], [[[x_6, x_2], x_4], x_8]].\] 
Clearly, the homomorphism is $\Sigma_2$-equivariant (remember that $\Sigma_2$ is embedded in $\Sigma_8$ as the subgroup generated by the cycle $(1,2)(3,4)(5,6)(7,8)$). 
\end{example}
The homomorphism~\eqref{equation: lie diagonal} is $\Sigma_d$ equivariant. Since $\Sigma_d$ is a subgroup of $\Sigma_{n_1}\times\cdots\times\Sigma_{n_k}$, we can induce it up to a $\Sigma_{n_1}\times\cdots\times\Sigma_{n_k}$-equivariant homomorphism. Taking direct sum over all positive $d$ that divide $n_1, \ldots, n_k$, we obtain the homomorphism promised in~\eqref{equation: branching for lie}.

A routine diagram chase, using our explicit descriptions of the Hilton-Milnor isomorphism and the isomorphism~\eqref{equation: operadic}, establishes the following proposition.


\begin{proposition}\label{proposition: algebraic branching}
The homomorphism~\eqref{equation: branching for lie} that we just described is the isomorphism induced by the top map in diagram~\eqref{equation: Hilton-Milnor plus operadic}. 
\end{proposition}

\section{From the Lie representation to the partition complex}\label{section: lie to partitions}
It is well-known that there is a close connection between the Lie representation and the (co)homology of the partition complex. More precisely, there is an isomorphism of integral $\Sigma_n$ representations \[\Lierep_n\cong {\widetilde \HH}^{n-3}(\Pi_n)^{\pm}.\]
The superscript $\pm$ indicates tensoring with the sign representation. This isomorphism was first proved by H. Barcelo~\cite{barcelo}. Our treatment of this subject was influenced by the paper of Wachs~\cite{wachs}, in which Barcelo's isomorphism was generalized, and perhaps simplified.

We will now describe Barcelo's isomorphism, up to a sign indeterminacy. The nerve of $\Pi_n$ is an $n-3$-dimensional simplicial complex. Let us identify top-dimensional simplices of $|\Pi_n|$ with increasing chain of partitions of $\n$ of the form $\lambda_1<\lambda_1<\cdots<\lambda_{n-2}$. Since this is a maximal sequence of partitions, each $\lambda_{i+1}$ is obtained from $\lambda_i$ by dividing one of the equivalence classes of $\lambda_i$ in two parts. For $i=1, \ldots, n-2$, $\lambda_i$ has $i+1$ equivalence classes.

We work with simplicial homology and cohomology. A top-dimensional simplex of $|\Pi_n|$ represents an element of $\HH^{n-3}(|\Pi_n|)$ via the cochain (which is automatically a cocycle) that takes value $1$ on this simplex and is zero otherwise.

Later in the paper we will need the following simple lemma.
\begin{lemma}\label{lemma: nonzero}
Every top-dimensional simplex of $|\Pi_n|$ represents a non-divisible (in particular non-zero) element of $\HH^{n-3}(|\Pi_n|)$.
\end{lemma}
\begin{proof}
Let $\Lambda$ be a top-dimensional simplex of $|\Pi_n|$, and let $c_{\Lambda}$ be the cochain represented by $\Lambda$. It is enough to show that there exists an $n-3$-dimensional simplicial cycle $\sigma$ of $|\Pi_n|$ such that $\langle c_\Lambda, \sigma\rangle=\pm 1$. Such a cycle $\sigma$ is an integral linear combination of top-dimensional simplices, in which $\Lambda$ appears with coefficient $\pm 1$.

We will now describe a way to construct top-dimensional cycles of $|\Pi_n|$. We learned it from the paper of Wachs~\cite{wachs}. Let $T$ be a tree with vertex set $\{1, \ldots, n\}$. Let $E$ be the set of edges of $T$. This is a set of size $n-1$. A subset $U\subset E$ determines a partition of $\{1, \ldots, n\}$ - the connected components of the subgraph of $T$  with vertex set $\{1, \ldots, n\}$ and edge set $U$. The empty set and $E$ correspond to the discrete and the indiscrete partition respectively. Let $\overline{\Bcal}_E$  the poset of proper non-empty subsets of $E$. We have an inclusion of posets from the opposite of $\overline{\Bcal}_E$ into $\Pi_n$. We saw in Example~\ref{example: sphere} that the geometric realization of $\overline{\Bcal}_E$, and therefore also of its opposite, is homeomorphic to $S^{n-3}$. Passing to geometric realizations, we obtain a map $S^{n-3}\to |\Pi_n|$. Choose a fundamental simplicial cycle in dimension $n-3$ of $|\overline{\Bcal}_E|$. Its a linear combination of maximal simplices of $|\overline{\Bcal}_E|$, where each simplex appears with coefficient $\pm 1$. It determines a cycle of $|\Pi_n|$. Thus we associated with every tree $T$ with vertex set $\{1, \ldots, n\}$ a simplicial cycle of $\Pi_n$. By construction, this cycle is a linear combination of maximal simplices of $|\Pi_n|$, where each simplex occurs with coefficient $\pm 1$ or $0$.

It remains to show that for every maximal simplex $\Lambda$ of $|\Pi_n|$, one can find a tree $T_\Lambda$, such that $\Lambda$ occurs in the cycle associated with $T_\Lambda$ with a non-zero coefficient. Let $\lambda_1<\lambda_2 <\cdots <\lambda_{n-2}$ be the chain of partitions comprising $\Lambda$. Let us augment $\Lambda$ by adding the indiscrete and the discrete partition, to obtain the chain $\hat0=\lambda_0<\lambda_1<\lambda_2 <\cdots <\lambda_{n-2}<\lambda_{n-1}=\hat 1$. Each $\lambda_i$ has $i+1$ equivalence classes, and each $\lambda_{i+1}$ is obtained from $\lambda_i$ by breaking up one equivalence class in two. Define the tree $T_\Lambda$ as follows: the vertex set is $\{1, \ldots, n\}$. For each $i$, let $S_i$ be the equivalence class of $\lambda_i$ that is broken in two parts in $\lambda_{i+1}$. Let the two parts of $S_i$ by $U_i$ and $U_i'$. Then for each $i=0, \ldots, n-2$, let $e_{i+1}$ be the edge connecting the smallest element of $U_i$ with the smallest element of $U_i'$ (in fact, one can connect any element of $U_i$ with any element of $U_i'$). Let $T_\Lambda$ be the graph with vertex set $\{1,\ldots, n\}$ and edge set $e_1, \ldots, e_{n-1}$. Thus $T_\Lambda$ has $n-1$ edges. We claim that $T_\Lambda$ is connected. We want to prove that any two vertices of $T_\Lambda$ are connected by a path. We will prove that if $s$ and $t$ are in the same class of $\lambda_i$ for some $i$ then they are connected by a path. We argue by downward induction on $i$. For $i=n-1$, $\lambda_i$ is the discrete partition, so it is obvious. Suppose it is true for $i+1$. Suppose $s$ and $t$ are in the same class of $\lambda_i$. If they are in the same class of $\lambda_{i+1}$ then they are connected by a path, by induction assumption. Suppose they are not in the same class of $\lambda_{i+1}$. Then necessarily one of them is in $U_i$ and the other in $U_i'$. By construction, there is an edge connecting an element of $U_i$ with an element of $U_i'$. Thus there is a path connecting $s$ and $t$. It follows that if $s$ and $t$ are in the same component of $\lambda_0$ then they are in the same path. But $\lambda_0$ is the indiscrete partition so any two vertices of $T_\Lambda$ are connected by a path. Since $T_\Lambda$ is a connected graph with $n$ vertices and $n-1$ edges, it must be a tree.

It remains to show that $\Lambda$ occurs in the cycle associated with $T_\Lambda$. This is equivalent to showing that $\Lambda$ occurs as a chain of partitions associated with a decreasing chain of subgraphs of $T_\Lambda$. This is obvious: just take the sequence of graphs $T_{\Lambda}\setminus\{e_1\}, T_{\Lambda}\setminus\{e_1, e_2\}, \dots, T_\Lambda\setminus\{e_1, \ldots, e_{n-2}\}$. An easy induction argument shows that the partition associated with $T_{\Lambda}\setminus\{e_1, \ldots, e_j\}$ is $\lambda_j$, and we are done.
\end{proof}

We want to describe a map (at least up to sign) $\Lierep_n\to \widetilde\HH^{n-3}(\Pi_n)$. As usual, we represent elements of $\Lierep_n$ by parenthesized monomials in $x_1, \ldots, x_n$. 
There is a well-known bijection between such monomials and planar binary trees. More precisely, a planar binary tree is rooted tree, where every node has either zero or two children, called the left and right child. Nodes with zero children are called leaves, and nodes with two children are called internal. A labeling of a binary tree (by a set $S$) is a bijection between its set of leaves and $S$. 

%
There is a bijection between the set of parenthesized multilinear monomials in a given non-empty finite set $S$, and the set of (isomorphism classes of) planar binary trees labeled with $S$. The bijection is defined as follows. If $S$ has one element, say $x$, the tree associated with the monomial $x$ is the tree with one vertex, labeled with $x$. Now suppose that the bijection was defined for all $S$ with fewer than $n$ elements. Let $w$ be a monomial in $n$ variables. Then necessarily $w=[u, v]$, where $u$ and $v$ are monomials in fewer than $n$ variables. By hypothesis, we have associated to $u$ and $v$ binary trees $T_u$ and $T_v$. Let $T_w$ be the unique binary tree whose left and right subtrees are $T_u$ and $T_v$ respectively. This defines the bijection.

Let $w$ be a monomial on $n$ letters, as above, and let $T_w$ the labeled binary tree associated with $w$. Let us call the labels $\{1, \ldots, n\}$. The set of internal nodes of $T_w$ has a natural partial ordering: we say that $s<t$ if $t$ is a descendant of $s$. 
\begin{definition}
A {\it linearization} of a binary tree is a linear order on the set of its internal nodes that is compatible with the natural order. 
\end{definition}
\begin{example}
The {\it preorder} is the linear order of internal nodes of a binary tree that is defined recursively as follows. Assuming the tree has more than one node, list the root first, then list the internal nodes of the right subtree in preorder, then list the internal nodes of the left subtree in preorder.
\end{example}

To a linearization of $T_w$ we associate a maximal chain of partitions of $\n$ as follows. Let $s_1, \ldots, s_{n-1}$ be the internal nodes of $T_w$, listed in the linear order (a binary tree with $n$ leaves has $n-1$ internal nodes). For $1\le j<n-1$, let $T_w\setminus \{s_1, \ldots, s_j\}$ be the graph obtained from $T_w$ by removing the nodes $s_1, \ldots, s_j$ and the edges adjacent to them. This graph is a forest that contains the leaves of $T_w$, which are $\{1, 2, \ldots, n\}$. Connected components $T_w\setminus \{s_1, \ldots, s_j\}$ define a partition of $\n$. Let us denote this partition $\lambda_j$. Clearly $\hat 0\le \lambda_1\le \cdots \le \lambda_{n-1}=\hat 1$. We claim that in fact these are strict inequalities, and moreover each $\lambda_{j+1}$ is obtained from $\lambda_j$ by braking a single equivalence class in two parts. Since the linear order is compatible with the natural order, $s_1$ must be the root of $T_w$. It follows that $T_w\setminus \{s_1\}$ has two connected components: the left subtree and the right subtree (we may as well assume that $n>1$). It follows that $\lambda_1$ has two equivalence classes. For each $j\ge 1$, the parent of $s_{j+1}$ occurs earlier than $s_{j+1}$ in the linear order, while descendants of $s_{j+1}$ occur later in the linear order. It follows that the connected component of $s_{j+1}$ in $T_w\setminus\{s_1, \dots, s_j\}$ is a binary tree with root $s_j$. It follows that $T_w\setminus\{s_1, \dots, s_{j+1}\}$ is obtained from $T_w\setminus\{s_1, \dots, s_j\}$ by braking up the connected component of $s_{j+1}$ into the left subtree and the right subtree. It follows that $\lambda_{j+1}$ is obtained from $\lambda_j$ by breaking a single class in two parts. It follows that we have associated with a linearization of $T_w$ a maximal simplex of $\Pi_n$, given by the chain $\lambda_1<\cdots < \lambda_{n-2}$. 

\begin{remark}
It is easy to show that the procedure that we just described has a kind of converse. For every maximal chain of $\Pi_n$ there exists a labeled planar binary tree, together with a linearization, that induces the given maximal chain. The binary tree is unique except for the planar structure.
\end{remark}

\begin{lemma}\label{lemma: order}
Let $w$ be a  monomial in $x_1, \ldots, x_n$, as before. The maximal chains of $\Pi_n$ obtained from different choices of linearization of the binary tree $T_w$ all project to the same element of $\HH^{n-3}(\Pi_n)$, up to sign.
\end{lemma}
\begin{proof}
This follows from Lemma 5.2 in the paper of Wachs~\cite{wachs}. Our statement is simpler than [loc. cit.], because Wachs identifies the sign precisely, while we don't need to. We remark that there are slight differences of notation and convention between our presentation and Wachs's. For example, we list chains of partitions from coarsest to finest, rather than the other way around. Because of this, the chain of partition that we get from ``preorder'' corresponds to the chain that Wachs gets out of a ``postorder''. Notwithstanding these unimportant differences, it should be clear that what we call ``different linearizations of $T_w$'' corresponds exactly to what Wachs calls ``linear extensions of the postorder of the internal nodes of $T$''.
\end{proof}
From now on, choose a linearization of every binary planar tree labeled by $\{x_1, \ldots, x_n\}$. Suppose that $w$ is a multilinear monomial in $x_1, \ldots, x_n$. Let $T_w$ be the planar binary tree obtained from $w$. Let $\Lambda_w$ be the maximal chain of $\Pi_n$ that we obtained from the tree $T_w$ using our chosen linearization. Let $[\Lambda_w]\in \HH^{n-3}(\Pi_n)$ be the cohomology class corresponding to $\Lambda_w$. By lemma~\ref{lemma: order}, the cohomology class $[\Lambda_w]$ is independent of the choice of linearization of $T_w$, up to sign. A monomial $w$ represents an element of $\Lierep_n$, but the assignment $w\mapsto [\Lambda_w]$ does not quite induce a homomorphism $\Lierep_n\to \HH^{n-3}(\Pi_n)^{\pm}$, because of the sign indeterminacy. But there is an isomorphism between these modules that is compatible with this assignment. This is the content of the following theorem.
\begin{theorem}\label{theorem: barcelo}
For all $n\ge 1$ there is an isomorphism of $\Sigma_n$-modules
\[
\epsilon\colon \Lierep_n\to \widetilde\HH^{n-3}(\Pi_n)^{\pm}
\]
At least for $n\ge 3$, the isomorphism sends every element of $\Lierep_n$ represented by a monomial $w$ either to $[\Lambda_w]$ or to $-[\Lambda_w]$. 
\end{theorem}
\begin{proof}
This is a theorem of Barcelo. It follows from~\cite[Theorem 5.4]{wachs}.
\end{proof}

\section{The collapse map}\label{section: collapse}
Throughout this section, fix positive integers $n=n_1+\cdots +n_k$. Let $\Sigma_{n_1}\times\cdots\times\Sigma_{n_k}\subset \Sigma_n$ be a fixed Young subgroup. In this section we construct the $\Sigma_{n_1}\times\cdots\times\Sigma_{n_k}$-equivariant map that appears in the statement of the main theorem
\begin{equation}\label{equation: collapse}
|\Pi_n| \longrightarrow \bigvee_{d|g} \left(\bigvee_{B(\frac{n_1}{d}, \ldots, \frac{n_k}{d})} \Sigma_{n_1}\times \cdots \times {\Sigma_{n_k}}_+ \wedge_{\Sigma_d} S^{n-d-1} \wedge |\Pi_d|^\diamond\right).
\end{equation}
This map will be constructed as a collapse map associated with an embedding of a codimension zero subcomplex into $\Pi_n$. Fix a positive $d$ that divides $n_1, \ldots, n_k$ and a monomial $w\in B(\frac{n_1}{d}, \ldots, \frac{n_k}{d})$. In Section~\ref{section: lie}, Definition~\ref{definition: resolution} we defined a function from $B(\frac{n_1}{d}, \ldots, \frac{n_k}{d})$ to the set of multilinear monomials in ${\frac{n}{d}}$ variables. In Section~\ref{section: lie to partitions} we showed how to associate with such a monomial a maximal chain of $\Pi_{\frac{n}{d}}$. Starting with $w\in B(\frac{n_1}{d}, \ldots, \frac{n_k}{d})$, let $\lambda_1<\cdots < \lambda_{\frac{n}{d}-2}$ be the associated maximal chain of $\Pi_{\frac{n}{d}}$. Let us add the indiscrete partition to this chain, to obtain the extended chain $\hat0=\lambda_0<\lambda_1<\cdots < \lambda_{\frac{n}{d}-2}$. Let us denote this extended chain $\Lambda_w$. Note that since $\Lambda_w$ is a maximal chain, it is, in particular a binary chain (Definition~\ref{definition: binary}).

Since $d$ divides $n$, we may associate with it a partition with $d$ equivalent classes, each of size $\frac{n}{d}$. Let us give a formal definition.
\begin{definition}
Suppose $0<d|n$. Then $\rho_d$ is the partition of $\n$ whose equivalence classes are
\begin{multline*}
\left\{1, 1+ d, 1 + 2d, \ldots, 1+\left(\frac{n}{d}-1\right)d\right\}, \\ \left\{2, 2+ d, 2 + 2d, \ldots, 2+  \left(\frac{n}{d}-1\right)d\right\}, \ldots \\\ldots, \{d, 2d, \ldots, n\}.
\end{multline*}
In particular, $\rho_d$ has $d$ equivalence classes, each canonically isomorphic to $\mathbf{\frac{n}{d}}$, via the order-preserving map.
\end{definition}
Suppose $\lambda$ is a partition of $\mathbf{\frac{n}{d}}$. Define $d\times \lambda$ to be the partition of $\n$ that agrees with $\lambda$ on each equivalence class of $\rho_d$, using the canonical isomorphism of the equivalence class with $\mathbf{\frac{n}{d}}$. In particular, if $\lambda_0$ is the indiscrete partition of $\mathbf{\frac{n}{d}}$ then $d\times \lambda_0=\rho_d$.

Applying this operation to the chain of partitions $\hat0=\lambda_0<\lambda_1<\cdots < \lambda_{\frac{n}{d}-2}$ of $\mathbf{\frac{n}{d}}$ that we defined above, we obtain the following chain of partitions of $\n$ $$d\times \Lambda_w=(\rho_d <d\times \lambda_1 <\cdots < d\times \lambda_{\frac{n}{d}-2}).$$
Note that $d\times\Lambda_w$ is still a binary chain of partitions. Note that in the case $d=1$, $\rho_d$ is the indiscrete partition, and therefore is not in $\Pi_n$. 
\begin{definition}
Let $\Pi_w$ be the star (in $\Pi_n$) of the chain $d\times \Lambda_w$ (for $d=1$, $\Pi_w$ is really just the  chain $\lambda_1<\cdots < \lambda_{n-2}$)
\end{definition}
Since $d\times \Lambda_w$ is a binary chain, whose initial term is $\rho_d$, which is a partition with $d$ equivalence classes, the following lemma is an immediate consequence of isomorphism~\eqref{equation: binary star} and Proposition~\ref{proposition: boundary}.
\begin{lemma}\label{lemma: homeomorphism}
There is a natural homeomorphism
\[
|\Pi_w| \cong  \Delta^{n-d-1}*|\Pi_d|.  
\]
Moreover, the boundary of $|\Pi_w|$ in $|\Pi_n|$ is $S^{n-d-2}*|\Pi_d|$, where $S^{n-d-2}$ is the topological boundary of $\Delta^{n-d-1}$. 
\end{lemma}
Moreover, the homeomorphism is $\Sigma_d$-equivariant, where $\Sigma_d$ acts on $|\Pi_w|$, $\Delta^{n-d-1}$ and $|\Pi_d|$ in a natural way.
\begin{remark}
In the case $d=1$ the lemma says that $|\Pi_w|\cong \Delta^{n-3}$. The reader may want to check this case ``by hand''.
\end{remark}

 The embedding $|\Pi_w|\hookrightarrow |\Pi_n|$ induces a ``collapse'' map 
 \[
 \Pi_n \to |\Pi_w|/\partial|\Pi_w| \stackrel{\cong}{\to} S^{n-d-1} * |\Pi_d|.\]
 Let us digress briefly to review the generalities of this construction. 
 
 Let $X$ be a topological space, and let $A_1, \ldots, A_m\subset X$ be closed subsets whose interiors are pairwise disjoint. Let $A$  be the union of the $A_i$s. There is a sequence of maps
\begin{equation}\label{equation: general collapse}
X \to X/\overline{X\setminus A} \stackrel{\cong}{\leftarrow} A / \partial A \to A/\cup_i \partial A_i \stackrel{\cong}{\leftarrow} \bigvee_{i=1}^m A_i/\partial A_i.
\end{equation}
Here the first map is the quotient map. The second map is induced by the map of pairs $(A, \partial A) \to (X, \overline{X\setminus A})$. The third map is induced by the inclusion $\partial A \hookrightarrow \cup_i \partial A_i$. The last map is induced by the map from the disjoint union to the actual union inside $X$
\[
\coprod_i A_i \to \bigcup_i A_i.
\]
It is easy to see that under the assumption that the $A_i$s only intersect at boundary points, the last map is a homeomorphism. Since  the ``wrong way'' maps in~\eqref{equation: general collapse} are homeomorphisms, we obtain a well defined map $X \to  \bigvee_{i=1}^m A_i/\partial A_i$. We refer to it as the collapse map associated with the inclusions $A_i\hookrightarrow X$.

To go back to the relevant example, the embedding $|\Pi_w|\hookrightarrow |\Pi_n|$ induces a $\Sigma_d$-equivariant quotient map
 \begin{multline*}
q_w\colon  \Pi_n \to |\Pi_w|/\partial|\Pi_w| \stackrel{\cong}{\to} \\ \stackrel{\cong}{\to}(\Delta^{n-d-1} * |\Pi_d|)/ (S^{n-d-2} * |\Pi_d|)\stackrel{\simeq}{\to}S^{n-d-1} * |\Pi_d|.
 \end{multline*}
Our next task is to compare the induced map in cohomology with the homomorphism $\phi_w\colon\Lierep_d\to \Lierep_n$ that we constructed for each $w\in B(\frac{n_1}{d}, \ldots, \frac{n_k}{d})$ ($\phi_w$ is a specialization of~\eqref{equation: lie diagonal}). Recall that there is an isomorphism (for each $n$) of $\Sigma_n$-modules $\epsilon_n\colon \Lierep_n \to \HH^{n-3}(\Pi_n)^{\pm}$ (Theorem~\ref{theorem: barcelo}). 

There are two natural suspension isomorphisms $\HH^{d-3}(|\Pi_d|) \to \HH^{n-3}(S^{n-d-1} * |\Pi_d|)$. Choose one and call it $\sigma$. 

Consider the following (not necessarily commutative) diagram 
\[
\xymatrix{
\Lierep_d \ar[rr]^{\phi_w} \ar[d]^{\epsilon_d}  & &\Lierep_n  \ar[d]^{\epsilon_n}\\
\widetilde\HH^{d-3}(|\Pi_d|)  \ar[r]^-\sigma &  \widetilde\HH^{n-3}(S^{n-d-1} * |\Pi_d|) \ar[r]^-{q_w^*} & \widetilde\HH^{n-3}(|\Pi_n|)
}
\]

\begin{lemma}\label{lemma: commutes}
Let $w_d$ be a parenthesized multilinear monomial in $x_1, \ldots, x_d$. It represents an element of $\Lierep_d$. The images of this element in $\widetilde\HH^{n-3}(|\Pi_n|)$ under the two maps around the diagram are the same up to sign.
\end{lemma}
\begin{proof}
In what follows, suppose that $d\ge 3$. The argument needs some minor adjustments for $d=1, 2$. We will do these cases at the end of the proof.

Consider first the image of $w_d$ in $\Lierep_n$ under $\phi_w$. By~\eqref{equation: explicit formula}, it is represented by the monomial $$w_d(\tilde w(x_1, x_{1+d},..., x_{1+n-d}), \tilde w(x_2, x_{2+d},..., x_{2+n-d}), \ldots, \tilde w(x_d, x_{2d},..., x_{n})),$$ where $\tilde w$ is the resolution of $w$ (Definition~\ref{definition: resolution}). The image of this monomial in $\HH^{n-3}(|\Pi_n|)$ is represented up to sign (throughout the proof, ``represented'' means ``represented up to sign'') by a maximal simplex associated with the planar binary tree associated with this monomial. Given a monomial $w$, let $T_w$ be the associated binary tree. It is easy to see that the binary tree for our monomial is obtained by grafting a copy of $T_{\tilde w}$ onto each leaf of $T_{w_d}$, and then labeling the leaves of the $i$-th copy of $T_{\tilde w}$ by $i, i + d, i + 2d, \ldots$. Let us call this tree $T_{w_d\circ \tilde w}$.

Going the other way around the diagram, the image of $w_d$ in $\widetilde\HH^{d-3}(|\Pi_d|) $ is represented by a maximal simplex associated with the planar binary tree $T_{w_d}$. Let's denote this maximal simplex $\Lambda_{w_d}$. The image of $[\Lambda_{w_d}]$ in $\widetilde\HH^{n-3}(S^{n-d-1} * |\Pi_d|)$ is represented by a maximal simplex that is the join of $\Lambda_{w_d}$ with a maximal simplex of $S^{n-d-1}$. The image in $ \widetilde\HH^{n-3}(|\Pi_n|)$ is then represented by a maximal simplex that is the join of $\Lambda_{w_d}$, which is now considered a maximal simplex of $(\Pi_w)_{<\rho_s}\cong \Pi_d$ and a maximal simplex of $(\Pi_w)_{\ge \rho_d}$. It is easy to see that any such join of maximal simplices is obtained from some linearization of the tree $T_{w_d\circ \tilde w}$. A linearization in which the interval vertices of $T_{w_d}$ precede the internal vertices of the copies of $\tilde w$. It follows that the two images of $w_d$ in $\Lierep_n$ are represented by two linearizations of the same planar binary tree $T_{w_d\circ \tilde w}$. It follows by Lemma~\ref{lemma: order} that they are the same up to sign.

In the case $d=1$ or $2$, observe that the image of a generator of $\Lierep_d$ in $\widetilde\HH^{n-3}(S^{n-d-1} * |\Pi_d|)$ is still represented by a maximal simplex of $\star(\Lambda_w)$. The rest of the proof remains the same as before.
\end{proof}
Now let's go back to the situation where $n=n_1+\cdots+n_k$. For every positive $d$ that divides $n_1, \ldots, n_k$ we have a $\Sigma_d$-equivariant homomorphism, induced by the collapse maps defined above.
\[
\bigoplus_{B(\frac{n_1}{d}, \ldots, \frac{n_k}{d})}  \HH^{n-3}(S^{n-d-1}*|\Pi_d|)\to \HH^{n-3}(|\Pi_n|).
\]
This homomorphism can be induced up to a $\Sigma_{n_1}\times \cdots \times \Sigma_{n_k}$-equivariant homomorphism. Taking direct sum over all $d$ that divide $n_1, \ldots, n_k$, one obtains the homomorphism referred to in the following proposition.
\begin{proposition}\label{proposition: collapse isomorphism}
The following homomorphism, induced by the collapse maps that we defined, is an isomorphism
\[
\bigoplus_{\underset{w\in B(\frac{n_1}{d}, \ldots, \frac{n_k}{d})}{ d|\gcd(n_1, \ldots, n_k)}}  \HH^{n-3}(S^{n-d-1}*|\Pi_d|)\otimes_{\Sigma_d} \integers[\Sigma_{n_1}\times\cdots\times \Sigma_{n_k}]\to \HH^{n-3}(|\Pi_n|).
\]
\end{proposition}
\begin{proof}
For typographical reasons, we introduce the following notation. Suppose $M$ is a (right) $\integers[\Sigma_d]$-module. Then $$M\!\uparrow\, :=M\otimes_{\integers[\Sigma_d]} \integers[\Sigma_{n_1}\times\cdots\times \Sigma_{n_k}].$$

Consider the following diagram of $\integers[\Sigma_{n_1}\times\cdots\times \Sigma_{n_k}]$-modules. It is closely related to the diagram preceding Lemma~\ref{lemma:  commutes}. For typographical reasons, we omit the indexing sets under the $\bigoplus$ signs. All the direct sums are taken over all positive $d|\gcd(n_1, \ldots, n_k)$, and for each such $d$, over all $w\in B(\frac{n_1}{d}, \ldots, \frac{n_k}{d})$
\[
\xymatrix{
\bigoplus\Lierep_d\!\uparrow \ar[rr]^{\Sigma\phi_w} \ar[d]^{\oplus\epsilon_d}  & &\Lierep_n  \ar[d]^{\epsilon_n}\\
\bigoplus \HH^{d-3}(|\Pi_d|)\!\uparrow  \ar[r]^-\sigma &  \bigoplus \HH^{n-3}(S^{n-d-1} * |\Pi_d|)\!\uparrow \ar[r]^-{\Sigma q_w^*} & \HH^{n-3}(|\Pi_n|)
}
\]
Let $w\in B(\frac{n_1}{d}, \ldots, \frac{n_k}{d})$, $\theta_w$ a monomial representing an element of $\Lierep_d$, and $\sigma\in \Sigma_{n_1}\times\cdots\times\Sigma_{n_k}$. Then $[\theta_w]\otimes \sigma$ represents an element of the top left corner, and elements of this form provide a generating set for the top left corner. The top homomorphism in the diagram is an isomorphism by Theorem~\ref{theorem: hilton-milnor} (the Hilton-Milnor theorem). The vertical homomorphisms are induced by the standard isomorphism from $\Lierep_n$ to the cohomology of the partition poset (Theorem~\ref{theorem: barcelo}). It follows that a generating set for the top left corner will be mapped to a generating set of the bottom right corner via the top and right maps. It follows from Lemma~\ref{lemma: commutes} that elements of the form $[\theta_w\otimes \sigma]$ are mapped to the same element of $\HH^{n-3}(\Pi_n)$ by the top and right maps as by the left and bottom maps, except possibly for some negative signs. Therefore, the set of elements of the form $[\theta_w\otimes \sigma]$ is mapped to a generating set of the bottom right corner via the left and bottom maps. It follows that the composed map on the bottom is an epimorphism. Since it is a homomorphism between free abelian groups of the same rank, it is an isomorphism. Since the left bottom map is clearly an isomorphism, it follows that the right bottom map is an isomorphism as well. This is what we wanted to prove.
\end{proof}
Suppose that $d|\gcd(n_1, \ldots, n_k)$, and $w\in B(\frac{n_1}{d}, \ldots, \frac{n_k}{d})$. Then $\Pi_w$ is a $\Sigma_d$-invariant subposet of $\Pi_n$. Let $\sigma\in \Sigma_d \backslash\Sigma_{n_1}\times...\times \Sigma_{n_k}$ be a coset. Then $\Pi_w \cdot \sigma$ is a well-defined subposet of $\Pi_n$. 
\begin{lemma}\label{lemma: simplices}
Suppose that $d, d'|\gcd(n_1, \ldots, n_k)$, $w\in B(\frac{n_1}{d}, \ldots, \frac{n_k}{d})$, $w'\in B(\frac{n_1}{d'}, \ldots, \frac{n_k}{d'})$, $\sigma\in \Sigma_d \backslash\Sigma_{n_1}\times...\times \Sigma_{n_k}$, $\sigma'\in \Sigma_{d'}\backslash\Sigma_{n_1}\times...\times \Sigma_{n_k}$. Suppose, furthermore, that $(d, w, \sigma)\ne(d', w', \sigma')$. Then the subcomplexes $\Pi_w\cdot\sigma$ and $\Pi_{w'}\cdot\sigma'$ of $\Pi_n$ have disjoint sets of maximal simplices.
\end{lemma}
\begin{proof}
Suppose that $\Delta$   and $\Delta'$ are maximal simplices of $\Pi_w$ and $\Pi_{w'}$ respectively. It follows from Lemma~\ref{lemma: nonzero} that $\Delta$ and $\Delta'$ represent non-zero classes in the groups $\HH^{n-3}(S^{n-d-1}*|\Pi_d|)$ and $\HH^{n-3}(S^{n-d'-1}*|\Pi_{d'}|)$ corresponding to $w$ and $w'$. By Proposition~\ref{proposition: collapse isomorphism}, the following homomorphism is an isomorphism
\[
\bigoplus\HH^{n-3}(S^{n-d-1}*|\Pi_d|)\otimes_{\Sigma_d}\integers[\Sigma_{n_1}\times...\times \Sigma_{n_k}]  \to \HH^{n-3}(\Pi_n).
\]
Since $(d, w, \sigma)\ne(d', w', \sigma')$, $[\Delta]\otimes \sigma$ and $[\Delta']\otimes \sigma'$ represent non-zero elements in two different summands of the source. It follows that $\Delta\cdot \sigma$ and $\Delta'\cdot \sigma'$ represent different elements of $\HH^{n-3}(\Pi_n)$. Therefore, they can not be the same maximal simplex.
\end{proof}
\begin{proposition}\label{proposition: collapse map}
The inclusions $\Pi_w\hookrightarrow \Pi_n$  induce a quotient map
\begin{equation}\label{equation: join version}
|\Pi_n| \longrightarrow \bigvee_{d|g} \left(\bigvee_{B(\frac{n_1}{d}, \ldots, \frac{n_k}{d})} \Sigma_{n_1}\times \cdots \times {\Sigma_{n_k}}_+ \wedge_{\Sigma_d} S^{n-d-1} * |\Pi_d|\right).
\end{equation}
This map is $ \Sigma_{n_1}\times \cdots \times \Sigma_{n_k}$-equivariant, and it induces an isomorphism in cohomology.
\end{proposition}
\begin{remark}
By Lemma~\ref{lemma: join}, $S^{n-d-1} * |\Pi_d|$ is naturally equivalent to $S^{n-d-1} \wedge |\Pi_d|^\diamond $. Therefore, the map~\eqref{equation: join version} determines a map as in~\eqref{equation: collapse}.
\end{remark}
\begin{proof}
Consider all the subcomplexes $\Pi_w\cdot\sigma\subset \Pi_n$, as $w$ ranges over all elements of $B(\frac{n_1}{d}, \ldots, \frac{n_k}{d})$ (for all $d$), and $\sigma$ ranges over all cosets $\Sigma_d \backslash\Sigma_{n_1}\times...\times \Sigma_{n_k}$. By Lemma~\ref{lemma: simplices}, these subcomplexes do not share any maximal simplices. It follows that they only may intersect at boundary points. Therefore we obtain a collapse map, as in~\eqref{equation: general collapse}
\[
|\Pi_n|\to \bigvee_{d|g} \bigvee_{B(\frac{n_1}{d}, \ldots, \frac{n_k}{d})} \Sigma_{n_1}\times \cdots \times {\Sigma_{n_k}}_+ \wedge_{\Sigma_d} |\Pi_w|/\partial|\Pi_w|.
\]
It follows from Lemmas~\ref{lemma: homeomorphism} and~\ref{lemma: boundary} that $ |\Pi_w|/\partial|\Pi_w|\cong S^{n-d-1} * |\Pi_d|$, and we obtain the desired map. The induced homomorphism in cohomology is the one of Proposition~\ref{proposition: collapse isomorphism}, and so it is an isomorphism.
\end{proof}
\begin{example}\label{example: n=4}
Let us describe explicitly the map~\eqref{equation: collapse} or equivalently the map~\eqref{equation: join version} for the case of the group $\Sigma_2\times \Sigma_2$ acting on $|\Pi_4|$. The possible values of $d$ are $1$ and $2$. The set $B(2,2)$ consists of one element: the monomial $w_{2}:=[[[x_2, x_1], x_1], x_2]$. The resolution of $w_{2}$ is the monomial $[[[x_3, x_1], x_2], x_4]$. The associated chain of partitions is $123/4 < 13/2/4$ (in this case the binary tree associated with the monomial has only one linearization, so there is only one possible maximal simplex to associate with $w_2$). Since this is a maximal chain, the star of the chain is the chain itself. So $\Pi_{w_2}$ is the chain $123/4 < 13/2/4$. The poset ${\Sigma_2\times\Sigma_2} \times \Pi_{w_2}$ is the union of the following four chains:
\[
(123/4 < 13/2/4), (123/4 < 23/1/4), (124/3 < 14/2/3), (124/3 < 24/1/3)
\]
   
The set $B(1,1)$ also consists of one element: the monomial $w_{1}:=[x_2, x_1]$. The associated partition $\rho_2$ is $13/24$. The poset $\Pi_{w_1}$ consists of all partitions comparable with this one. It can be described as follows
\[
1/3/24 > 13/24 < 13/2/4.
\]
As expected, $\Pi_{w_1}$ is a $\Sigma_2$-invariant subposet of $\Pi_4$. The poset $\Sigma_2\times \Sigma_2 \times_{\Sigma_2} \Pi_{w_1}$ is the union of the poset above with the poset
\[
2/3/14 > 23/14 < 23/1/4.
\]
The reader is invited to check that the geometric realization of the subcomplex 
\[
|\Sigma_2\times \Sigma_2\times \Pi_{w_2}\coprod \Sigma_2\times \Sigma_2 \times_{\Sigma_2} \Pi_{w_1}| \subset |\Pi_4|
\]
is exactly the subcomplex marked with fat lines on Figure~\ref{figure: n=4}. It follows that the map~\eqref{equation: join version} is, in this case, the same one as defined in that figure. Since the map is defined by collapsing a contractible subcomplex, it is in fact a homotopy equivalence (and not just a cohomology equivalence).
\end{example}

\section{Fixed points}\label{section: fixed points}
Let $G$ be a group acting (on the left) on the set $\n$. Specifying such an action is equivalent to specifying a homomorphism $G\to \Sigma_n$. Sometimes we will need to assume that the action is effective, which means that $G$ injects into $\Sigma_n$. The action of $G$ on $\n$ induces an action on $\Pi_n$. Our goal in this section is to describe the fixed points space $|\Pi_n|^G\cong |\Pi_n^G|$.

The orbits of the action of $G$ partition the set $\n$ into $G$-invariant subsets. Each orbit is isomorphic to a set of the form $G/H$, where $H$ is a subgroup of $G$, determined up to conjugation. We say that the action of $G$ on $\n$ is isotypical, if all the $G$-orbits are isomorphic as $G$-sets. We suspect that the following simple lemma was observed by others, but we do not know a reference.
\begin{lemma}\label{lemma: not isotypical}
If the action of $G$ on $\n$ is not isotypical, then $|\Pi_n|^G$ is contractible.
\end{lemma}
\begin{proof}
Let $\theta$ be the canonical partition of $\n$ into isotypical components. Specifically, $x$ and $y$ belong to the same equivalence class of $\theta$ if the $G$-orbits of $x$ and $y$ are isomorphic $G$-sets. Since $G$ is not isotypical, $\theta$ has at least $2$ equivalence classes. It is also clear that $\theta$ has fewer than $n$ classes, because otherwise all the classes would be singletons, which would mean that $G$ acts trivially, and in particular isotypically.

Let $\lambda\in \Pi_n^G$ be a non-trivial $G$-invariant partition. Let $\theta \vee \lambda$ be the coarsest common refinement of $\theta$ and $\lambda$. It is the partition whose equivalence classes are all the intersections of classes of $\theta$ with classes of $\lambda$. We claim that $\theta \vee \lambda$ is not the discrete partition. Indeed, suppose that $\theta \vee \lambda$ is the discrete partition. Let $c$ be a non-singleton equivalence class of $\lambda$. If  $\theta \vee \lambda$ is the discrete partition then different elements of $c$ belong to different classes of $\theta$. Let $x, y$ be two elements of $c$. Since they belong to different classes of $\theta$, the $G$-orbits $Gx$ and $Gy$ are not isomorphic. It follows that the stabilizer groups of $x$ and $y$ are not the same. In particular, there exists a $g\in G$ that stabilizes one but not the other. Suppose that $gx=x$ but $gy\ne y$. Since $\lambda$ is $G$-invariant, it follows that $gx\sim gy$ in $\lambda$. But $gx=x$ and $x\sim y$, so $y\sim gy$ in $\lambda$. But $y$ and $gy$ are in the same class of $\theta$, so they are in the same class of $\lambda\vee\theta$, contradicting that $\lambda\cap \theta$ is discrete.

It follows that $\theta\vee \lambda$ is an element of $\Pi_n^G$, for all $\lambda\in \Pi_n^G$. By lemma~\ref{lemma: lattice}, the space $|\Pi_n^G|$ is contractible. 
\end{proof}
It remains to analyze the fixed points of isotypical actions. An important special case is that of transitive actions. Suppose $G$ acts transitively on $\n$. Then $\n$ can be identified (non-canonically) with the set $G/H$ for some subgroup $H\subset G$. So it is enough to describe $(\Pi_{G/H})^G$. The following is essentially~\cite[Lemma 3]{ww}.
\begin{lemma}\label{lemma: transitive}
The poset of $G$-invariant partitions of $G/H$ is isomorphic to the poset of subgroups $H\subseteq K \subseteq G$, ordered by reverse inclusion. The non-trivial partitions correspond to subgroups $H \subsetneq K \subsetneq G$.
\end{lemma}
\begin{proof}
Let $\lambda$ be a $G$-invariant partition of $G/H$. Let $K=\{g\in G\mid g(eH)\sim_\lambda eH\}$. Then $K$ is a subgroup of $G$. Indeed, suppose that $eH\sim_\lambda g_1H\sim_\lambda g_2H$. Multiplying the first equivalence by $g_2$ we obtain that $g_2H \sim_\lambda g_2g_1H$. Since $g_2H\sim_\lambda eH$, it follows that $g_2g_1\in K$. This shows that $K$ is closed under multiplication. A similar argument shows that $K$ is closed under taking inverses. Clearly, $H\subseteq K$. It is easy to see that a finer partition will lead to a smaller group $K$.

Conversely, let $K$ be a group, $H\subseteq K\subseteq G$. Let $\lambda$ be the partition of $G/H$ defined by the rule $g_1H \sim_\lambda g_2H$ if and only if $g_2^{-1}g_1\in K$. It is easy to check that $\lambda$ is a well defined $G$-invariant partition, and that we defined two maps between posets that are inverse to each other. Finally, it is clear that $K=H, G$ correspond to the discrete and the indiscrete partition respectively. 
\end{proof}
\begin{example}\label{example: full symmetric}
Consider the standard action of $\Sigma_n$ on $\n$. This can be identified with the action of $\Sigma_n$ on $\Sigma_n/\Sigma_{n-1}$. Since there are no groups $\Sigma_{n-1}\subsetneq H \subsetneq \Sigma_n$, it follows that $\Pi_n^{\Sigma_n}=\emptyset$.
\end{example}
\begin{example}\label{example: wreath}
Suppose that $n=k_1\cdot k_2\cdots k_l$, with $l>1$ and all factors greater than $1$. Consider the iterated wreath product $W:=\Sigma_{k_1}\wr\cdots\wr\Sigma_{k_l}\subset \Sigma_n$. This group fits into a normal series
\[
\{e\}\mathrel{\unlhd} G_1 \mathrel{\unlhd} \cdots \mathrel{\unlhd} G_l=W
\]
where $G_i/G_{i-1}=(\Sigma_{k_i})^{k_{i+1}\cdots k_l}$.
$W$ acts transitively on $\n$, with stabilizer groups (conjugate to) $W\cap \Sigma_{n-1}$. 
\begin{lemma}\label{lemma: wreath}
Let $W$ be an iterated wreath product as above. The space $|\Pi_n|^W$ is contractible.
\end{lemma}
\begin{proof}
By Lemma~\ref{lemma: transitive}, the poset $(\Pi_n)^W$ is (reverse) isomorphic to the poset of groups $G$ that satisfy $W\cap \Sigma_{n-1}\subsetneq G\subsetneq W$. So it is enough to show that this poset of groups has contractible nerve. We sill show that it has a minimal element.

Let $H$ be the subgroup of $W$ generated by $W\cap \Sigma_{n-1}$ and the transposition $(n-1, n)$. Explicitly, 
\[
H= \Sigma_{k_1}\wr\cdots\wr\Sigma_{k_l-1}\times \Sigma_{k_1}\wr\cdots\wr\Sigma_{k_{l-1}-1}\times \cdots \Sigma_{k_1}\wr\Sigma_{k_2-1}\times \Sigma_{k_1}.
\]
We claim that every group $G$ satisfying $ W\cap \Sigma_{n-1}\subsetneq G \subset W$ contains $H$. Indeed, let $G$ be such a group. We need to show that $G$ contains the transposition $(n-1, n)$. By assumption, $G$ contains an element $\sigma\in W\setminus \Sigma_{n-1}$. Thus $\sigma(n)\ne n$. Suppose first that $\sigma(n)\in \{n-k_1+1, \ldots, n\}$. Then, since $\sigma\in W$, $\sigma$ leaves invariant the set $\{n-k_1+1, \ldots, n\}$. Let $\overline{\sigma}$ be the permutation of $\n$ that agrees with $\sigma$ on the set $ \{n-k_1+1, \ldots, n\}$, and is the identity on the complement of this set. Clearly $\overline{\sigma}\in W$ and $\sigma^{-1}\overline{\sigma}\in W\cap \Sigma_{n-1}\subset G$. It follows that $\overline{\sigma}\in G$. Thus $G$ contains all the permutations of the set $ \{n-k_1+1, \ldots, n-1\}$, plus an additional permutation of the set $ \{n-k_1+1, \ldots, n\}$. Therefore, $G$ contains the full group of permutations of $ \{n-k_1+1, \ldots, n\}$, which is what we wanted to show. 

Now suppose that $\sigma(n)\in \{1, \ldots, n-k_1\}$. Then $\sigma(n)\in \{jk_1+1, jk_1+2, \cdots, (j+1)k_1\}$, where $jk_1<n-k_1$. Then $\sigma(n-1)\in \{jk_1+1, jk_1+2, \cdots, (j+1)k_1\}$ as well. Let $\tau$ be the transposition $(\sigma(n), \sigma(n-1))$. Clearly, $\tau\in W\cap \Sigma_{n-1}\subset G$. Thus $\sigma^{-1}\tau\sigma\in G$. But $\sigma^{-1}\tau\sigma$ is the transposition $(n-1, n)$.
\end{proof}
\end{example}

It remains to analyze the fixed points of a general isotypical action. Our strategy is to first analyze the fixed points of groups that occur as isotropy groups  of the spaces in~\eqref{equation: main}. This will enable us to prove the main theorem. That, in turn, will enable us to describe the fixed points of $\Pi_n$ for all subgroups of $\Sigma_n$.

First, some easy generalities. Suppose $G$ acts on $\n$. Let $\lambda$ be a partition of $\n$ preserved by $G$. Let $S$ and $T$ be two $G$-orbits in $\n$. We say that $S$ and $T$ are connected by $\lambda$ if there exist $s\in S$ and $t\in T$ such that $s\sim_\lambda t$. The proof of the following elementary lemma is left to the reader.
\begin{lemma}\label{lemma: equivalence}
The relation of being connected by $\lambda$ is an equivalence relation on the set of $G$-orbits.
\end{lemma}
\begin{lemma}\label{lemma: isotypical}
Suppose $G$ acts isotypically on $\n$, $\lambda$ is a $G$-invariant partition, and $S$ and $T$ are orbits of $G$ connected by $\lambda$ then the restrictions of $\lambda$ to $S$ and $T$ are isomorphic.
\end{lemma}
\begin{proof}
Let $s\in S$ and $t\in T$ be such that $s\sim_\lambda t$. Let $K_s=\{g\in G\mid gs\sim_\lambda s\}$ and $K_t=\{g\in G\mid gt\sim_\lambda t\}$. Since $\lambda$ is $G$-invariant, and $s\sim_\lambda t$, it follows that for all $g\in G$, $gs\sim_\lambda gt$. Therefore, $gs\sim s$ if and only if $gt\sim t$. In other words, $K_s=K_t$. By the proof of Lemma~\ref{lemma: transitive}, $K_s$ and $K_t$ determine the restrictions of $\lambda$ to the transitive $G$-sets $S$ and $T$. 
\end{proof}
Let us suppose that $d|n$ and $d$ factors as a product of positive integers $d=d_1\cdots d_l$ with $l>1$. Consider the iterated wreath product $\Sigma_{d_1}\wr\cdots\wr\Sigma_{d_l}$ as a subgroup of $\Sigma_n$ via the inclusions
\[
\Sigma_{d_1}\wr\cdots\wr\Sigma_{d_l}\subset \Sigma_d\subset (\Sigma_d)^{\frac{n}{d}}\subset \Sigma_n.
\]
\begin{lemma}\label{lemma: isotypical wreath}
The fixed points space $|\Pi_n|^{\Sigma_{d_1}\wr\cdots\wr\Sigma_{d_l}}$ is contractible.
\end{lemma}
\begin{proof}
The orbits of $\Sigma_{d_1}\wr\cdots\wr\Sigma_{d_l}$ are the sets \[\{1, \ldots, d\}, \{d+1, \ldots, 2d\},\ldots, \{n-d+1, \ldots, n\}.\] Let $\theta$ be the partition of $\n$ whose equivalence classes are $\{1, \ldots, d_1\}, \{d_1+1, \ldots, 2d_1\},\ldots, \{n-d_1+1, \ldots, n\}$. Then $\theta$ is a $\Sigma_{d_1}\wr\cdots\wr\Sigma_{d_l}$-invariant partition. It follows from the proof of Lemma~\ref{lemma: wreath} that the restriction of $\theta$ to each orbit of $\Sigma_{d_1}\wr\cdots\wr\Sigma_{d_l}$ is the finest non-discrete $\Sigma_{d_1}\wr\cdots\wr\Sigma_{d_l}$-invariant partition of the orbit. Let $\lambda$ be another $\Sigma_{d_1}\wr\cdots\wr\Sigma_{d_l}$-invariant partition of $\n$. Let $\lambda\wedge \theta$ be the finest common coarsening of $\lambda$ and $\theta$. Clearly, $\lambda\wedge \theta$ is a $\Sigma_{d_1}\wr\cdots\wr\Sigma_{d_l}$-invariant partition. We claim that if $\lambda$ is not indiscrete then $\lambda\wedge \theta$ is not indiscrete. To see this, suppose first that there exists two orbits $S$ and $T$ of $\Sigma_{d_1}\wr\cdots\wr\Sigma_{d_l}$ that are not connected by $\lambda$. Since each equivalence class of $\theta$ is contained in some orbit of $\Sigma_{d_1}\wr\cdots\wr\Sigma_{d_l}$, there do not exists elements $s\in S$ and $t\in T$ that are in same equivalence class of $\lambda\wedge \theta$. In particular, $\lambda\wedge \theta$ is not indiscrete. 

Now suppose that all orbits of $\Sigma_{d_1}\wr\cdots\wr\Sigma_{d_l}$ are connected by $\lambda$. Then by Lemma~\ref{lemma: isotypical} the restrictions of $\lambda$ to the orbits of $\Sigma_{d_1}\wr\cdots\wr\Sigma_{d_l}$ are pairwise isomorphic. If the restriction of $\lambda$ to each orbit is indiscrete, then (since we assume that all the orbits are connected) $\lambda$ is indiscrete. Conversely, if $\lambda$ is not indiscrete, then the restrictions of $\lambda$ to the orbits of $\Sigma_{d_1}\wr\cdots\wr\Sigma_{d_l}$ are not indiscrete. Since the restriction of $\theta$ to an orbit of  $\Sigma_{d_1}\wr\cdots\wr\Sigma_{d_l}$ is the finest non-discrete invariant partition of the orbit, it follows that the restriction of $\lambda$ to each orbit of $\Sigma_{d_1}\wr\cdots\wr\Sigma_{d_l}$ is either discrete, or is coarser than the restriction of $\theta$ to the orbit. In the first case, the restriction of $\lambda\wedge\theta$ to each orbit of $\Sigma_{d_1}\wr\cdots\wr\Sigma_{d_l}$ is the same as the restriction of $\theta$ to the orbit. In the second case, it is the same as the restriction of $\lambda$ to the orbit. In any case, the restriction of $\lambda\wedge\theta$ to each orbit of $\Sigma_{d_1}\wr\cdots\wr\Sigma_{d_l}$ is not indiscrete. It follows that $\lambda\wedge\theta$ is not indiscrete.

It follows from the claim that $\lambda\wedge \theta$ is an element of $\Pi_n^{\Sigma_{d_1}\wr\cdots\wr\Sigma_{d_l} }$ for all $\lambda\in \Pi_n^{\Sigma_{d_1}\wr\cdots\wr\Sigma_{d_l}}$. By lemma~\ref{lemma: lattice} it follows that the geometric realization of $\Pi_n^{\Sigma_{d_1}\wr\cdots\wr\Sigma_{d_l}}$ is contractible.
\end{proof}

Suppose $G$ acts isotypically on $\n$, so all $G$-orbits are isomorphic. Let us say that each orbit has $d$ elements, so there are $\frac{n}{d}$ orbits. Let us fix a transitive action of $G$ on $\mathbf{d}$, isomorphic to the action of $G$ on one of the orbits. In particular, we have fixed a homomorphism $G\to \Sigma_d$. Let us also fix a $\Sigma_d$-equivariant isomorphism $\mathbf{d}\times\mathbf{\frac{n}{d}}\stackrel{\cong}{\to}\n$. Without loss of generality we can take this isomorphism to be $(i, j)\mapsto i+(j-1)d$. Recall that $\Part_n$ is the full poset of partitions of $\n$. Consider the following composition of $\Sigma_d$-equivariant maps of posets
\[
\Part_d\times\Part_{\frac{n}{d} }\to \Part_{\mathbf{d}\times\mathbf{\frac{n}{d}}}\to \Part_n.
\]
Here the first map takes cartesian product of equivalence relations, and the second map uses the bijection $\mathbf{d}\times\mathbf{\frac{n}{d}}\stackrel{\cong}{\to}\n$.
Removing the initial and final elements, we obtain a map of posets 
\[
\overline{\Part_d\times\Part_{\frac{n}{d}}}\to \Pi_n.
\]
By Proposition~\ref{proposition: fancy join}, there are homeomorphisms $$|\overline{\Part_d\times\Part_{\frac{n}{d}}}|\cong \Sigma | \Pi_d| * |\Pi_{\frac{n}{d} }|
$$
Thus
we obtain a $\Sigma_d$-equivariant map of spaces
\[
\Sigma |\Pi_d| * |\Pi_{\frac{n}{d} }|\to |\Pi_n|.
\]
Note that the partition of $\n$ into the orbits of $\Sigma_d$ determines a $\Sigma_d$-invariant point of $|\overline{\Pcal_d\times\Pcal_{\frac{n}{d}}}|$, and a $(\Sigma_d)^{\frac{n}{d}}$-invariant point of $|\Pi_n|$. The map above preserves these basepoints. Therefore it induces a $(\Sigma_d)^{\frac{n}{d}}$-equivariant map 
\[
\Sigma_d\times \cdots \times \Sigma_d{_+} \wedge_{\Sigma_d}\Sigma| \Pi_d * \Pi_{\frac{n}{d} }| \to |\Pi_n|.
\]
Viewing $G$ as a subgroup of $(\Sigma_d)^{\frac{n}{d}}$, we obtain a map of $G$-fixed points
\begin{equation}\label{equation: Gmap}
(\Sigma_d\times \cdots \times \Sigma_d{_+} \wedge_{\Sigma_d}\Sigma| \Pi_d * \Pi_{\frac{n}{d} }|)^G \to |\Pi_n|^G.
\end{equation}
Recall that $C_{\Sigma_d}(G)$ is the centralizer of $G$ in $\Sigma_d$. Clearly, the centralizer of $G$ in $(\Sigma_d)^{\frac{n}{d}}$ is $C_{\Sigma_d}(G)^{\frac{n}{d}}$. Note also that $\Sigma| \Pi_d * \Pi_{\frac{n}{d} }|^G\cong \Sigma | \Pi_d|^G * \Pi_{\frac{n}{d} }|.$ By corollary~\ref{corollary: simplified fixed points}, there is a homeomorphism
\begin{equation}
\label{equation: fixed}
C_{\Sigma_d}(G)^{\frac{n}{d}}_+\wedge_{C_{\Sigma_d}(G)}  \Sigma | \Pi_d|^G * |\Pi_{\frac{n}{d} }| \stackrel{\cong}{\to} \left((\Sigma_d)^{\frac{n}{d}}_+ \wedge_{\Sigma_d} \Sigma| \Pi_d * \Pi_{\frac{n}{d} }|\right)^G.
\end{equation}
Combine~\eqref{equation: fixed} with~\eqref{equation: Gmap} to obtain the following map:
\begin{equation}\label{equation: fixed points}
{C_{\Sigma_d}(G)^{\frac{n}{d}}}_+ \wedge_{C_{\Sigma_d}(G)}\Sigma |\Pi_d|^G * |\Pi_{\frac{n}{d} }| \to |\Pi_n|^G.
\end{equation}

\begin{lemma}\label{lemma: fixed points}
Suppose $G=\Sigma_{d_1}\wr\cdots\wr\Sigma_{d_l}$, where $d=d_1\cdots d_l$, $l\ge 1$ (note that we include $l=1$). The map~\eqref{equation: fixed points} is a homotopy equivalence.
\end{lemma}
\begin{remark} We will need Lemma~\ref{lemma: fixed points} to prove our main result. Later we will use the main result to prove that the conclusion of the lemma holds when $G$ is any transitive subgroup of $\Sigma_d$. 
\end{remark}
%
\begin{lemma}\label{lemma: invariant partitions}
Suppose, as usual, that $d|n$ and $G$ is a transitive subgroup of $\Sigma_d$, embedded diagonally in $\Sigma_n$. Let $T$ be the set of $G$-invariant partitions of $\n$ with the property that each equivalence class intersects each $G$-orbit at a single element. There are bijections 
\[
T\cong C_{\Sigma_d}(G)^{\frac{n}{d}-1}\cong C_{\Sigma_d}(G)^{\frac{n}{d}}/C_{\Sigma_d}(G).
\]
\end{lemma}
\begin{proof}
By our assumptions the action of $G$ on $\n$ has $\frac{n}{d}$ orbits, each isomorphic to $\mathbf d$. Let $O_1, \ldots, O_{\frac{n}{d}}$ be the orbits of $G$. Let $\lambda\in T$ be a partition of $\n$ that intersects each orbit at a single element. Then for each $i=2, \ldots, \frac{n}{d}$, $\lambda$ determines a $G$-equivariant bijection $O_1 \to O_i$ which associates to each element of $O_1$ the unique element of $O_i$ that is in the same equivalence class of $\lambda$. The set of $G$-equivariant bijections from $O_1$ to $O_i$ is canonically isomorphic to $C_{\Sigma_d}(G)$. It follows that $\lambda$ determines an element of $C_{\Sigma_d}(G)^{\frac{n}{d}-1}$. We have defined a map $T\to C_{\Sigma_d}(G)^{\frac{n}{d}-1}$. It is easy to check that this map is a bijection.
\end{proof}

\begin{proof}[Proof of Lemma~\ref{lemma: fixed points}]
Consider first the case $l>1$, so that we have a wreath product of at least two factors. In that case $|\Pi_d|^G$ is contractible by Lemma~\ref{lemma: wreath}, and $|\Pi_n|^G$ is contractible by Lemma~\ref{lemma: isotypical wreath}. It follows that the source and the target of the map~\eqref{equation: fixed points} are contractible, and so the map is a homotopy equivalence.

It remains to analyze the case $l=1$, or in other words, $G=\Sigma_d$. Our goal is to prove that the map~\eqref{equation: fixed points} is an equivalence. We will now analyze the homotopy type of the target of the map, i.e, of the space $|\Pi_n|^{\Sigma_d}$.  

Before we proceed, it is worth noting that $C_{\Sigma_d}(\Sigma_d)$ is the center of $\Sigma_d$. We sometimes may write it simply as $C(\Sigma_d)$. This group is trivial if $d\ne 2$ and is $\Sigma_2$ when $d=2$. However, we will not separate the cases, but just write $C(\Sigma_d)$ as per the general case. Note also that $C(\Sigma_d)$ acts trivially on  $|\Pi_d|^{\Sigma_d}*|\Pi_{\frac{n}{d}}|$. Therefore, there is a natural isomorphism
\begin{equation}\label{equation: T}
{C(\Sigma_d)^{\frac{n}{d}}}_+ \wedge_{C(\Sigma_d)}\Sigma |\Pi_d|^{\Sigma_d} * |\Pi_{\frac{n}{d} }|\cong T_+ \wedge \Sigma |\Pi_d|^{\Sigma_d} * |\Pi_{\frac{n}{d} }|.
\end{equation}
Recall that $T\subset (\Pi_n)^{\Sigma_d}$ is the set of ${\Sigma_d}$-invariant partitions that intersect each ${\Sigma_d}$ orbit at a single element (in particular, every $\lambda\in T$ has $d$ equivalence classes, with $\frac{n}{d}$ elements in each). Let $\Pi_{\le T\le}^{\Sigma_d}$ be the poset of ${\Sigma_d}$-invariant partitions of $\n$ that are either finer or coarser than at least one element of $T$. 

Let $\lambda_0<\cdots <\lambda_n$ be a chain of ${\Sigma_d}$-invariant partitions. We claim that either the entire chain is contained in the poset $\Pi_n^{\Sigma_d}\setminus T$ or it is contained in $\Pi_{\le T\le}^{\Sigma_d}$. Indeed, 
suppose the entire chain is not contained in $(\Pi_n)^{\Sigma_d}\setminus T$. Then $\lambda_i\in T$ for some $\lambda_i$ in the chain. But then it is clear that the entire chain is in $\Pi_{\le T\le}^{\Sigma_d}$.

It follows from this claim that $|\Pi_n|^{\Sigma_d}\cong |\Pi_n^{\Sigma_d}\setminus T| \cup |\Pi_{\le T\le}^{\Sigma_d}|$, and moreover that $|\Pi_n|^{\Sigma_d}$ is equivalent to the homotopy pushout of the diagram
\begin{equation}\label{equation: pushout}
 |\Pi_n^{\Sigma_d}\setminus T| \leftarrow |(\Pi_n^{\Sigma_d}\setminus T) \cap \Pi_{\le T\le}^{\Sigma_d}|\to  |\Pi_{\le T\le}^{\Sigma_d}|.
\end{equation}

Next we analyze $|\Pi_n^{\Sigma_d}\setminus T|$. Specficially, we will show that it is contractible. Let $\lambda\in \Pi_n^{\Sigma_d}\setminus T$. Then in particular $\lambda$ is a $\Sigma_d$-invariant partition of $\n$. The restriction of $\lambda$ to an orbit of $\Sigma_d$ defines a $\Sigma_d$-invariant partition of this orbit. Since each orbit is isomorphic to $\mathbf{d}$ with the standard action of $\Sigma_d$, the restriction of $\lambda$ to each orbit is either the discrete or the indiscrete partition of that orbit (this follows from Example~\ref{example: full symmetric}). We claim that not all the orbits of $\Sigma_d$ are connected by $\lambda$. Indeed, suppose that all orbits are connected. Then the restriction of $\lambda$ to each orbit is either indiscrete for all orbits, or discrete for all orbits. In the first case, $\lambda$ is itself indiscrete, in which case it is not in $\Pi_n$. In the second case, $\lambda\in T$, contradicting our assumption. Let $\theta$ be the partition of $\n$ into $\Sigma_d$-orbits, and let $\lambda\wedge \theta$ be the finest common coarsening of $\lambda$ and $\theta$. Since not all the classes of $\theta$ are connected by $\lambda$, it follows that $\lambda\wedge \theta$ is not the indiscrete partition. By Lemma~\ref{lemma: lattice} $|\Pi_n^{\Sigma_d}\setminus T|$ is a contractible space.
 
 Since $|\Pi_n^{\Sigma_d}\setminus T|$ is contractible, the space $|\Pi_n|^{\Sigma_d}$, which is equivalent to the homotopy pushout on line~\eqref{equation: pushout}, is homotopy equivalent to the cofiber (as well as homotopy cofiber) of the inclusion map
 \[
  |(\Pi_n^{\Sigma_d}\setminus T) \cap \Pi_{\le T\le}^{\Sigma_d}|\to  |\Pi_{\le T\le}^{\Sigma_d}|.
  \]
For an element $t\in T$. Let $\Pi_{\le t \le}$ be the subposet of $(\Pi_n)^{\Sigma_d}$ consisting of partitions that are comparable with $t$. Let $\Pi_{< t <}=\Pi_{\le t \le}\setminus\{t\}$. Then clearly we have decompositions
\[
|\Pi_{\le T\le}^{\Sigma_d}|=\bigcup_{t\in T} |\Pi_{\le t\le}^{\Sigma_d}|
\]
and
\[
( |\Pi_n^{\Sigma_d}\setminus T) \cap \Pi_{\le T\le}^{\Sigma_d}|\cong |\Pi_{< T<}^{\Sigma_d}|=\bigcup_{t\in T} |\Pi_{< t<}^{\Sigma_d}|.
\]

It is clear that if $t_1, t_2$ are distinct elements of $T$ then \[\Pi_{\le t_1\le}^{\Sigma_d}\cap \Pi_{\le t_2\le}^{\Sigma_d}\subset (\Pi_n)^{\Sigma_d}\setminus T.\] It follows that there is a homeomorphism
\[
|\Pi_{\le T\le}^{\Sigma_d}|/  |\Pi_n^{\Sigma_d}\setminus T) \cap \Pi_{\le T\le}^{\Sigma_d}|\cong \bigvee_{t\in T}|\Pi_{\le t\le}^{\Sigma_d}|/|\Pi_{< t<}^{\Sigma_d}|.
\]
There are decompositions 
\[
|\Pi_{\le t\le}^{\Sigma_d}|\cong \{t\} *|\Pi_{<t}|^{\Sigma_d} * |\Pi_{>t}|^{\Sigma_d}
\]
and
\[
|\Pi_{< t<}^{\Sigma_d}|\cong |\Pi_{<t}|^{\Sigma_d} * |\Pi_{>t}|^{\Sigma_d}.
\]
It follows that 
\[
|\Pi_{\le t\le}^{\Sigma_d}|/|\Pi_{< t<}^{\Sigma_d}|\cong \Sigma |\Pi_{<t}|^{\Sigma_d} * |\Pi_{>t}|^{\Sigma_d}.
\]
Therefore, we obtain a homeomorphism, 
\[
|\Pi_{\le T\le}^{\Sigma_d}|/  |(\Pi_n)^{\Sigma_d}\setminus T) \cap \Pi_{\le T\le}^{\Sigma_d}|\cong \bigvee_{t\in T} \Sigma |\Pi_{t<}|^{\Sigma_d}*|\Pi_{>t}|^{\Sigma_d}.
\]
Recall that each $t\in T$ is a ${\Sigma_d}$-invariant partition of $\n$ with the property that each class of the partition intersects each ${\Sigma_d}$-orbit at exactly one element. In particular, $t$ has $d$ equivalence classes, that are permuted transitively by ${\Sigma_d}$, with each class contaning $\frac{n}{d}$ elements. It follows that $\Pi_{<t}^{\Sigma_d}$, the poset of ${\Sigma_d}$-invariant coarsening of $t$, is isomorphic to $\Pi_d^{\Sigma_d}$, which is in fact empty. On the other hand $\Pi_{t<}^{\Sigma_d}$, the poset of $G$-invariant refinements of $t$, is, by Proposition~\ref{proposition: fancy join}, isomorphic to $(\Sigma^{d-2}*\Pi_{\frac{n}{d}}^{* d} )^{\Sigma_d}$, where ${\Sigma_d}$ acts on $\Sigma^{d-2}$ with no fixed points, and acts on $\Pi_{\frac{n}{d}}^{* d}$ by permuting the $d$ copies of $\Pi_{\frac{n}{d}}$. The fixed points space of this action is $\Pi_{\frac{n}{d}}$. It follows that 
\[
|\Pi_n|^{\Sigma_d}\simeq T_+\wedge \Sigma |\Pi_d|^{\Sigma_d} * |\Pi_{\frac{n}{d}}|.
\]
In view of isomorphism~\eqref{equation: T}, we proved that in the case $G=\Sigma_d$, the space $|\Pi_n|^G$, which is the target of the map~\eqref{equation: fixed points} is homotopy equivalent to the source of that same map. The source of the map also can be decomposed as a pushout analogous to~\eqref{equation: pushout}. The map  respects the decomposition and induces a termwise equivalence of the pushout diagrams, therefore it induces an equivalence between homotopy pushouts.
\end{proof}

\section{Proof of the main theorem}\label{section: main proof}
As usual, let $\Sigma_{n_1}\times \cdots \times\Sigma_{n_k}\subset \Sigma_n$ be a Young subgroup. We need to classify the isotypical isotropy groups of $\Pi_n$, considered as a space with an action of $\Sigma_{n_1}\times \cdots \times\Sigma_{n_k}\subset \Sigma_n$.
\begin{lemma}\label{lemma: isotropy}
Consider the action of $\Sigma_{n_1}\times \cdots \times\Sigma_{n_k}$ on $\Pi_n$. Let $G\subset \Sigma_{n_1}\times \cdots \times\Sigma_{n_k}$ be an isotropy group of a simplex in $\Pi_n$. Suppose that $G$ acts isotypically on $\n$. Then $G$ is conjugate to a group of the form $\Sigma_{d_l}\wr\cdots\wr \Sigma_{d_1}$, where $l\ge 1$, and $d_1\cdots d_l | \gcd(n_1, \ldots, n_k)$. As usual, the group $G$ acts diagonally on blocks of size $d_1\cdots d_l$. The possibility that $G$ is the trivial group is included.
\end{lemma}
\begin{proof}
Let $\lambda_1 < \cdots < \lambda_i$ be a chain of partitions of $\n$. Let $G$ be the stabilizer of this chain in $\Sigma_{n_1}\times \cdots \times\Sigma_{n_k}\subset \Sigma_n$. Let us analyze the general form of $G$. The group $G$ acts on the set of equivalence classes of each $\lambda_j$. Recall that $\lambda_1$ is the coarsest partition of the chain. Let us say that two equivalence classes of $\lambda_1$ are of the same type, if they are in the same orbit of the action of $G$. This is an equivalence relation on the set of equivalence classes of $\lambda_1$. Let us say that there are $s$ different types of equivalence classes of $\lambda_1$, and there are $t_1$ classes of type $1$, $t_2$ classes of type $2$ and so forth. It is easy to see that in this case $G$ is isomorphic to a group of the following form
\[
G\cong K_1\wr\Sigma_{t_1} \times \cdots \times K_s\wr \Sigma_{t_s}
\]
where $K_j$ is the group of automorphisms (in $\Sigma_{n_1}\times \cdots \times\Sigma_{n_k}$) of the restriction of the chain  $\lambda_1 < \cdots < \lambda_i$ to an equivalence class of $\lambda_1$ of type $j$. We conclude that if $G$ acts isotypically on $\n$ then either $G$ is trivial, or $s=1$ and $G\cong K_1\wr \Sigma_{t_1}$. Furthermore, if $G$ acts isotypically on $\n$ then $K_1$ has to act isotypically on a component of $\lambda_1$. The lemma follows by an induction on $i$.
%
\end{proof}
Now we are ready to prove the main theorem.
\begin{proof}[Proof of Theorem~\ref{theorem: main}]
We need to show that the $\Sigma_{n_1}\times\cdots\times\Sigma_{n_k}$-equivariant map
\begin{equation}\label{equation: main map again}
\Pi_n \longrightarrow \bigvee_{d|g} \left(\bigvee_{B(\frac{n_1}{d}, \ldots, \frac{n_k}{d})} \Sigma_{n_1}\times \cdots \times {\Sigma_{n_k}}_+ \wedge_{\Sigma_d} S^{n-d-1} * \Pi_d\right)
\end{equation}
that we constructed in Proposition~\ref{proposition: collapse map} is a $\Sigma_{n_1}\times\cdots\times\Sigma_{n_k}$-equivariant homotopy equivalence. This means that we need to prove that the map induces a homotopy equivalence of $G$-fixed points, for every subgroup $G\subset \Sigma_{n_1}\times\cdots\times\Sigma_{n_k}$. By Lemma~\ref{lemma: dwyer lemma}, it is enough to prove it for groups $G$ that occur as isotropy groups of either $\Pi_n$ (with respect to the action of $\Sigma_{n_1}\times\cdots\times\Sigma_{n_k}$) or of $ \Sigma_{n_1}\times \cdots \times {\Sigma_{n_k}}_+ \wedge_{\Sigma_d} S^{n-d-1} * \Pi_d$. 

Suppose first that $G$ is not an isotypical subgroup of $\Sigma_{n_1}\times\cdots\times\Sigma_{n_k}$. Then $|\Pi_n|^G$ is contractible by Lemma~\ref{lemma: nondiagonal}. On the other hand, the fixed point space $\left(\Sigma_{n_1}\times \cdots \times {\Sigma_{n_k}}_+ \wedge_{\Sigma_d} S^{n-d-1} * \Pi_d\right)^G$ is contractible if $G$ is not conjugate to a subgroup of $\Sigma_d$. If $G$ is a subgroup of $\Sigma_d$ then $\left(\Sigma_{n_1}\times \cdots \times {\Sigma_{n_k}}_+ \wedge_{\Sigma_d} S^{n-d-1} * |\Pi_d|\right)^G$ is homeomorphic to a wedge sum of some copies of $\left(S^{n-d-1} * |\Pi_d|\right)^G$ (Corollary~\ref{corollary: simplified fixed points}). If $G\subset \Sigma_d\subset\Sigma_n$ does not act isotypically on $\n$, then it does not act isotypically on $\mathbf d$, and therefore $\left(S^{n-d-1} * |\Pi_d|\right)^G\cong (S^{n-d-1})^G * |\Pi_d|^G$ is contractible. We conclude that if $G$ is not isotypical, then both the source and the target of the map~\eqref{equation: main map again} are contractible, and in particular the map is a homotopy equivalence.

It remains to consider the case when $G$ is an isotypical group that occurs as an isotropy group of the action of $\Sigma_{n_1}\times \cdots \times {\Sigma_{n_k}}$ on $\Pi_n$ and on $\Sigma_{n_1}\times \cdots \times {\Sigma_{n_k}}_+ \wedge_{\Sigma_d} S^{n-d-1} * |\Pi_d|$. By Lemma~\ref{lemma: isotropy}, isotypical isotropy groups of $|\Pi_n|$ are groups of the form $\Sigma_{d_1}\wr\cdots \wr\Sigma_{d_l}$. On the other side of the map, the isotropy groups of the action of $\Sigma_{n_1}\times \cdots \times {\Sigma_{n_k}}$ on $\Sigma_{n_1}\times \cdots \times {\Sigma_{n_k}}_+ \wedge_{\Sigma_d} S^{n-d-1} * |\Pi_d|$ are the same as the isotropy groups of the action of $\Sigma_d$ on  $S^{n-d-1} * |\Pi_d|$. The isotropy groups of $S^{n-d-1}$ are groups of the form $\Sigma_{d_1}\times\cdots\times\Sigma_{d_j}$ (Young subgroups of $\Sigma_d$). Therefore the isotropy groups of the action of $\Sigma_d$ on $S^{n-d-1} * |\Pi_d|$ are the isotropy groups of the action of Young subgroups of $\Sigma_d$ on $|\Pi_d|$. Again by Lemma~\ref{lemma: isotropy}, the only isotypical isotropy groups are groups of the form $\Sigma_{d_1}\wr\cdots \wr\Sigma_{d_l}$. So we need to check that the map~\eqref{equation: main map again} induces an equivalence of fixed points of groups of the form $\Sigma_{d_1}\wr\cdots \wr\Sigma_{d_l}$ (including the trivial group).

For trivial isotopy group, the map is a cohomology isomorphism by Proposition~\ref{proposition: collapse map}. Since both the source and the target of the map are equivalent to wedges of spheres of dimension $n-3$, it follows that the map is a homotopy equivalence except possibly when $n=4$. But this case can easily be checked ``by hand''. Indeed, Figure~\ref{figure: n=4} illustrates the case of the group $\Sigma_2\times\Sigma_2$, which is the most interesting Young subgroup of $\Sigma_4$.

Assume therefore that the isotropy group is non-trivial.

If $l>1$, then $|\Pi_n|^{\Sigma_{d_1}\wr\cdots \wr\Sigma_{d_l}}$ is contractible by Lemma~\ref{lemma: isotypical wreath}. On the other hand $\left(\Sigma_{n_1}\times \cdots \times {\Sigma_{n_k}}_+ \wedge_{\Sigma_d} S^{n-d-1} * |\Pi_d|\right)^{\Sigma_{d_1}\wr\cdots \wr\Sigma_{d_l}}$ is contractible if $\Sigma_{d_1}\wr\cdots \wr\Sigma_{d_l}$ is not (conjugate to) a subgroup of $\Sigma_d$, or is equivalent to a wedge sum of copies of $\left(S^{n-d-1} * \Pi_d\right)^{\Sigma_{d_1}\wr\cdots \wr\Sigma_{d_l}}$, which is, again, contractible by Lemma~\ref{lemma: isotypical wreath}.

It remains to consider isotropy groups of the form $\Sigma_{d_1}$, where $d_1$ is some number greater than $1$ that divides $n_1, \ldots, n_k$. Fix such a number. We need to show that the map~\eqref{equation: main map again} induces a homotopy equivalence of $\Sigma_{d_1}$ fixed points spaces:
\begin{equation}\label{equation: goal}
|\Pi_n|^{\Sigma_{d_1}} \longrightarrow \bigvee_{d|g} \bigvee_{B(\frac{n_1}{d}, \ldots, \frac{n_k}{d})} \left(\Sigma_{n_1}\times \cdots \times {\Sigma_{n_k}}_+ \wedge_{\Sigma_d} S^{n-d-1} * |\Pi_d|\right)^{\Sigma_{d_1}}.
\end{equation}
By Lemma~\ref{lemma: fixed points}, $|\Pi_n|^{\Sigma_{d_1}}$ is homotopy equivalent to 
\[
\left(C(\Sigma_{d_1})^{\frac{n}{d_1}}\right)_+\wedge_{C(\Sigma_{d_1})} \Sigma |\Pi_{\frac{n}{d_1}}| .
\]
Here $C(\Sigma_{d_1})$ is the center of $\Sigma_{d_1}$, which is trivial if $d_1>2$ and is $\Sigma_2$ if $d_1=2$.

On the other hand $ \left(\Sigma_{n_1}\times \cdots \times {\Sigma_{n_k}}_+ \wedge_{\Sigma_d} S^{n-d-1} * |\Pi_d|\right)^{\Sigma_{d_1}}$ is contractible if $d_1$ does not divide $d$. If $d_1|d$, then by Corollary~\ref{corollary: simplified fixed points} the fixed points space is equivalent to 
\begin{equation}\label{equation: first fixed}
C_{\Sigma_{n_1}\times\cdots\times\Sigma_{n_k}}(\Sigma_{d_1})_+\wedge_{C_{\Sigma_d}(\Sigma_{d_1})} (S^{n-d-1} * |\Pi_d|)^{\Sigma_{d_1}}.
\end{equation}
The centralizer $C_{\Sigma_d}(\Sigma_{d_1})$ is isomorphic to $\Sigma_{\frac{d}{d_1}}\ltimes C(\Sigma_{d_1})^{\frac{d}{d_1}}$. Similarly,
\[
C_{\Sigma_{n_1}\times\cdots\times\Sigma_{n_k}}(\Sigma_{d_1}) \cong {\Sigma_{\frac{n_1}{d_1}}\times\cdots\times\Sigma_{\frac{n_k}{d_1}}}\ltimes C(\Sigma_{d_1})^{\frac{n}{d_1}}.
\]
It follows that the fixed points space~\eqref{equation: first fixed} is equivalent to
\begin{equation}\label{equation: second fixed}
 {\Sigma_{\frac{n_1}{d_1}}\times\cdots\times\Sigma_{\frac{n_k}{d_1}}\ltimes C(\Sigma_{d_1})^{\frac{n}{d_1}}}_+\wedge_{\Sigma_{\frac{d}{d_1}}\ltimes C(\Sigma_{d_1})^{\frac{d}{d_1}}}  (S^{n-d-1} * |\Pi_d|)^{\Sigma_{d_1}}.
 \end{equation}
Clearly $(S^{n-d-1})^{\Sigma_{d_1}}\cong S^{\frac{n}{d_1}-\frac{d}{d_1} -1}$. By Lemma~\ref{lemma: fixed points}, 
\[
|\Pi_d|^{\Sigma_{d_1}}\cong C(\Sigma_{d_1})^{\frac{d}{d_1}}_+\wedge_{C(\Sigma_{d_1})} \Sigma |\Pi_{\frac{d}{d_1}}|
\]
It follows that~\eqref{equation: second fixed} can be rewritten as
\[
 {\Sigma_{\frac{n_1}{d_1}}\times\cdots\times\Sigma_{\frac{n_k}{d_1}}\ltimes C(\Sigma_{d_1})^{\frac{n}{d_1}}}_+\wedge_{\Sigma_{\frac{d}{d_1}}\times C(\Sigma_{d_1})}  (S^{\frac{n}{d_1}-\frac{d}{d_1}-1} * \Sigma |\Pi_{\frac{d}{d_1}}|).
 \]
 This can be rewritten, as a space, in the following form
 \[
C(\Sigma_{d_1})^{\frac{n}{d_1}}/{C(\Sigma_{d_1})}_+ \wedge  {\Sigma_{\frac{n_1}{d_1}}\times\cdots\times\Sigma_{\frac{n_k}{d_1}}}_+\wedge_{\Sigma_{\frac{d}{d_1}}}  (S^{\frac{n}{d_1}-\frac{d}{d_1}-1} * \Sigma |\Pi_{\frac{d}{d_1}}|).
 \]
It follows that the map~\eqref{equation: goal} is equivalent to a map of the following form, suspended, and (if $d_1=2$) smashed with $C(\Sigma_{d_1})^{\frac{n}{d_1}}_+$ over $C(\Sigma_{d_1})$
\begin{equation}\label{equation: identified}
 \Pi_{\frac{n}{d_1}}\to  \bigvee_{d'|\frac{g}{d_1}} \bigvee_{B(\frac{n_1}{d'd_1}, \ldots, \frac{n_k}{d'd_1})}\Sigma_{\frac{n_1}{d_1}}\times \cdots \times {\Sigma_{\frac{n_k}{d_1}}}_+ \wedge_{\Sigma_{d'}} S^{\frac{n}{d_1}-d'-1} * |\Pi_{d'}|.
\end{equation}
We showed in Proposition~\ref{proposition: collapse map} that there is a homotopy equivalence between these spaces (more precisely, we showed that there is a cohomology equivalence, which is a homotopy equivalence because the spaces are simply-connected except in the case $\frac{n}{d_1}=4$, which can be checked directly). A tedious but straightforward inspection of all the maps involved shows the map~\eqref{equation: goal} is the same as the one that was proved to be an equivalence in Proposition~\ref{proposition: collapse map}.
\end{proof}
\section{Application: fixed points}\label{section: application fixed}
Let $G\subset \Sigma_n$ be a subgroup. We know that if $G$ is not isotypical then $|\Pi_n|^G$ is contractible. Suppose $G$ is isotypical. Then we may assume that $G$ is a transitive subgroup of $\Sigma_d\subset(\Sigma_d)^{\frac{n}{d}}\subset \Sigma_n$, for some $d$ that divides $n$. The following proposition describes the fixed points of an action of an isotypical group in terms of the fixed points of an action of a transitive group (which was analyzed in Lemma~\ref{lemma: transitive})
\begin{proposition}\label{proposition: fixed points}
The map that was constructed in~\eqref{equation: fixed points} 
\[
C_{\Sigma_d}(G) \times\cdots \times C_{\Sigma_d}(G)\wedge_{C_{\Sigma_d}(G)}\Sigma| \Pi_d|^G * |\Pi_{\frac{n}{d} }| \to |\Pi_n|^G.
\]
is a homotopy equivalence for all transitive subgroups $G$ of $\Sigma_d$.
\end{proposition}
\begin{proof}
By theorem~\ref{theorem: main} there is a $(\Sigma_d)^{\frac{n}{d}}$-equivariant equivalence
\[
\Pi_n \longrightarrow \bigvee_{d'|d} \bigvee_{B(\frac{d}{d'}, \ldots, \frac{d}{d'})} \Sigma_{d}\times \cdots \times {\Sigma_{d}}_+ \wedge_{\Sigma_{d'}} S^{n-d'-1} * |\Pi_{d'}|.
\]
Since $G\subset (\Sigma_d)^{\frac{n}{d}}$, this map induces an equivalence of $G$-fixed points. Since $G$ is a transitive subgroup of $\Sigma_d$, the fixed point space \[ \left(\Sigma_{d}\times \cdots \times {\Sigma_{d}}_+ \wedge_{\Sigma_{d'}} S^{n-d'-1} * |\Pi_{d'}|\right)^G\] is contractible for $d'<d$. On the other hand, for $d'=d$ the fixed points space is homeomorphic to
\[
\bigvee_{B(1, \ldots, 1)} C_{\Sigma_d}(G)^{\frac{n}{d}}_+ \wedge_{C_{\Sigma_d}(G)} (S^{n-d-1} * \Pi_{d})^G\cong \bigvee_{B(1, \ldots, 1)} C_{\Sigma_d}(G)^{\frac{n}{d}}_+ \wedge_{C_{\Sigma_d}(G)} S^{\frac{n}{d}-2} * |\Pi_{d}|^G.
\]
So we need to show that the following composed map is a homotopy equivalence
\[
{C_{\Sigma_d}(G)^{\frac{n}{d}}}_+ \wedge_{C_{\Sigma_d}(G)}\Sigma |\Pi_{\frac{n}{d} }|*| \Pi_d|^G  \to \bigvee_{B(1, \ldots, 1)} C_{\Sigma_d}(G)^{\frac{n}{d}} \wedge_{C_{\Sigma_d}(G)} S^{\frac{n}{d}-2} * |\Pi_{d}|^G.
\]
A bit of diagram chasing confirms that this map is induced by the map \[|\Pi_{\frac{n}{d} }|\to \bigvee_{B(1, \ldots, 1)} S^{\frac{n}{d}-3}\] that collapses to a point the complement of the maximal simplices of $|\Pi_{\frac{n}{d}}|$ that correspond to the elements of $B(1, \ldots, 1)$. The collapse map is a homotopy equivalence, and therefore the map on the fixed points is a homotopy equivalence.
\end{proof}
\begin{remark}
It is interesting to compare our results with those of~\cite{adl2}. In that paper, the authors analyzed the fixed point space of $|\Pi_n|^P$ where $P$ is a $p$-subgroup of $\Sigma_n$. More precisely, the following result was proved: the only $p$-subgroups $P\subset \Sigma_n$ for which the fixed points space $|\Pi_n|^P$ is not contractible are elementary abelian subgroups acting freely on $\n$.

It also was observed in [loc. cit.] that if $n=p^k$ and $P\cong({\integers}/p)^k$ is an elementary abelian group acting freely and transitively on $\n$, then ${\Pi_n}^P$ is isomorphic to the poset of proper non-trivial subgroups of $P$, which is closely related to the Tits building for $\GL_k(\field_p)$. Let $T_k$ be the geometric realization of the poset of proper non-trivial subgroups of $P$. It is well-known that $$T_k\simeq \bigvee_{p^{k\choose 2}} S^{k-2}.$$ Now we can complete the calculation of fixed points for all $p$-subgroups of $\Sigma_n$. The following is an immediate consequence of Proposition~\ref{proposition: fixed points}.
\begin{corollary}\label{corollary: el abelian}
Suppose $P\cong({\integers}/p)^k$ is an elementary abelian group acting freely on $\n$. Write $n=mp^k$ (note that $p$ may divide $m$). The fixed points space of the action of $P$ on $|\Pi_n|$ is given by any one of the following formulas
\[
|\Pi_n|^P\simeq ({\integers}/p)^{km}_+\wedge_{({\integers}/p)^k} \Sigma T_k  *|\Pi_{m}|\simeq \bigvee_{p^{k(m-1)+{k\choose 2}}(m-1)!} S^{m+k-3}.
\]
\end{corollary}
\end{remark}

\section{Application: orbit spaces} \label{section: application orbit}
In the previous section we showed how Theorem~\ref{theorem: main} can be used to analyze the fixed points of the action of a subgroup of $\Sigma_n$ on $\Pi_n$. In this section we will apply the theorem to obtain information about the orbit space of the action of a Young subgroup on $\Pi_n$. Indeed, Theorem~\ref{theorem: main} has the following immediate corollary.
\begin{proposition}\label{proposition: orbits again}
Suppose $n=n_1+\cdots+n_k$. Let $g=\gcd(n_1, \ldots, n_k)$. There is a homotopy equivalence
\[
_{\Sigma_{n_1}\times\cdots\times\Sigma_{n_k}}\!\!\backslash\Pi_n \longrightarrow \bigvee_{d|g}\bigvee_{B(\frac{n_1}{d}, \ldots, \frac{n_k}{d})} {_{\Sigma_d}}\backslash \! \left(S^{n-d-1} \Smash |\Pi_d|^\diamond \right).
\]
\end{proposition}
\begin{corollary}\label{corollary: gcd one}
If $\gcd(n_1, \ldots, n_k)=1$ then 
\[
_{\Sigma_{n_1}\times\cdots\times\Sigma_{n_k}}\!\!\backslash|\Pi_n| \simeq \bigvee_{B({n_1}, \ldots, {n_k})} S^{n-3}.
\]
\end{corollary}
The proposition reduces the problem of analyzing the space $_{\Sigma_{n_1}\times\cdots\times\Sigma_{n_k}}\!\!\backslash|\Pi_n|$ to spaces of the form $_{\Sigma_d}\backslash \! \left(S^{ld-1} \Smash |\Pi_d|^\diamond \right)$, where $l\ge 1$. Lemma~\ref{lemma: contractible} below will tell us that we may as well assume $l\ge 2$ (in particular, we may assume that ${_{\Sigma_d}}\backslash \! \left(S^{ld-1} \Smash |\Pi_d|^\diamond \right)$ is simply connected). 
First we need an auxiliary lemma. 
\begin{lemma}\label{lemma: preliminary contractible}
Let $Y:=\Sigma_{d_1}\times\cdots\times \Sigma_{d_i}$ be a non-trivial Young subgroup of of $\Sigma_d$, where $d=d_1+\cdots+d_i$. Let $N$ be the normalizer of $Y$ and let $W$ be any group satisfying $Y\subseteq W \subseteq N$. Equip $S^{d-1}$ with the standard action of $\Sigma_d$. Then the orbit space $_{W}\backslash S^{d-1}$ is contractible.
\end{lemma}
\begin{proof}
To begin with, we claim that the orbits space $_{\Sigma_d}\backslash S^{d-1}$ is contractible for $d>1$. This is well-known. One way to see this is to identify $S^{d-1}$ with $|\overline{\Bcal_d}|^\diamond$, where $\overline{\Bcal_d}$ is the poset of proper, non-trivial subsets of ${\mathbf d}$ (Example~\ref{example: sphere}). Then $$_{\Sigma_d}\backslash S^{d-1}\cong {_{\Sigma_d}}\backslash|\overline{\Bcal_d}|^\diamond.$$ It is easy to see that quotient of the simplicial nerve of ${\Bcal_d}$ by the action of $\Sigma_d$ is isomorphic to the nerve of the linear poset $1<2<\cdots< d-1$. In particular, its geometric realization is contractible.

For the general case, observe that there is a $Y$-equivariant homeomorphism $S^{d-1}\cong S^{d_1-1}*\cdots*S^{d_i-1}$. It follows that 
\[
_{Y}\backslash S^{d-1}\cong {_{\Sigma_{d_1}}}\backslash S^{d_1-1}*\cdots*{_{\Sigma_{d_i}}}\backslash S^{d_i-1}.
\]
By our assumption at least one of the $d_j$s is greater than one, and so the right hand side is a contractible space. Moreover, we claim that it is contractible as a $N/Y$-equivariant space. In other words, we claim that for every subgroup $H\subset N/Y$, the fixed point space 
\begin{equation}\label{equation: contractible fix}
({_{\Sigma_{d_1}}}\backslash S^{d_1-1}*\cdots*{_{\Sigma_{d_i}}}\backslash S^{d_i-1})^H
\end{equation}
is contractible. To see this, observe that $N/Y$ is a Young subgroup of $\Sigma_i$ (recall that $i$ is the number of factors $\Sigma_{d_j}$ of $Y$). $N/Y$ acts on the space  
$$S^{d-1}\cong {_{\Sigma_{d_1}}}\backslash S^{d_1-1}*\cdots*{_{\Sigma_{d_i}}}\backslash S^{d_i-1}$$ by permuting join factors that happen to be homeomorphic. It follows that the fixed point space~\eqref{equation: contractible fix} is a join of factors of the form ${_{\Sigma_{d_j}}}\backslash S^{d_j-1}$. Once again, at least one of the $d_j$ is greater than $1$, so at least one of these factors is contractible, and therefore the whole space is contractible.

It follows that the orbit space
$$_{H}\backslash \left({_Y}\backslash S^{d-1}\right)$$
is contractible for every $H\subset N/Y$. Finally, it follows that for every $Y\subset W \subset N$, the orbit space $_{W}\backslash S^{d-1}$
is contractible.
\end{proof}

\begin{lemma}\label{lemma: contractible}
The space ${_{\Sigma_d}}\backslash \! \left(S^{d-1} \Smash |\Pi_d|^\diamond \right)$ is contractible for $d>1$.
\end{lemma}
\begin{proof}
The space ${_{\Sigma_d}}\backslash \! \left(S^{d-1} \Smash |\Pi_d|^\diamond \right)$ is a pointed homotopy colimit of spaces of the form $_G\backslash S^{d-1}$, where $G$ is an isotropy group of $ |\Pi_d|^\diamond $. We claim that the space $_G\backslash S^{d-1}$ is contractible for every $G$ that occurs. From the claim it follows that the pointed homotopy colimit is contractible.

It remains to prove the claim. Let $G$ be an isotropy group of $|\Pi_d|^\diamond$. Then either $G=\Sigma_d$ or $G$ is the stabilizer group of a chain of proper non-trivial partitions of $\mathbf d$. In the first case, $$_G\backslash S^{d-1}={_{\Sigma_d}}\backslash S^{d-1}\simeq *.$$ In the second case, suppose that $G$ is the stablizer of the chain of partitions $\lambda_0<\cdots <\lambda_l$. By our convention, $\lambda_l$ is the finest partition in the chains. Suppose $\lambda_l$ has $i$ equivalence classes, of sizes $d_1, \ldots, d_i$. Let $Y\cong \Sigma_{d_1}\times\cdots\times \Sigma_{d_i}$ be the group that leaves invariant the equivalence classes of $\lambda_l$. Then $G$ contains $Y$ and is contained in the normalizer of $Y$, and $_G \backslash S^{d-1}$ is contractible by Lemma~\ref{lemma: preliminary contractible}.
\end{proof}

The following result concerns rational homology of these spaces. 
\begin{lemma}\label{lemma: rational}
The space ${_{\Sigma_d}}\backslash \! \left(S^{ld-1} \Smash |\Pi_d|^\diamond\right)$ has the rational homology of a point, unless $d=1$ or $d=2$ and $l$ is even.
\end{lemma}
\begin{proof}


Let us form the pointed homotopy orbit space
\[
_{h\Sigma_d} \! \left(S^{ld-1} \Smash |\Pi_d|^\diamond \right):=E\Sigma_{d_+}\wedge _{\Sigma_d}\left(S^{ld-1} \Smash |\Pi_d|^\diamond \right).
\]
There is a natural map
\[
_{h\Sigma_d} \! \left(S^{ld-1} \Smash |\Pi_d|^\diamond \right) \longrightarrow {_{\Sigma_d}}\backslash \! \left(S^{ld-1} \Smash |\Pi_d|^\diamond \right)
\]
which is well-known to induce an isomorphism on rational homology. Therefore it is enough to prove the lemma with the homotopy orbit space replacing the strict orbit space. This was done in~\cite{A-M, arone}. For the sake of completeness we will sketch a proof.

The space $_{h\Sigma_d} \! \left(S^{ld-1} \Smash |\Pi_d|^\diamond \right)$ is a pointed homotopy colimit of spaces of the form $_{hG} S^{ld-1}$, where $G$ is an isotropy group of $ |\Pi_d|^\diamond $. By an easy Serre spectral sequence argument, such a space is rationally trivial if $l$ is odd and $G$ contains a transposition. It is easy to see that every isotropy group of $|\Pi_d|^\diamond$ contains a transposition, so we are done in the case when $l$ is odd.

If $l$ is even then the homology of $_{h\Sigma_d} \! \left(S^{ld-1} \Smash |\Pi_d|^\diamond \right)$ is isomorphic, with dimension shift of $ld-1$, to the homology of $_{h\Sigma_d} |\Pi_d|^\diamond $, by Thom isomorphism. The space $_{h\Sigma_d}  |\Pi_d|^\diamond $ is rationally equivalent to $_{\Sigma_d} \!\backslash  |\Pi_d|^\diamond $, and the latter space is contractible by Kozlov's theorem~\cite{kozlov}.
\end{proof}
\begin{remark}
Let us consider the case $d=2$. In this case $|\Pi_2|=\emptyset$, so $${_{\Sigma_2}}  (S^{2l-1}\Smash \Pi_2|^\diamond)\cong {_{\Sigma_2}} S^{2l-1} \cong \Sigma^l \reals P^{l-1}.$$ This space is rationally contractible if $l$ is odd.
Note that if $l$ is even, the reduced homology of $\Sigma^l \reals P^{l-1}$ has a copy of $\integers$ in dimension $2l-1$, and is $2$-torsion otherwise.
\end{remark}
\begin{corollary}\label{corollary: rational}
The torsion free part of $\widetilde\HH_*(_{\Sigma_{n_1}\times\cdots\times\Sigma_{n_k}}\!\!\backslash\Pi_n)$ is concentrated in degree $n-3$. If $2|\gcd(n_1, \ldots, n_k)$ and $\frac{n}{2}$ is odd, then the torsion-free part is isomorphic to 
\[
\integers^{B(n_1, \ldots, n_k)}\oplus \integers^{B(\frac{n_1}{2}, \ldots, \frac{n_k}{2})}.
\]
In all other cases, the torsion-free part is $\integers^{B(n_1, \ldots, n_k)}$.
\end{corollary}
\begin{proof}
This follows easily from Proposition~\ref{proposition: orbits again} and Lemma~\ref{lemma: rational}.
\end{proof}
We also have the following companion to the rational calculation.
\begin{lemma}\label{lemma: torsion}
The groups $\widetilde\HH_*(_{\Sigma_{n_1}\times\cdots\times\Sigma_{n_k}}\!\!\backslash\Pi_n)$ have $p$-primary torsion only for primes that satisfy $p\le \gcd(n_1, \ldots, n_k)$.
\end{lemma}
\begin{proof}
The space $_{\Sigma_{n_1}\times\cdots\times\Sigma_{n_k}}\!\!\backslash\Pi_n$ is a wedge sum of spaces of the form $_{\Sigma_d}\backslash (S^{ld-1}\wedge |\Pi_d|^\diamond)$, where $d$ divides $\gcd(n_1, \ldots, n_k)$. So it is enough to show that this space has no $p$-primary torsion for $p> d$. The cases $d=1, 2$ can easily be checked ``by hand''. Let us suppose that $d>2$ and $p>d$. To prove that the homology of $_{\Sigma_d}\backslash (S^{ld-1}\wedge |\Pi_d|^\diamond)$ has no $p$-primary torsion it is enough to prove that its reduced homology groups with $\integers/p$ coefficients are zero. Since $p>d$, the map
$$_{h\Sigma_d}\backslash (S^{ld-1}\wedge |\Pi_d|^\diamond) \to_{\Sigma_d}\backslash (S^{ld-1}\wedge |\Pi_d|^\diamond)$$
induces an isomorphism on homology with $\integers/p$ coefficients. So it is enough to prove that the homotopy orbits space  $_{h\Sigma_d}\backslash (S^{ld-1}\wedge |\Pi_d|^\diamond)$ has no reduced homology with $\integers/p$ coefficients. In fact, it follows from~\cite{ad} that this space is contractible unless $d$ is a prime power, and is $p$-local if $d=p^k$. But for the reader's convenience we will give a relatively elementary self-contained proof of the result that we need. We saw in the proof of Lemma~\ref{lemma: rational} that the reduced homology of this space is all torsion, so it is enough to show that it has no $p$-primary torsion. To see this consider the following homomorphisms
\begin{equation}\label{equation: transfer}
\widetilde\HH_*\left(_{h\Sigma_d}\backslash \! \left(S^{ld-1} \Smash |\Pi_d|^\diamond\right)\right) \to \widetilde\HH_*\left(_{h\Sigma_{d-1}}\backslash \! \left(S^{ld-1} \Smash |\Pi_d|^\diamond\right)\right)\to \widetilde\HH_*\left(_{h\Sigma_d}\backslash \! \left(S^{ld-1} \Smash |\Pi_d|^\diamond\right)\right).
\end{equation}
Here the first homomorphism is the transfer in homology associated with the inclusion $\Sigma_{d-1}\hookrightarrow \Sigma_d$, and the second homomorphism is induced by the quotient map. It is well-known that the complosed homomorphism is multiplication by $d=|\Sigma_d/\Sigma_{d-1}|$. By the main theorem, $|\Pi_d|$ is $\Sigma_{d-1}$-equivariantly equivalent to $\Sigma_{{d-1}_+} \wedge S^{d-3}$. Therefore \[_{h\Sigma_{d-1}}\backslash \! \left(S^{ld-1} \Smash |\Pi_d|^\diamond\right)\simeq S^{ld+d-3}.\] The reduced homology of this space is $\integers$ in dimension $ld+d-3$ and zero in all other dimensions. It follows that the composed homomorphism~\eqref{equation: transfer} is zero on torsion elements. So multiplication by $d$ is zero on the torsion elements of the group $\widetilde \HH_*\left({_{h\Sigma_d}}\backslash \! \left(S^{ld-1} \Smash |\Pi_d|^\diamond\right)\right)$. It follows that it is a $d$-torsion group and in particular it has no $p$-primary torsion for $p>d$.
\end{proof}

Our next task is to analyze the space $_{\Sigma_p}\backslash \! \left(S^{lp-1} * |\Pi_p|\right)$ in the case when $p$ is a prime. In the case $p=2$ we already saw that 
$$_{\Sigma_2}\backslash \! \left(S^{2l-1} * |\Pi_2|\right)\cong \Sigma^l\reals P^{l-1}.$$

%
%
Recall that there is a $\Sigma_p$-equivariant homeomorphism $S^{lp-1}\cong S^{l-1} \Smash (S^{p-1})^{l}$, where $\Sigma_p$ acts on $S^{p-1}$ via the reduced standard representation, and acts trivially on $S^{l-1}$. The fixed-points space ${S^{lp-1}}^{\Sigma_p}$ is homeomorphic to $S^{l-1}$. Let $S^{l-1}\hookrightarrow S^{lp-1}$ be the inclusion of fixed points.
The key to analyzing $_{\Sigma_p}\backslash \! \left(S^{lp-1} \Smash |\Pi_p|^\diamond \right)$ for a general prime $p$ is the following proposition.
\begin{proposition}\label{proposition: primal pushout}
The following is a homotopy pushout square
\begin{equation}\label{equation: primal pushout}
\begin{array}{ccc}
_{h\Sigma_p}\backslash \! \left(S^{l-1} \Smash |\Pi_p|^\diamond\right) & \to & _{\Sigma_p}\backslash \! \left(S^{l-1} \Smash |\Pi_p|^\diamond\right)\\
\downarrow & & \downarrow \\
_{h\Sigma_p}\backslash \! \left(S^{lp-1} \Smash |\Pi_p|^\diamond\right) & \to & _{\Sigma_p}\backslash \! \left(S^{lp-1} \Smash |\Pi_p|^\diamond\right)
\end{array}.
\end{equation}
\end{proposition}
\begin{proof}
The homotopy cofiber of the map $S^{l-1}\to S^{lp-1}$ is equivalent to $S^{l}\Smash (S^{l(p-1)-1})_+$. Therefore, taking homotopy cofibers of the vertical maps we obtain the map
\[
_{h\Sigma_p}\backslash \! \left(S^{l}\Smash(S^{l(p-1)-1})_+ \Smash |\Pi_p|^\diamond\right)  \to  {_{\Sigma_p}}\backslash \! \left(S^{l}\Smash(S^{l(p-1)-1})_+ \Smash |\Pi_p|^\diamond\right).
\]
We want to prove that this map is a homotopy equivalence. For this, it is enough to prove that the following map is an equivalence
\[
_{h\Sigma_p}\backslash  (S^{l(p-1)-1})_+ \Smash |\Pi_p|^\diamond  \to  {_{\Sigma_p}}\backslash  (S^{l(p-1)-1})_+ \Smash |\Pi_p|^\diamond.
\]
The $\Sigma_p$-space $S^{l(p-1)-1}$ can be written as a homotopy colimit of sets of the form $\Sigma_p/G$, where $G$ is an isotropy group of $S^{l(p-1)-1}$. Therefore, it is enough to prove that for every group $G$ that occurs as an isotropy of $S^{l(p-1)-1}$, the map $_{h\Sigma_p}\backslash  (\Sigma_p/G)_+ \wedge |\Pi_p|^\diamond \to  {_{\Sigma_p}}\backslash  (\Sigma_p/G)_+ \wedge  |\Pi_p|^\diamond$ is an equivalence. We claim that every isotropy group of $S^{l(p-1)-1}$ is is contained in a group of the form $\Sigma_{p_1}\times \cdots\times \Sigma_{p_l}$, where $l>1$ and $p_1+\cdots+p_l=p$. Indeed, for $l=1$, $S^{p-2}$ is the boundary of the $p-1$-simplex, and it is easy to see that the isotropy groups have the form $\Sigma_{p_1}\times \cdots\times \Sigma_{p_l}$, where $l>1$ and $p_1+\cdots+p_l=p$. For $l>1$ one may argue by induction on $l$. Since $p$ is a prime, $\gcd(p_1, \ldots, p_l)=1$. It follows, using our main theorem, that the action of $\Sigma_{p_1}\times \cdots\times \Sigma_{p_l}$ on $\Pi_p$ is essentially pointed-free, in the sense that for every subgroup $H\subset \Sigma_{p_1}\times \cdots\times \Sigma_{p_l}$, the fixed points space $\Pi_p^H$ is contractible. It follows that the action of $G$ on $\Pi_p$ is essentially pointed-free. It follows that the map $_{h\Sigma_p}\backslash  (\Sigma_p/G)_+ \wedge |\Pi_p|^\diamond \to  {_{\Sigma_p}}\backslash  (\Sigma_p/G)_+ \wedge |\Pi_p|^\diamond$, which is the same as the map $_{hG}\backslash  |\Pi_p|^\diamond \to  {_G}\backslash  |\Pi_p|^\diamond$, is an equivalence.
\end{proof}
\begin{corollary}\label{corollary: odd cofibration}
If $p$ is an odd prime then there is a homotopy cofibration sequence
\[
_{h\Sigma_p}\backslash \! \left(S^{l-1} \Smash |\Pi_p|^\diamond\right) \to {_{h\Sigma_p}}\backslash \! \left(S^{lp-1} \Smash |\Pi_p|^\diamond\right)  \to  {_{\Sigma_p}}\backslash \! \left(S^{lp-1} \Smash |\Pi_p|^\diamond\right).
\]
\end{corollary}
\begin{proof}
Consider the upper right corner of~\eqref{equation: primal pushout}. Since $p>2$, ${_{\Sigma_p}}\backslash \! \left(S^{l-1}\Smash |\Pi_p|^\diamond\right)$, which is homeomorphic to $S^{l-1}\Smash \left({_{\Sigma_p}}\backslash |\Pi_p|^\diamond \right)$, is contractible by Kozlov's theorem~\cite{kozlov}. The corollary now follows from Proposition~\ref{proposition: primal pushout}.
\end{proof}
\begin{corollary}\label{corollary: p-local}
If $p$ is an odd prime then the space ${_{\Sigma_p}}\backslash \! \left(S^{lp-1} \Smash |\Pi_p|^\diamond\right)$ is $p$-local.
\end{corollary}
\begin{proof}
By Corollary~\ref{corollary: odd cofibration} it is enough to prove that the spaces $_{h\Sigma_p}\backslash \! \left(S^{l-1} \Smash |\Pi_p|^\diamond\right)$ and ${_{h\Sigma_p}}\backslash \! \left(S^{lp-1} \Smash |\Pi_p|^\diamond\right)$ are each $p$-local. Since $l\ge 1$ and $p>2$ both of these spaces are simply connected. Therefore, it is enough to show that the integral homology of these spaces with $\integers$ coefficients is $p$-torsion. This can be done using a transfer argument similar to the one in the proof of Lemma~\ref{lemma: torsion}.
\end{proof}

\begin{proposition}\label{proposition: prime}
Let $p$ be an odd prime. If $l$ is odd, then the orbits space ${_{\Sigma_p}}\backslash \! \left(S^{lp-1} \Smash |\Pi_p|^\diamond\right)$ is equivalent to the $p$-localization of  ${_{\Sigma_p}}\!\backslash S^{lp-1}$. If $l$ is even then $_{\Sigma_p}\backslash \! \left(S^{lp-1} \Smash |\Pi_p|^\diamond\right)$ is equivalent to the $p$-localization of the homotopy cofiber of the quotient map  $S^{lp-1}\to {_{\Sigma_p}}\!\backslash S^{lp-1}$.
\end{proposition}
\begin{proof}
Once again we use the homotopy cofibration sequence
\[
{_{\Sigma_p}}\backslash \! \left(S^{lp-1} \wedge |\Pi_p|_+\right)\to {_{\Sigma_p}}\!\backslash S^{lp-1}\to {_{\Sigma_p}}\backslash \! \left(S^{lp-1} \Smash |\Pi_p|^\diamond\right).
\]
Our task reduces to showing that if $l$ is odd then the $p$-localization of the space ${_{\Sigma_p}}\backslash \! \left(S^{lp-1} \wedge |\Pi_p|_+\right)$ is trivial, and if $l$ is even then for any choice of point in $ |\Pi_p|$, the induced map $S^{lp-1}\to {_{\Sigma_p}}\backslash \! \left(S^{lp-1}\wedge |\Pi_p|_+\right)$ is a $p$-local equivalence.

Consider the pointed $\Sigma_p$-space $S^{lp-1} \wedge |\Pi_p|_+$. It is easy to check that every isotropy group of the space (aside from basepoint) is a non-transitive subgroup of $\Sigma_p$. It follows that $p$ does not divide the order of any of the isotropy groups. It follows that for every isotropy group $G$, the map $BG\to *$ induces an isomorphism in homology with $\integers/p$-coefficients. It follows that the natural map ${_{h\Sigma_p}}\backslash \! \left(S^{lp-1} \wedge |\Pi_p|_+\right)\to {_{\Sigma_p}}\backslash \! \left(S^{lp-1} \wedge |\Pi_p|_+\right)$ is a mod $p$-homology isomorphism and therefore a $p$-local isomorphism. The space ${_{h\Sigma_p}}\backslash \! \left(S^{lp-1} \wedge |\Pi_p|_+\right)$ is a pointed homotopy colimit of spaces of the form ${_{hG}}\backslash S^{lp-1}$, where $G$ is an isotropy group of $|\Pi_p|$. The relevant properties of $G$ are that (1) $p$ does not divide the order of $G$ and (2) $G$ contains a transposition. Suppose that $l$ is odd. Then $\widetilde\HH_*\left({_{hG}}\backslash S^{lp-1};\integers/p\right)\cong \{0\}$, and so ${_{hG}}\backslash S^{lp-1}$ is $p$-locally trivial. It follows that ${_{h\Sigma_p}}\backslash \! \left(S^{lp-1} \wedge |\Pi_p|_+\right)$ is $p$-locally trivial. Now suppose $l$ is even. The space ${_{h\Sigma_p}}\backslash \! \left(S^{lp-1} \wedge |\Pi_p|_+\right)$ is a Thom space of a bundle over $E\Sigma_p\times_{\Sigma_p} |\Pi_p|$. Since $l$ is even, the bundle is orientable, and therefore Thom isomorphism holds. Any map $*\to E\Sigma_p\times_{\Sigma_p} |\Pi_p|$ is a mod $p$ homology isomorphism. Therefore by Thom isomorphism it induces a mod $p$ homology isomorphism of Thom spaces $S^{lp-1}\to  {_{h\Sigma_p}}\backslash \! \left(S^{lp-1} \wedge |\Pi_p|_+\right)$. 
\end{proof}
\begin{corollary}\label{corollary: also l=2}
Let $p$ be a prime. For $l> 2$ the space ${_{\Sigma_p}}\backslash \! \left(S^{lp-1} \Smash |\Pi_p|^\diamond\right)$ is not equivalent to a wedge of spheres. For $l=2$ it is equivalent to $S^{2l-1}$ if $p=2$ and is contractible if $p$ is odd.
\end{corollary}
\begin{remark}
Let us recall that for $l=1$ the space is contractible by Lemma~\ref{lemma: contractible}.
\end{remark}
\begin{proof}
Consider first the case $p=2$. In this case we identified ${_{\Sigma_p}}\backslash \! \left(S^{lp-1} \Smash |\Pi_p|^\diamond\right)$ with $\Sigma^l\reals P^{l-1}$. For $l=2$ this space is equivalent to $S^{3}$. For $l>2$ it is not equivalent to a wedge of spheres since its homology has $2$-torsion.

Suppose now that $p>2$. By Proposition~\ref{proposition: prime} we need to analyze the space ${_{\Sigma_p}}\!\backslash S^{lp-1}$ or the homotopy cofiber of the quotient map  $S^{lp-1}\to {_{\Sigma_p}}\!\backslash S^{lp-1}$. Note that the space ${_{\Sigma_p}}\!\backslash S^{lp-1}$ is a desuspension of $(S^{l})^{\Smash p}_{\Sigma_p}$. The homology with $\integers/p$ coefficients of these spaces is known by work of Nakaoka~\cite{nakaoka}. Applying Nakaoka's result to our situation, we find that the The homology with $\integers/p$ coefficients of ${_{\Sigma_p}}\backslash \! \left(S^{lp-1} \Smash |\Pi_p|^\diamond\right)$ is generated by symbols $\{i\}$, where $i$ is an integer satisfying the following constraints
\begin{itemize}
\item $i>1$
\item $i\equiv 0\mbox{ or } 1 \,(\mbox{mod } 2(p-1))$
\item $i<(p-1)l$
\end{itemize}
The degree of $\{i\}$ is $l+i-1$. It is easy to see that for $l\le 2$ there are no integers $i$ satisfying these constraints. It follows that for $l\le 2$ the mod $p$ homology of ${_{\Sigma_p}}\backslash \! \left(S^{lp-1} \Smash |\Pi_p|^\diamond\right)$ is trivial. Since the space is $p$-local, it follows that it is contractible. 

On the other hand, for $l>2$ it is easy to see that there is (more than one) value of $i$ satisfying the constraints. It follows that the space is not contractible. Since it is $p$-local, it is not equivalent to a wedge of spheres.
\end{proof}
\begin{corollary}\label{corollary: wedge of spheres}
The quotient space $_{\Sigma_{n_1}\times\cdots\times\Sigma_{n_k}}\!\!\backslash\Pi_n$ is homotopy equivalent to a wedge of spheres if and only if one of the following holds:
\begin{itemize}
\item $\gcd(n_1, \ldots, n_k)=1$.
\item $p$ is a prime, $n=2p$ or $3p$, and $\gcd(n_1, \ldots, n_k)=p$.
\end{itemize}
\end{corollary}
\begin{proof}
The case $\gcd(n_1, \ldots, n_k)=1$ was dealt with in Corollary~\ref{corollary: gcd one}. In the second case, Propositioin~\ref{proposition: orbits again} tells us that $_{\Sigma_{n_1}\times\cdots\times\Sigma_{n_k}}\!\!\backslash\Pi_n$ is equivalent to a wedge sum of spaces of the form $S^{n-3}$ and ${_{\Sigma_p}}\backslash \! \left(S^{n-p-1} \Smash |\Pi_p|^\diamond\right)$. Since $n=2p$ or $3p$, $n-p=lp$ where $l=1, 2$. By Lemma~\ref{lemma: contractible} and Corollary~\ref{corollary: also l=2} all these spaces are either equivalent to a sphere or are contractible.
 
In all other cases, let $p$ be the smallest prime that divides $\gcd(n_1, \dots, n_k)$. Then $p$ divides $n$ and $\frac{n}{p}>3$. It follows that $_{\Sigma_{n_1}\times\cdots\times\Sigma_{n_k}}\!\!\backslash\Pi_n$ has a wedge summand equivalent to 
\[
\bigvee_{B(\frac{n_1}{p}, \ldots, \frac{n_k}{p})} {_{\Sigma_p}}\backslash (S^{lp-1}\Smash|\Pi_p|^\diamond)
\]
where $l>2$. By Corollary~\ref{corollary: also l=2} this space is not equivalent to a wedge of spheres. 
\end{proof}
\begin{example}
To illustrate our results, let us analyze the homotopy type and the homology groups of $_{\Sigma_4\times \Sigma_4}\backslash |\Pi_8|^\diamond$. By Proposition~\ref{proposition: orbits again} there is a homotopy equivalence
\[
_{\Sigma_4\times \Sigma_4}\backslash |\Pi_8|^\diamond\simeq \bigvee_{B(4, 4)} S^5 \vee \bigvee_{B(2, 2)} {_{\Sigma_2}}\!\backslash S^5 \vee \bigvee_{B(1,1)} {_{\Sigma_4}}\!\backslash (S^3 \Smash |\Pi_4|^\diamond).
\]
The last factor is contractible by Lemma~\ref{lemma: contractible}. A quick calculation shows that $|B(4, 4)|=8$ and $|B(2,2)|=1$. We already observed that  ${_{\Sigma_2}}\backslash S^5\cong \Sigma^3\reals P^2$. We conclude that there is a homotopy equivalence
\[
_{\Sigma_4\times \Sigma_4}\backslash |\Pi_8|^\diamond\simeq \Sigma^3\reals P^2\vee \bigvee_{8} S^5 .
\]
It follows that the reduced homology of $_{\Sigma_4\times \Sigma_4}\backslash |\Pi_8|^\diamond$ is isomorphic to $\integers/2$ in dimension $4$, to $\integers^8$ in dimension $5$, and is zero otherwise. This is consistent with computer calculations done by Donau~\cite{donau1}. 
\end{example}


\end{document}